\title{Submodular setfunctions, matroids\\ and graph limits}
\author{László Lovász}

\documentclass[12pt,fleqn]{article}
\usepackage{amsmath,amssymb,graphicx,bbm,bm,delarray,pict2e,ntheorem}
\usepackage[hypertex]{hyperref}
\usepackage[latin2]{inputenc}
\usepackage[normalem]{ulem}

\newtheorem{theorem}{Theorem}[section]
\newtheorem{prop}[theorem]{Proposition}
\newtheorem{lemma}[theorem]{Lemma}
\newtheorem{claim}{Claim}
\newtheorem{corollary}[theorem]{Corollary}

\theorembodyfont{\rmfamily}
\newtheorem{remark}[theorem]{Remark}
\newtheorem{example}[theorem]{Example}
\newtheorem{prob}[theorem]{Problem}
\newtheorem{exc}[theorem]{Exercise}

\newenvironment{proof}{\medskip\noindent{\bf Proof. }}{\hfill$\square$\medskip}
\newenvironment{proof*}[1]{\medskip\noindent{\bf Proof of #1.}}{\hfill$\square$\medskip}

\begin{document}

\addtolength{\baselineskip}{3pt}

\def\AA{\mathcal{A}}\def\BB{\mathcal{B}}\def\CC{\mathcal{C}}
\def\DD{\mathcal{D}}\def\EE{\mathcal{E}}\def\FF{\mathcal{F}}
\def\GG{\mathcal{G}}\def\HH{\mathcal{H}}\def\II{\mathcal{I}}
\def\JJ{\mathcal{J}}\def\KK{\mathcal{K}}\def\LL{\mathcal{L}}
\def\MM{\mathcal{M}}\def\NN{\mathcal{N}}\def\OO{\mathcal{O}}
\def\PP{\mathcal{P}}\def\QQ{\mathcal{Q}}\def\RR{\mathcal{R}}
\def\SS{\mathcal{S}}\def\TT{\mathcal{T}}\def\UU{\mathcal{U}}
\def\VV{\mathcal{V}}\def\WW{\mathcal{W}}\def\XX{\mathcal{X}}
\def\YY{\mathcal{Y}}\def\ZZ{\mathcal{Z}}

\def\Ab{\mathbf{A}}\def\Bb{\mathbf{B}}\def\Cb{\mathbf{C}}
\def\Db{\mathbf{D}}\def\Eb{\mathbf{E}}\def\Fb{\mathbf{F}}
\def\Gb{\mathbf{G}}\def\Hb{\mathbf{H}}\def\Ib{\mathbf{I}}
\def\Jb{\mathbf{J}}\def\Kb{\mathbf{K}}\def\Lb{\mathbf{L}}
\def\Mb{\mathbf{M}}\def\Nb{\mathbf{N}}\def\Ob{\mathbf{O}}
\def\Pb{\mathbf{P}}\def\Qb{\mathbf{Q}}\def\Rb{\mathbf{R}}
\def\Sb{\mathbf{S}}\def\Tb{\mathbf{T}}\def\Ub{\mathbf{U}}
\def\Vb{\mathbf{V}}\def\Wb{\mathbf{W}}\def\Xb{\mathbf{X}}
\def\Yb{\mathbf{Y}}\def\Zb{\mathbf{Z}}

\def\ab{\mathbf{a}}\def\bb{\mathbf{b}}\def\cb{\mathbf{c}}
\def\db{\mathbf{d}}\def\eb{\mathbf{e}}\def\fb{\mathbf{f}}
\def\gb{\mathbf{g}}\def\hb{\mathbf{h}}\def\ib{\mathbf{i}}
\def\jb{\mathbf{j}}\def\kb{\mathbf{k}}\def\lb{\mathbf{l}}
\def\mb{\mathbf{m}}\def\nb{\mathbf{n}}\def\ob{\mathbf{o}}
\def\pb{\mathbf{p}}\def\qb{\mathbf{q}}\def\rb{\mathbf{r}}
\def\sb{\mathbf{s}}\def\tb{\mathbf{t}}\def\ub{\mathbf{u}}
\def\vb{\mathbf{v}}\def\wb{\mathbf{w}}\def\xb{\mathbf{x}}
\def\yb{\mathbf{y}}\def\zb{\mathbf{z}}

\def\Abb{\mathbb{A}}\def\Bbb{\mathbb{B}}\def\Cbb{\mathbb{C}}
\def\Dbb{\mathbb{D}}\def\Ebb{\mathbb{E}}\def\Fbb{\mathbb{F}}
\def\Gbb{\mathbb{G}}\def\Hbb{\mathbb{H}}\def\Ibb{\mathbb{I}}
\def\Jbb{\mathbb{J}}\def\Kbb{\mathbb{K}}\def\Lbb{\mathbb{L}}
\def\Mbb{\mathbb{M}}\def\Nbb{\mathbb{N}}\def\Obb{\mathbb{O}}
\def\Pbb{\mathbb{P}}\def\Qbb{\mathbb{Q}}\def\Rbb{\mathbb{R}}
\def\Sbb{\mathbb{S}}\def\Tbb{\mathbb{T}}\def\Ubb{\mathbb{U}}
\def\Vbb{\mathbb{V}}\def\Wbb{\mathbb{W}}\def\Xbb{\mathbb{X}}
\def\Ybb{\mathbb{Y}}\def\Zbb{\mathbb{Z}}

\def\Af{\mathfrak{A}}\def\Bf{\mathfrak{B}}\def\Cf{\mathfrak{C}}
\def\Df{\mathfrak{D}}\def\Ef{\mathfrak{E}}\def\Ff{\mathfrak{F}}
\def\Gf{\mathfrak{G}}\def\Hf{\mathfrak{H}}\def\If{\mathfrak{I}}
\def\Jf{\mathfrak{J}}\def\Kf{\mathfrak{K}}\def\Lf{\mathfrak{L}}
\def\Mf{\mathfrak{M}}\def\Nf{\mathfrak{N}}\def\Of{\mathfrak{O}}
\def\Pf{\mathfrak{P}}\def\Qf{\mathfrak{Q}}\def\Rf{\mathfrak{R}}
\def\Sf{\mathfrak{S}}\def\Tf{\mathfrak{T}}\def\Uf{\mathfrak{U}}
\def\Vf{\mathfrak{V}}\def\Wf{\mathfrak{W}}\def\Xf{\mathfrak{X}}
\def\Yf{\mathfrak{Y}}\def\Zf{\mathfrak{Z}}

\def\alphab{{\boldsymbol\alpha}}\def\betab{{\boldsymbol\beta}}
\def\gammab{{\boldsymbol\gamma}}\def\deltab{{\boldsymbol\delta}}
\def\etab{{\boldsymbol\eta}}\def\zetab{{\boldsymbol\zeta}}
\def\kappab{{\boldsymbol\kappa}}
\def\lambdab{{\boldsymbol\lambda}}\def\mub{{\boldsymbol\mu}}
\def\nub{{\boldsymbol\nu}}\def\pib{{\boldsymbol\pi}}
\def\rhob{{\boldsymbol\rho}}\def\sigmab{{\boldsymbol\sigma}}
\def\taub{{\boldsymbol\tau}}\def\epsb{{\boldsymbol\varepsilon}}
\def\epsilonb{{\boldsymbol\epsilon}} \def\inb{{\boldsymbol\in}}
\def\phib{{\boldsymbol\varphi}}\def\psib{{\boldsymbol\psi}}
\def\xib{{\boldsymbol\xi}}\def\omegab{{\boldsymbol\omega}}
\def\intl{\int\limits}
\def\sqprod{\mathbin{\square}}

\def\ybb{\mathbbm{y}}
\def\one{{\mathbbm1}}
\def\two{{\mathbbm2}}
\def\R{\Rbb}\def\Q{\Qbb}\def\Z{\Zbb}\def\N{\Nbb}\def\C{\Cbb}
\def\wh{\widehat}
\def\wt{\widetilde}
\def\var{\omega}
\def\eps{\varepsilon}
\def\sgn{{\rm sign}}
\def\dd{{\sf d}}
\def\Rv{\overleftarrow}
\def\Pr{{\sf P}}
\def\E{{\sf E}}
\def\T{^{\sf T}}
\def\proofend{\hfill$\square$}
\def\id{\hbox{\rm id}}
\def\conv{\hbox{\rm conc}}
\def\lin{\hbox{\rm lib}}
\def\conv{\hbox{\rm conc}}
\def\Dim{\hbox{\rm Dim}}
\def\const{\hbox{\rm const}}
\def\vol{\text{\rm vol}}
\def\diam{\text{\rm diam}}
\def\corank{\hbox{\rm cork}}
\def\cork{\hbox{\rm cork}}
\def\OR{\mathcal{OR}}
\def\GOR{\mathcal{GOB}}
\def\NOR{\mathcal{NOR}}
\def\LGOR{\mathcal{ALGOR}}
\def\cro{\text{\rm CR}}
\def\supp{\text{\rm sup}}
\def\grad{\text{\rm grad}}
\def\rk{\overline{\rho}}
\def\srk{\hbox{\rm src}}
\def\diag{{\rm diag}}
\def\pw{{\sf w}_\text{\rm prod}}
\def\tw{{\sf w}_\text{\rm tree}}
\def\aw{{\sf w}_\text{\rm alb}}
\def\bw{{\sf bow}}
\def\ld{{\sf d}_{\rm loc}}
\def\hd{{\sf d}_{\rm har}}
\def\tv{\text{\rm tv}}
\def\Tr{\hbox{\rm Tar}}
\def\tr{\hbox{\rm tr}}
\def\Prob{\hbox{\rm Pr}}
\def\bl{\text{{\rm bl}}}
\def\abl{\hbox{{\rm abl}}}
\def\Id{\hbox{\rm Id}}
\def\aff{\text{\rm ra}}
\def\MC{\CC_{\max}}
\def\Inf{\text{\sf Inf}}
\def\Str{\text{\sf Str}}
\def\Rig{\text{\sf Rig}}
\def\Mat{\text{\sf Mat}}
\def\comm{{\sf comm}}
\def\maxcut{{\sf maxcut}}
\def\disc{\text{\sf disc}}
\def\cond{\Phi}
\def\val{\text{\sf val}}
\def\dist{d_{\rm qu}}
\def\dhaus{d_{\rm haus}}
\def\dlp{d_{\rm LP}}
\def\dact{d_{\rm act}}
\def\Ker{{\rm Ker}}
\def\Rng{{\rm Rng}}
\def\gdim{{\rm gdim}}
\def\gap{\text{\rm gap}}
\def\intl{\int\limits}
\def\et{\qquad\text{and}\qquad}
\def\fin{\text{\sf fin}}
\def\Bd{{\sf Bd}}
\def\ba{{\sf ba}}
\def\ca{{\sf ca}}
\def\matp{{\sf mm}}
\def\blim{{\rm blim}}
\def\basp{{\sf bmm}}
\def\sep{{\sf sep}}
\def\Hom{{\rm Hom}}
\def\bp{{\sf bp}}
\def\fg{\varphi}
\def\fgx{{\varphi}^\sqcap}
\def\lc{\bullet}
\def\li{{\rm li}}
\def\ui{{\rm ui}}
\def\ls{{\rm ls}}
\def\us{{\rm us}}
\long\def\ignore#1{}

\def\QR{{R^{cc}}}

\def\url{}
\maketitle

{\small \tableofcontents}

\newpage

\begin{abstract}
Submodular setfunctions play an important role in potential theory, and a
perhaps even more important role in combinatorial optimization. The analytic
line of research goes back to the work of Choquet; the combinatorial, to the
work of Rado and Edmonds. The two research lines have not had much interaction
though. Recently, with the development of graph limit theory, the question of a
limit theory for matroids has been considered by several people; such a theory
will, most likely, involve submodular setfunctions both on finite and infinite
sets.

The goal of this paper is to describe several connections between the analytic
and combinatorial theory, to show parallels between them, and to propose
problems arising by trying the generalize the rich theory of submodular
setfunctions on finite sets to the analytic setting. It is aimed more at
combinatorialists, and it spends more time on developing the analytic theory,
often referring to results in combinatorial optimization by name or sketch
only.
\end{abstract}

\section{Introduction}

Submodular setfunctions (also known as ``strongly subadditive'') play an
important role in potential theory, and perhaps an even more important role in
combinatorial optimization. The analytic line of research goes back to a paper
by Choquet \cite{Choq}; the combinatorial, to the work of Rado \cite{Rado} and
Edmonds \cite{Edm}. The two research lines have not had much interaction.

Recently, with the development of graph limit theory, the question of a limit
theory for matroids has been raised by several people. In fact, continuous
limit objects for some very special sequences of matroids were constructed
before. John von Neumann \cite{JvN1} introduced ``continuous geometries'' in
1936; these objects have a dimension function which is modular, with values in
the interval $[0,1]$. A limit object of partition lattices on increasing finite
sets (which can be considered as cycle matroids of complete graphs) was
constructed by Björner \cite{Bj87}. A related construction, closer to our
approach, was given by Haiman \cite{Hai}. This construction was extended to a
few other special sequences of matroids (Björner and Lovász \cite{BjLo}).

A rather natural question is to consider a sequence of finite graphs $G_n$ with
bounded degree, which converge to a graphing $\Gb$ (see \cite{HomBook} for the
definition and basic properties of graphings), and determine in what sense
their cycle matroids converge to a submodular setfunction on the edge set of
$\Gb$. In \cite{Lov23b} one possible notion of the matroid of a graphing has
been defined and an embryonic theory of convergence of the matroids of the
graphs $G_n$ has been developed. Perhaps the limit theory of cycle matroids can
be extended from bounded degree graph sequences to denser graphs; but this
possibility has not been explored (see Problems \ref{PROB:PSEMOD} and
\ref{PROB:INGLETON}).

Another recent limit theory for certain matroids (linearly representable and
satisfying a sparseness condition) was initiated by Kardo\v{s} et
al.~\cite{KKLM}, based on ideas of first order convergence introduced by
Ne\@{s}et\v{r}il and Ossona de Mendez \cite{Nes-Ossona}. This, however, leads
to a different concept of the limit object from our approach, which is based on
submodular setfunctions on measurable spaces.

The goal of this paper is to describe some connections between the analytic and
combinatorial theory of submodular setfunctions, to show parallels between
them, and to propose problems arising when trying the generalize the rich
theory of submodular setfunctions on finite sets to the analytic setting. The
generalization is quite heavily measure theoretic. I hope that this can
facilitate a more comprehensive limit theory of matroids.

Let $\AA$ be a family of subsets of a (finite or infinite) set $I$, closed
under union and intersection of two elements (briefly, a lattice family). A
setfunction $\fg:~\AA\to\R$ is called {\it submodular}, if
\begin{equation}\label{EQ:SUBMOD-DEF0}
\fg(X\cup Y)+ \fg(X\cap Y) \le \fg(X)+\fg(Y)
\end{equation}
for all $X,Y\in\AA$. It is called {\it modular}, if this condition holds with
equality, and {\it supermodular}, if it holds with the inequality reversed. In
this paper we will assume that $\emptyset\in\AA$, and $\AA$ is closed under
complementation. We can shift $\fg$ by a constant without changing
\eqref{EQ:SUBMOD-DEF0}; in particular, we can almost always assume that
$\fg(\emptyset)=0$ (and we are going to do so in this introduction). In this
case, a modular setfunction is just a finitely additive signed measure, also
called a {\it charge}.

Submodular setfunctions have important applications in combinatorial
optimization: they serve in providing common generalizations of many
combinatorial results in graph connectivity, flow theory, matching theory, the
theory of rigidity of frameworks, and many other places (see the monographs of
Oxley \cite{Oxley}, Schrijver \cite{Schr} and Frank \cite{Frank11}).

Extending the notion of submodular setfunctions on finite sets to an infinite
setting was initiated by Murota \cite{Mur1,Mur2}, and has been studied by
several authors \cite{Bach,Csir}. One line of these studies replace functions
defined on $0$-$1$ vectors in a finite dimensional space (i.e., subsets of a
finite set) by functions defined on all lattice (grid) points or on all vectors
in this space.

A different approach is to consider submodularity of setfunctions defined on
the sets in an infinite set algebra or sigma-algebra $(J,\BB)$. In most cases,
this condition is used to generalize integration with respect to a measure to
integration with respect to increasing submodular setfunctions $\fg$. This can
be done by using the Choquet integral or ``Layer Cake Representation'', giving
rise to the functional
\[
\wh\fg(f) = \intl_0^\infty \fg\{f\ge t\}\,dt
\]
on bounded measurable functions $f:~J\to\R_+$. (It takes some work to show that
$\fg\{f\ge t\}$ is an integrable function of $t$ for all bounded submodular
setfunctions; see Lemmas \ref{LEM:BD-VAR} and \ref{LEM:SUBMOD-FINVAR}.) If the
setfunction is modular, then this functional is linear (and vice versa). For a
submodular setfunction it is convex; this was shown by Choquet \cite{Choq} and
\v{S}ipo\v{s} \cite{Sipos} (for the case of increasing setfunctions) and in
\cite{LL83} (for the finite case). We prove this for arbitrary bounded
submodular setfunctions in Section \ref{SEC:UNCROSS}; see also \cite{Denn} for
a detailed treatment.

As another illustration that there are interesting generalizations of several
basic uses of submodular setfunctions in the finite case, we extend the
definition of the independent set polytope and the basis polytope of a matroid
(Section \ref{SEC:MATROID}). It turns out that in the infinite case independent
sets of a matroid correspond to charges minorizing the given submodular
setfunction. We also develop analogues of the notions of flats (closed subsets)
in Section \ref{SEC:BJ-DIST} and strong maps of matroids in Section
\ref{SSEC:DIVERGE}.

In the finite case, several important algorithmic methods are tied to
submodularity. Convexity of $\wh\fg$ can be proved by a technique called
``uncrossing'', which is generalized in Section \ref{SEC:UNCROSS}. Many
properties of matroid bases are proved using the Greedy Algorithm, which has
its infinite analogue in one of the basic results of Choquet (see Corollary
\ref{COR:FIN-ADD-LAT}).

We conclude with suggesting a number of further open problems, along with
Appendices giving some background on charges and matroids, and with a table of
the most often used notation.

I take the point of view of a combinatorialist talking to combinatorialists, so
I will spend more time on developing the analytic theory, sometimes referring
to results in combinatorial optimization by name or sketch only. When writing
about results coming from several independent lines of research, it is
difficult to determine what is ``known'', what is a ``straightforward
adaptation'', and what is ``new''. Some of the less trivial results I could not
find in the literature are the extension of Choquet integration to non-monotone
submodular setfunctions (however natural it is; Lemmas \ref{LEM:BD-VAR} and
\ref{LEM:SUBMOD-FINVAR}), the lopsided Fubini Theorem
\ref{THM:SUBMOD-UNCROSS3}, the intersection theorem for submodular setfunctions
(Theorem \ref{THM:INTERSECT}), the constructions for the positive part (Lemma
\ref{LEM:DILW}) and for the part continuous from below (Lemma
\ref{LEM:SEMIC-PART}), this form of the Björner distance and its application to
generalizing flats in matroids (Section \ref{SEC:BJ-DIST}), and the
representation of strongly submodular setfunctions in the algebra of finite
subsets of a set (Theorem \ref{THM:KOLM-BOR}).

\section{Preliminaries}

\subsection{Set families and measurable functions}

Let $J$ be a (finite or infinite) set, and $\FF\subseteq 2^J$, a family of its
subsets such that $\emptyset,J\in\FF$. We say that $\FF$ is a {\it lattice of
sets}, if it is closed under the union and intersection of two elements. A set
lattice closed under the operation of difference ($A\setminus B$) is called a
{\it ring of sets}. A ring of sets containing $J$ (equivalently, closed under
complementation) is a {\it set-algebra}. A set-algebra closed under countable
unions is a {\it sigma-algebra}. Classical combinatorial results concern the
algebra of all subsets of a finite set; our interest will be in infinite
set-algebras. In most cases, when we speak about a set-algebra, the reader may
think (without losing the main ideas) of the family of Borel subsets of
$[0,1]$.

Let $U$ and $V$ be two sets and $R\subseteq U\times V$, a relation. For
$X\subseteq U$, define
\begin{align}
&R(X)=\{y\in V:~(x,y)\in R\text{ for some }x\in X\},\label{EQ:RDEF}\\
&\QR(X)=R^c(X)^c=\{y\in V:~(x,y)\in R \text{ for all } x\in X\},\label{EQ:QDEF}
\end{align}
where $R^c$ is the complementary relation $(U\times V)\setminus R$.

If $\BB$ is a sigma-algebra, then a function $f:~J\to\R$ is said to be
$\BB$-measurable (or simply measurable) if $\{f\ge t\}\in\BB$ for all
$t\in\RR$. (Here $\{f\ge t\}$ is shorthand for $\{x\in J:~f(x)\ge t\}$.) In the
case of more general set families like set-algebras, this would not be a good
definition; for example, it would not imply that $-f$ is measurable. Even if we
add the condition that $\{f\le t\}\in\BB$, it would not follow that the sum of
two $\BB$-measurable functions is $\BB$-measurable. Hence we adopt the
following weaker notion: given a set family $\FF\subseteq 2^J$, we call $f$
{\it $\FF$-measurable}, if for every $s<t$ there is a set $A\in\FF$ such that
$\{f\ge t\}\subseteq A\subseteq\{f\ge s\}$. We denote by $\Bd(\FF)$ the set of
bounded $\FF$-measurable functions.

A closely related notion of {\it $\FF$-continuous functions} is introduced by
Rao and Rao \cite{Rao}, Section 4.7. This is defined by the condition that for
every $\eps>0$ there is a partition $J=B_1\cup\ldots\cup B_n$ into a finite
number of sets in $\FF$ such that $|f(x)-f(y)|<\eps$ for every $1\le i\le n$
and $x,y\in B_i$. For a set-algebra $(J,\FF)$ and a bounded function $f$, this
is equivalent to $\FF$-measurability.

If $\{f\ge t\}\in\FF$ for all $t\in\R$, then $f$ is $\FF$-measurable, but not
the other way around. If $\FF$ is closed under countable intersections (in
particular, if $(J,\FF)$ is a sigma-algebra), then this notion coincides with
the traditional definition of measurability. It also known \cite{Rao} that for
a set-algebra $(J,\BB)$, bounded $\BB$-measurable functions form a linear
space, which we denote by $\Bd(\BB)$ (or simply by $\Bd$ if $\BB$ is
understood). With the norm $\|f\|=\sup_{x\in J} |f(x)|$, the space $\Bd$ is a
Banach space.

We need to extend this fact a little.

\begin{lemma}\label{LEM:FF-LINEAR}
Let $\FF$ be a lattice family containing $\emptyset$ and $J$. Then $\Bd(\FF)$
is a linear space of functions.
\end{lemma}

\begin{proof}
It is trivial that $\Bd(\FF)$ is closed under multiplication by a scalar. Let
$f,g\in\Bd(\FF)$, we want to prove that $f+g\in\Bd(\FF)$. Clearly $f+g$ is
bounded, so we want to show that it is $\FF$-measurable.

Let $s<t$. We want to construct a set $C\in\FF$ such that
\begin{equation}\label{EQ:FG-C}
\{f+g\ge t\}\subseteq C\subseteq\{f+g\ge s\}.
\end{equation}
Let $\eps=(t-s)/4$ and $m=\lceil t/\eps\rceil$. There is a $K\in\Nbb$ such that
$|f(x)|,\,|g(x)|\le K\eps$ for all $x\in J$. Since $f$ is $\FF$-measurable, for
every $-K\le n\le K$ there is a set $A_n\in\FF$ such that $\{f\ge
n\eps\}\subseteq A_n\subseteq \{f\ge (n-1)\eps\}$. Similarly, there is a set
$B_r\in\FF$ such that $\{g\ge r\eps\}\subseteq B_r\subseteq \{g\ge
(r-1)\eps\}$. Consider the set
\[
C=\bigcup_{n=-K}^K A_n\cap B_{m-n-2},
\]
which is clearly in $\FF$. Let $x\in C$, then $x\in A_n\cap B_{m-n-2}$ for some
$n$, and
\[
f(x)+g(x) \ge (n-1)\eps + (m-n-3)\eps = (m-4)\eps \ge t-4\eps= s,
\]
and so $C\subseteq \{f+g\ge s\}$. On the other hand, let $x\in J$ be such that
$f(x)+g(x)\ge t$. Set $n=\lfloor f(x)/\eps\rfloor$, then $f(x)\ge n\eps$ and
hence $x\in A_n$. Furthermore, $g(x)\ge t-f(x) \ge (m-1)\eps-(n+1)\eps
=(m-n-2)\eps$, and hence $x\in B_{m-n-2}$. So $x\in C$, showing that $\{f+g\ge
t\}\subseteq C$.
\end{proof}

\subsection{Setfunctions}

A {\it setfunction} is any (finite) real valued function defined on a family
$\FF$ of subsets of an underlying set $J$. If we say ``setfunction'' without
further specification, we mean a setfunction on the set family $\BB$ for the
set algebra denoted by $(J,\BB)$.

A setfunction is {\it increasing}, if $\fg(X)\le\fg(Y)$ for $X\subseteq Y$; the
phrases {\it decreasing} and {\it strictly increasing} are defined in the
obvious way. (Whenever writing $\fg(X)$, we tacitly assume that $X\in\BB$.) For
every set $X\subseteq J$, we denote its complement by $X^c=J\setminus X$, and
we define the ``complementary'' setfunction by $\fg^c(X)=\fg(X^c)$. We set
$\|\fg\| = \sup_X |\fg(X)|$. If $\fg$ is increasing, then $\|\fg\| =
\max\{\fg(J), -\fg(\emptyset)\}$, and so any increasing (or decreasing)
setfunction is necessarily bounded.

Let $\FF\subseteq 2^J$ be a family of sets containing $\emptyset$ and $J$. (For
our purposes, $\FF$ will always be a lattice family.) For a setfunction
$\fg:~\FF\to\R$, we define four ``monotonized'' setfunctions
$\fg^\li,\fg^\ls,\fg^\ui,\fg^\us:~2^J\to\RR$ by
\begin{align*}
\fg^\li(X)&=\inf_{Y\in\FF\atop Y\subseteq X}\fg(Y),\qquad \fg^\ui(X)=\inf_{Y\in\FF\atop Y\supseteq X}\fg(Y),\\
\fg^\ls(X)&=\sup_{Y\in\FF\atop Y\subseteq X}\fg(Y),\qquad \fg^\us(X)=\sup_{Y\in\FF\atop Y\supseteq X}\fg(Y).
\end{align*}
(Here the superscipt ``$\li$'' refers to ``lower-infimum'' etc.) Not all four
of these operations will play equally important roles below, but it is easiest
to introduce all of them. The setfunction $\fg^\ui$ is the exterior measure
defined by $\fg$, and it is often denoted by $\fg^*$.

It is straightforward to see that $\fg^\ui$ and $\fg^\ls$ are increasing
setfunctions on $2^J$, and $\fg^\li$ and $\fg^\us$ are decreasing. If $\fg$ is
increasing, then $\fg^\li(X)=\fg(\emptyset)$ and $\fg^\us(X)=\fg(J)$ for all
$X\subseteq J$, and $\fg^\ui$ and $\fg^\ls$ are extensions of $\fg$ (they agree
with $\fg$ on $\FF$). Also note that if $\FF$ is closed under complementation,
then
\begin{equation}\label{EQ:UP-DOWN}
\fg^\li=((\fg^c)^\ui)^c.
\end{equation}

The next operations generalize a construction known for charges. For two
bounded setfunctions $\fg,\psi$ on a set-algebra $(J,\BB)$, let us define the
setfunctions
\begin{align}
(\fg\lor\psi)(X) &= \sup_{Y\subseteq X}\big( \fg(Y)+\psi(X\setminus Y)\big),\label{EQ:LOR-DEF}\\
(\fg\land\psi)(X) &= \inf_{Y\subseteq X} \big(\fg(Y)+\psi(X\setminus Y)\big)\label{EQ:LAND-DEF}.
\end{align}
If $\fg(\emptyset)=\psi(\emptyset)=0$, then
$\fg\land\psi\le\min(\fg,\psi)\le\max(\fg,\psi)\le\fg\lor\psi$. It is easy to
see that these operations are commutative and associative. Trivially,
\[
\fg^\li=\fg\land 0\et \fg^\ls=\fg\lor 0.
\]
For every constant $c$, we have
\begin{equation}\label{EQ:LAND-C}
(\fg+c)\land \psi = (\fg\land \psi) + c, \et (\fg+c)\lor \psi = (\fg\lor \psi) + c.
\end{equation}

\subsection{Setfunctions with bounded variation}

Let $(J,\BB)$ be a set-algebra. We say that a setfunction $\fg$ on $\BB$ has
{\it bounded variation}, if there is a $K\in\R$ such that
\[
\sum_{i=1}^n |\fg(X_i)-\fg(X_{i-1})| \le K
\]
for every chain of subsets $\emptyset=X_0\subseteq X_1\subseteq\ldots\subseteq
X_n=J$. We denote the smallest number $K$ for which this holds by $\var(\phi)$.
It is clear that every bounded increasing or decreasing setfunction has bounded
variation. Every charge has bounded variation. In Section \ref{SEC:SUBMOD} we
will encounter a large class of setfunctions with bounded variation, namely
submodular setfunctions.

\begin{lemma}\label{LEM:BD-VAR}
A setfunction on a set-algebra $(J,\BB)$ can be written as the difference of
two increasing setfunctions if and only if it has bounded variation.
\end{lemma}

For a charge, $\alpha=\alpha_+-\alpha_-$ provides such a decomposition.

\begin{proof}
First, assume that $\fg=\mu-\nu$, where $\mu$ and $\nu$ are increasing
setfunctions. We may assume that
$\fg(\emptyset)=\mu(\emptyset)=\nu(\emptyset)=0$. Let $\emptyset=X_0\subseteq
X_1\subseteq\ldots\subseteq X_n=J$. Then
\begin{align*}
\sum_{i=1}^n &|\fg(X_i)-\fg(X_{i-1})| = \sum_{i=1}^n |\mu(X_i)-\mu(X_{i-1}) - \nu(X_i) + \nu(X_{i-1})|\\
&\le \sum_{i=1}^n  |\mu(X_i)-\mu(X_{i-1})| +  \sum_{i=1}^n |\nu(X_i)-\nu(X_{i-1})|\\
&= \sum_{i=1}^n  (\mu(X_i)-\mu(X_{i-1})) +  \sum_{i=1}^n (\nu(X_i)-\nu(X_{i-1})) = \mu(J)+\nu(J).
\end{align*}
So $\fg$ satisfies the definition of bounded variation with $K=\mu(J)+\nu(J)$.

Second, assume that $\fg$ has bounded variation. We may assume that
$\fg(\emptyset)=0$. For $S\in\BB$, let
\begin{align}\label{EQ:ALPHA-BETA-DEF}
\mu(S)&=\sup \sum_{i=1}^n |\fg(X_i)-\fg(X_{i-1})|_+,\nonumber\\
\nu(S)&=\sup \sum_{i=1}^n |\fg(X_i)-\fg(X_{i-1})|_-,
\end{align}
where the suprema are taken over all chains $\emptyset=X_0\subset
X_1\subset\ldots\subset X_n=S$. Note that
\[
\sum_{i=1}^n |\fg(X_i)-\fg(X_{i-1})|_+ - \sum_{i=1}^n |\fg(X_i)-\fg(X_{i-1})|_- = \sum_{i=1}^n (\fg(X_i)-\fg(X_{i-1})) = \fg(S),
\]
which implies that in the suprema above the same chains of sets approximate the
optima.

The setfunction $\mu$ is increasing. Indeed, let $S\subseteq T$; whenever a
sequence $\emptyset=X_0\subseteq X_1\subseteq\ldots\subseteq X_n=S$ competes in
the definition of $\mu(S)$, the sequence $\emptyset=X_0\subseteq
X_1\subseteq\ldots\subseteq X_n\subseteq X_{n+1}=T$ competes in the definition
of $\mu(T)$. Similarly $\nu$ is increasing.

We can rewrite the definition of $\nu$ as
\begin{align*}
\nu(S)&=\sup \sum_{i=1}^n |\fg(X_i)-\fg(X_{i-1})|_+ - (\fg(X_i)-\fg(X_{i-1}))\\
&= \sup \sum_{i=1}^n |\fg(X_i)-\fg(X_{i-1})|_+ - \fg(S) = \mu(S)-\fg(S),
\end{align*}
showing that $\mu-\nu=\fg$.
\end{proof}

Let $X=(X_0,X_1,\dots, X_n)$ range over all finite chains of sets in $\BB$ such
that $X_0=\emptyset$ and $X_n=J$. Then for the setfunctions $\mu$ and $\nu$
defined above,
\begin{align*}
\mu(J)+\nu(J) &= \sup_X \sum_{i=1}^n |\fg(X_i)-\fg(X_{i-1})|_+ + \sup_X \sum_{i=1}^m |\fg(X_i)-\fg(X_{i-1})|_-\\
&= \sup_X \Big(\sum_{i=1}^n |\fg(X_i)-\fg(X_{i-1})|_+ + \sum_{i=1}^m |\fg(X_i)-\fg(X_{i-1})|_-\Big)\\
&=\sup_X \sum_{i=1}^n |\fg(X_i)-\fg(X_{i-1})| = \var(\fg).
\end{align*}
(we can take the supremum over the same chains $X$, as remarked above). In
other words, $\|\mu\|+\|\nu\|=\|\mu+\nu\|=\var(\fg)$. Similar computation shows
that this is the most ``economical'' decomposition of $\fg$ as the difference
of two increasing setfunctions, in the sense that if $\fg=\gamma-\delta$, where
$\gamma$ and $\delta$ are increasing setfunctions with
$\gamma(\emptyset)=\delta(\emptyset)=0$, then
$\gamma(J)+\delta(J)\ge\var(\fg)$.

\subsection{Choquet integrals}\label{SSEC:CHOQ-INT}

Let $(J,\BB)$ be a sigma-algebra and $f\in\Bd_+(\BB)$. The ``Layer Cake
Representation'' of $f$ is the following formula:
\[
f(x) = \int_0^\infty \one(f(x)\ge t))\,dt.
\]
Let $\fg$ be an increasing setfunction with $\fg(\emptyset)=0$. The {\it
Choquet integral} was introduced by Choquet \cite{Choq}, and it is defined by
\begin{equation}\label{EQ:CHOQ-DEF}
\wh\fg(f) = \intl_0^\infty \fg\{f\ge t\}\,dt,
\end{equation}
where $\{f\ge t\}$ is shorthand for $\{x\in J:~f(x)\ge t\}$. This integral is
well defined, since $\fg\{f\ge t\}$ is a bounded monotone decreasing function
of $t$, and the integrand is zero for sufficiently large $t$.

More generally, if $f\in\Bd$ may have negative values, then we select any
$c\le\inf(f)$, and define
\begin{equation}\label{EQ:WH-DEF}
\wh\fg(f) = \intl_c^\infty \fg\{f\ge t\}\,dt + c\fg(J) = \wh\fg(f-c)+c\fg(J).
\end{equation}
It is easy to see that this value is independent of $c$ once $c\le\inf(f)$.

The formula defining $\wh\fg(f)$ makes sense whenever $\fg\{f\ge t\}$ is an
integrable function of $t$. One sufficient condition for this is that
$\fg\{f\ge t\}$ is the difference of two bounded monotone functions of $t$. In
turn, a sufficient condition for this is that $\fg$ is the difference of two
increasing setfunctions. By Lemma \ref{LEM:BD-VAR}, this means that $\fg$ has
bounded variation. Thus for a setfunction $\fg$ with bounded variation and
$\fg(\emptyset)=0$, we can define $\wh\fg(f)$ by the integral formula
\eqref{EQ:WH-DEF}. We call $\wh\fg$ the {\it Choquet extension} of $\fg$.

\begin{remark}\label{RE:NONLIN-INT}
The Choquet integral is usually denoted by $\int f\,d\fg$ instead of
$\wh\fg(f)$. However, to avoid confusion, we prefer to use the notation
$\wh\fg$. In the case of a charge $\fg$, the integral notation is well
established, along with the simple notation $\fg(f)$, and it coincides with the
appropriate special case of our definition.
\end{remark}

Let $\BB$ be a set-algebra (rather than a sigma-algebra), $\fg$, an increasing
setfunction on $\BB$, and $f\ge0$, a bounded function that is $\BB$-measurable
in the sense defined above. Formula \eqref{EQ:CHOQ-DEF} is not necessarily
meaningful, because the level sets $\{f\ge t\}$ may not belong to $\BB$. One
remedy is to replace $\fg$ by either $\fg^\ls$ or by $\fg^\ui$, which agree
with $\fg$ on $\BB$ and are defined everywhere. But can we get different
integrals? Luckily, we have the following lemma.

\begin{lemma}\label{LEM:SET-ALG-INT}
Let $(J,\BB)$ be a set-algebra, let $\fg$ be an increasing setfunction on $\BB$
with $\fg(\emptyset)=0$, and let $f:~J\to\R$ be a bounded $\BB$-measurable
function. Then $\fg^\ui\{f\ge t\} = \fg^\ls\{f\ge t\}$ for almost all real
numbers $t$.
\end{lemma}

\begin{proof}
Trivially $\fg^\ls\le \fg^\ui$, and $\fg^\ls\le\fg\le \fg^\ui$ on $\BB$.
$\BB$-measurability of $\fg$ implies that for every $t$ and $\eps>0$ there is a
set $A\in\BB$ such that
\[
\{f\ge t\}\subseteq A \subseteq\{f\ge t-\eps\},
\]
which implies that
\[
\fg^\ui\{f\ge t\} \le \fg(A) \le \fg^\ls\{f\ge t-\eps\},
\]
and hence (assuming for simplicity that $f\ge 0$)
\begin{align*}
\int_0^\infty \fg^\ui\{f\ge t\}\,dt &\le \int_0^\infty \fg^\ls\{f\ge t-\eps\}\,dt
= \int_{-\eps}^\infty \fg^\ls\{f\ge t\}\,dt\\
&= \eps\fg(J) + \int_0^\infty \fg^\ls\{f\ge t\}\,dt.
\end{align*}
Letting $\eps\to 0$, we get that $\wh{\fg^\ui}=\wh{\fg^\ls}$, which implies the
lemma.
\end{proof}

We will denote this common value by $\fg\{f\ge t\}$ (keeping in mind that it is
defined for almost all $t$ only). Thus the integral in \eqref{EQ:CHOQ-DEF} is
well defined, and it gives that $\wh\fg(f)=\wh{\fg^\ui}(f)=\wh{\fg^\ls}(f)$ for
nonnegative functions $f$. Formula \eqref{EQ:WH-DEF} remains valid to define
$\wh\fg(f)$ for all bounded functions.

\begin{example}\label{EXA:FINADD-INT}
If $\fg$ is a charge on a sigma-algebra $(J,\BB)$, then $\wh\fg$ is a linear
functional on $\Bd$, and conversely, if the functional $\wh\fg$ is linear, then
$\fg$ is a charge.
\end{example}

\begin{example}\label{EXA:LOV-EXT1}
Let $r$ be the rank function of a matroid $(E,r)$, and let $w:~E\to\R_+$. We
can write, uniquely, $w=\sum_{u=1}^k a_i \one_{A_i}$, where $k\ge 0$, $a_i>0$,
and $\emptyset\subset A_1\subset A_2\subset\ldots\subset A_k\subseteq E$. Then
\[
\wh{r}(w)=\sum_{i=1}^k a_i r(A_i).
\]
We can interpret this as the result of the Greedy Algorithm for finding a basis
$B$ with maximum weight. We start with finding the most valuable elements
(those with maximal weight $w$, namely the elements of $A_1$), and insert into
$B$ as many of them as possible, which is $r(A_1)$. Then look at the next best
elements etc.
\end{example}

It is easy to see that the map $\fg\mapsto\wh\fg$ is linear:
\begin{equation}\label{EQ:WHFI-LIN}
\wh{\fg_1+\fg_2} = \wh\fg_1 +\wh\fg_2,\et \wh{c\fg}=c\wh{\fg}
\end{equation}
for any two setfunctions $\fg_1$ and $\fg_2$ with bounded variation and for any
$c\in\R$. In particular, if $\fg = \fg_1 -\fg_2$ is a representation of $\fg$
as the difference of two increasing setfunctions, then
\[
\wh\fg(f) = \wh\fg_1(f)-\wh\fg_2(f).
\]
We have to warn the reader that the functional $f\mapsto\wh\fg(f)$ is not
linear in general (it is linear if $\fg$ is a charge).

For $S\in\BB$, we have $\wh\fg(\one_S)=\fg(S)$. So $\wh\fg$ can be considered
as an extension of $\fg$ from $0$-$1$ valued measurable functions to all
bounded measurable functions (in the finite case, this was introduced in
\cite{LL83}). It is also trivial that $\wh\fg$ is positive homogeneous: for
every real number $a\ge0$, we have $\wh\fg(af)=a\,\wh\fg(f)$. We also have the
identity
\begin{equation}\label{EQ:ADD-CONST}
\wh\fg(f+a)=\wh\fg(f)+a\fg(J)\qquad(a\in\R).
\end{equation}
Indeed, notice that $\{f+a\ge t\} = \{f\ge t-a\} = J$ if $t<a$, and so for a
sufficiently small $c$,
\[
\wh\fg(f+a)= \intl_c^a \fg(J)\,dt + \intl_a^\infty \fg\{f\ge t-a\}\,dt = a\fg(J)+\wh\fg(f).
\]

The following inequality follows easily from the definition. Let $\fg$ be a
setfunction with bounded variation and with $\fg(\emptyset)=0$. Let $f\in\Bd$,
$a\le f\le b$ (where $a\le b\in\R$). Then
\begin{equation}\label{EQ:WHFI-BOUND}
(b-a) \inf_{X\in\BB} \fg(X) \le \wh\fg(f) \le (b-a) \sup_{X\in\BB} \fg(X).
\end{equation}
Hence
\begin{equation}\label{EQ:WHFI-NORM}
|\wh\fg(f)|\le \|\fg\|\,\|f\|.
\end{equation}
If $f\le g$, then $\{f\ge t\}\subseteq \{g\ge t\}$; this implies that for an
increasing setfunction $\fg$, the functional $f\mapsto \wh\fg(f)$ is monotone
increasing. The following lemma shows that even though the functional
$f\mapsto\wh\fg(f)$ is not linear, it is continuous (even Lipschitz).

\begin{lemma}\label{LEM:WHPHI-CONT}
Let $\fg$ be a setfunction with bounded variation on a set-algebra $(J,\BB)$.
Then
\begin{equation*}\label{EQ:WHF-CONT}
|\wh\fg(f)-\wh\fg(g)| \le \var(\fg)\, \|f-g\| \qquad(\forall f,g\in\Bd).
\end{equation*}
\end{lemma}

In particular, if $\fg$ is increasing, then the following more explicit bound
holds:
\begin{equation}\label{EQ:WHF-CONT1}
|\wh\fg(f)-\wh\fg(g)| \le \|\fg\|\, \|f-g\|.
\end{equation}

\begin{proof}
First, assume that $\fg$ is increasing. We may assume that $\fg(\emptyset)=0$,
and that $\wh\fg(f)\ge \wh\fg(g)$. Let $\eps=\|f-g\|$, then $\{f\ge
t+\eps\}\subseteq \{g\ge t\}$, and so $\fg\{f\ge t+\eps\}\le\fg\{g\ge t\}$. Let
$c\le \min\{\inf(f),\inf(g)\}$, then
\begin{align*}
\wh\fg(f)-\wh\fg(g) &= \intl_c^\infty \fg\{f\ge t\}\,dt - \intl_c^\infty \fg\{g\ge t\}\,dt\\
&=\intl_{c-\eps}^\infty \fg\{f\ge t+\eps\}\,dt  - \intl_c^\infty \fg\{g\ge t\}\,dt\\
&= \intl_{c-\eps}^c \fg\{f\ge t+\eps\}\,dt + \intl_c^\infty \fg\{f\ge t+\eps\}\,dt  - \intl_c^\infty \fg\{g\ge t\}\,dt\\
&\le \intl_{c-\eps}^c \fg\{f\ge t+\eps\}\,dt \le \fg(J)\, \eps= \|\fg\|\, \|f-g\|.
\end{align*}
This proves the lemma in the increasing case.

Second, let $\fg=\fg_1-\fg_2$, where $\fg_1$ and $\fg_2$ are increasing. As
remarked above, $\var(\fg)=\|\fg_1\|+\|\fg_2\|$, and so
\begin{align*}
|\wh\fg(f)-\wh\fg(g)| &= |\wh\fg_1(f)-\wh\fg_2(f)-\wh\fg_1(g)+\wh\fg_2(g)|\\
&\le |\wh\fg_1(f)-\wh\fg_1(g)|+|\wh\fg_2(f)-\wh\fg_2(g)|\\
&\le \|\fg_1\|\,\|f-g\|+\|\fg_2\|\,\|f-g\| = \var(\fg)\,\|f-g\|.
\end{align*}
This completes the proof.
\end{proof}

We consider the continuity of $\fg\mapsto\wh\fg(f)$ in one more sense. Let
$(J,\BB)$ be a set-algebra, and consider the space $T=\R^\BB$, with the product
topology. A setfunction on $\BB$ is the same as a point in $T$.

\begin{lemma}\label{LEM:SETS-CLOSED}
The sets of submodular setfunctions, of nonnegative setfunctions, of charges,
of setfunctions with values in a closed interval $[a,b]$, and also of
setfunctions $\fg$ with $\var(\fg)\le 1$, are closed.
\end{lemma}

\begin{proof}
Let us prove that the set $Z$ of setfunctions with $\var(\fg)\le 1$ is closed;
all the other assertions follow by the same argument. Consider a setfunction
$\psi\notin Z$, then $\var(\psi)>1$. By the definition of $\var$, there is a
chain $\emptyset=X_0\subset X_1\subset\ldots\subset X_n=S$ such that
\[
\sum_{i=1}^n |\fg(X_i)-\fg(X_{i-1}|>1.
\]
For sufficiently small open intervals $U_0,\dots,U_n$ such that $\fg(X_i)\in
U_i$, we have $\sum_{i=1}^n |y_i-y_{i-1}|>1$ for all choices of $y_i\in U_i$.
This means that $\prod_{B\in \BB} W_B$, where
\[
W_B=
  \begin{cases}
    U_i, & \text{if $B=X_i$, $i=0,\dots,n$}, \\
    \R, & \text{otherwise},
  \end{cases}
\]
is an open set containing $\psi$ but avoiding $Z$.
\end{proof}

\begin{lemma}\label{LEM:T-CONT}
For every $f\in\Bd$, the mapping $\fg\mapsto\wh\fg(f)$ is continuous on
setfunctions with uniformly bounded variation.
\end{lemma}

\begin{proof}
Let $\fg$ be a setfunction with $\var(\fg)\le 1$ (say). The map
$\fg\mapsto\wh\fg(X)$ is trivially continuous for any $X\in\BB$ (it is the
projection onto a coordinate). Next we argue that $\fg\mapsto\wh\fg(g)$ is
continuous for every nonnegative stepfunction $g$. Indeed, we can write
\[
g=\sum_{i=1}^m a_i\one_{X_i},\quad a_i>0,\quad X_1\subset \dots\subset X_m,
\]
and then
\[
\wh\fg(g)= \sum_{i=1}^m a_i\fg(X_i),
\]
which depends continuously on $\fg$. For a general stepfunction, we can use
\eqref{EQ:WH-DEF} to reduce it to the nonnegative case.

Finally, consider any bounded function $f$. For any $n\ge 1$ there is a
stepfunction $g_n$ such that $|f(x)-g_n(x)|\le 1/n$ for every $x\in J$, and
hence by Lemma \ref{LEM:WHPHI-CONT},
\[
|\wh\fg(f)-\wh\fg(g_n)|\le \var(\fg)\|f-g_n\| \le \frac1n.
\]
This shows that $\fg\mapsto\wh\fg(f)$ is the uniform limit of continuous
functions, and so it continuous.
\end{proof}

\section{Submodular setfunctions}\label{SEC:SUBMOD}

\subsection{Definition and examples}\label{SSEC:SUBMOD-DEF}

A setfunction $\fg$ defined on a lattice family $\FF$ of subsets of a set $J$
is called {\it submodular}, if
\begin{equation}\label{EQ:SUBMOD-DEF}
\fg(S\cup T)+\fg(S\cap T)\le \fg(S)+\fg(T)
\end{equation}
for all $S,T\in\FF$. An equivalent way of saying this is that the difference
$\fg(X\cup A)-\fg(X)$ is a decreasing function of $X\in\FF$ for every fixed
$A\in\FF$. Adding a constant (which does not change the validity of
\eqref{EQ:SUBMOD-DEF}), we can usually assume for free that $\fg(\emptyset)=0$.
If $\fg$ is submodular on a set-algebra, then so is $\fg^c(X)=\fg(X^c)$, which
implies that most results about increasing submodular setfunctions can be
translated into results about decreasing ones.

A setfunction $\fg$ is {\it supermodular}, if \eqref{EQ:SUBMOD-DEF} holds with
the inequality reversed. We say that $\fg$ is {\it modular}, if it is both
submodular and supermodular, i.e., \eqref{EQ:SUBMOD-DEF} holds with equality
for all $S,T\in\BB$. As mentioned in the introduction, modular setfunctions
with $\fg(\emptyset)=0$ are exactly the charges on $\BB$.

Every submodular setfunction with $\fg(\emptyset)=0$ is {\it subadditive} in
the sense that
\begin{equation}\label{EQ:SUBADD-DEF}
\fg(X\cup Y)\le \fg(X)+\fg(Y)\qquad(X\cap Y=\emptyset)
\end{equation}
(but not the other way around). Monotone submodular setfunctions are similar
to, but different from, outer measures. Both are subadditive, but submodular
setfunctions are not necessarily countably subadditive (even modular
setfunctions are not, see any finitely additive but not countably additive
measure). Outer measures are subadditive, but not necessarily submodular.

Increasing submodular setfunctions on a set-algebra $(J,\BB)$ can be
characterized by the inequality
\begin{equation}\label{EQ:ALTERNATE}
\fg(A_0)-\fg(A_0\cup A_1)-\fg(A_0\cup A_2)+\fg(A_0\cup A_1\cup A_2)\le 0,
\end{equation}
which should hold for all $A_0,A_1,A_2\in\BB$. The increasing property follows
by taking $A_1=A_2$, and submodularity follows by taking $A_0=A_1\cap A_2$. It
is also easy to see that increasing submodular setfunctions satisfy
\eqref{EQ:ALTERNATE}: Indeed, we may replace $A_1$ by $A_1\cup A_0$ and $A_2$
by $A_2\cup A_0$ (which does not change the quantity in \eqref{EQ:ALTERNATE}),
and then replace $A_0$ by $A_1\cap A_2$ (which makes the inequality tighter),
to get the submodular inequality.

Inequality \ref{EQ:ALTERNATE} was introduced by Choquet \cite{Choq}, who calls
setfunctions satisfying it {\it 2-alternating}. See Section
\ref{SEC:STRONG-SUBMOD} for generalizations of this inequality.

We conclude this section with a number of interesting examples.

\begin{example}[Supremum]\label{EXA:SUP}
The setfunction $\fg(X)=\sup(X)$ is submodular on Borel subsets of $[0,1]$.
\end{example}

\begin{example}[Concave function of a measure]\label{EXA:CONCAVE}
Let $f:~[0,1]\to\R$ be a bounded function, and let $\lambda$ denote the
Lebesgue measure. Then $\fg(X)=f(\lambda(X))$ is submodular on Borel sets if
and only if $f$ is concave. (Note the almost paradoxical contrast with
Corollary \ref{COR:SUBMOD-CONV}.) To see this, let $X,Y\in\BB$, and
$\ell:~[0,1]\to\R$ be a linear function such that
$\ell(\lambda(X))=f(\lambda(X))$ and $\ell(\lambda(Y))=f(\lambda(Y))$. By the
concavity of $f$, we have $\ell(\lambda(X\cap Y))\ge f(\lambda(X\cap Y))$ and
$\ell(\lambda(X\cup Y))\ge f(\lambda(X\cup Y))$, and so by the linearity of
$\ell$, we get
\begin{align*}
f(\lambda(X\cup Y))+f(\lambda(X\cap Y)) &\le \ell(\lambda(X\cup Y))+\ell(\lambda(X\cap Y))\\
&=\ell(\lambda(X))+\ell(\lambda(Y)) = f(\lambda(X))+f(\lambda(Y)).
\end{align*}

In particular, $\fg=\sqrt{\lambda}$ is submodular. In this case, $\wh\fg$ is
the {\it Lorentz norm} with parameters $(1,2)$.
\end{example}

\begin{example}[Ideals]\label{EXA:IDEAL}
Let $\II\subseteq\BB$ be an ideal. For example, $\II$ could be the set of
finite subsets of $J$, or countable subsets, or Borel subsets of $[0,1]$ with
zero Lebesgue measure, etc. Then $\chi_\II=\one(X\notin\II)$ is an increasing
submodular setfunction. If $f\in\Bd$, then $\wh\chi_\II(f)$ is the supremum of
all values $t$ such that $\{x\in J:~f(x)\ge t\}\in\II$. In particular, if
$\II=\{\emptyset\}$, then $\wh\chi_\II(f)=\sup_{x\in J} f(x)$. If $\II$
consists of Borel subsets of $[0,1]$ with zero Lebesgue measure, then
$\wh\chi_\II(f)=\text{\rm sup~ess}_{x\in J} f(x)$.
\end{example}

\begin{example}[Dimension]\label{EXA:LIN-DEP}
Let $L$ be any finite dimensional linear space; for any $X\subseteq L$, let
$r(X)$ denote the maximum number of linearly independent elements in $X$. Then
$r(X)$ is an increasing submodular setfunction with $r(\emptyset)=0$. In
addition, it is integer valued.

A matroid is an abstraction of this example, although usually studied over
finite underlying sets. The rank function of any matroid is nonnegative,
increasing, submodular, integer valued, and bounded by the cardinality (see
Appendix \ref{APP-MATROIDS}). If we drop the last condition, we get {\it
polymatroids}.
\end{example}

\begin{example}[Image under a relation]\label{EXA:TRANSVERSAL}
Let $R\subseteq[0,1]^2$ be a Borel measurable relation and $\lambda$, the
Lebesgue measure on $[0,1]$. Then $\fg(X)=\lambda(R(X))$ is an increasing
submodular setfunction on the Borel sets in $[0,1]$. Indeed, trivially $R(X\cup
Y)=R(X)\cup R(Y)$ and $R(X\cap Y)\subseteq R(X)\cap R(Y)$, so
\begin{align*}
\lambda(R(X\cup Y)) + \lambda(R(X\cap Y)) &\le \lambda(R(X)\cup R(Y)) + \lambda(R(X)\cap R(Y))\\
&= \lambda(R(X))+\lambda(R(Y)).
\end{align*}
(The set $R(X)$ is not necessarily a Borel set, but it is analytic and
therefore Lebesgue measurable.) Example \ref{EXA:SUP} is a special case when
$R=\{(x,y)\in[0,1]^2: x\ge y\}$.

The finite analogue of this construction plays an important role in matching
theory. This goes back to the work of Rado \cite{Rado}. Let $G$ be a bipartite
graph with bipartition $V(G)=U\cup W$. For $X\subseteq U$, let $N(X)$ denote
the set of its neighbors, then $\gamma(X)=|N(X)|$ is a submodular setfunction.
The important role of this construction in Choquet's {\it theory of capacities}
will be sketched in Section \ref{SEC:STRONG-SUBMOD}.
\end{example}

\begin{example}[Measure of closure]\label{EXA:CLOSURE}
Let $\BB$ be the family of Borel sets in $[0,1]$. For $X\subseteq [0,1]$, let
$\overline{X}$ denote its closure. Define
$\overline{\lambda}(X)=\lambda(\overline{X})$. Then $\overline{\lambda}$ is
submodular. Indeed, let $S,T\in\BB$, then using that $\overline{S\cup
T}=\overline{S}\cup\overline{T}$ and $\overline{S\cap T}\subseteq
\overline{S}\cap\overline{T}$, we obtain the submodularity of
$\overline{\lambda}$ by a similar computation as in Example
\ref{EXA:TRANSVERSAL}.

This example can be generalized in several directions. We can replace $\BB$ by
the sigma-algebra of Borel sets in any Polish space. We can also replace
$\lambda$ by any increasing submodular setfunction on $\BB$. (These facts are
all special cases of Lemma \ref{LEM:STAR} below.)
\end{example}

\begin{example}[Capacity of cuts]\label{EXA:CUT}
Let $\Phi$ be a nonnegative charge on $\BB\times\BB$. Then the setfunction
$\fg(X)=\Phi(X\times X^c)$ is submodular. Indeed, it is straightforward to
check that
\begin{align*}
\fg(S)+\fg(T)&-\fg(S\cup T)-\fg(S\cap T)\\
&= \Phi((T\setminus S)\times(S\setminus T)) +\Phi((S\setminus T)\times(T\setminus S)) \ge 0.
\end{align*}
Since $\fg(\emptyset) =\fg(J)=0$, the setfunction $\fg$ is neither increasing
nor decreasing unless it is identically zero.

The finite analogue of this setfunction $\fg$ is an important setfunction in
graph theory, low theory, and probability theory. If $G=(V,E)$ is a finite
graph and $\Phi$ is the counting measure on the set of edges, then $\fg(X)$ as
defined above is the size of the cut defined by $X\subseteq V$. If the edges
are assigned nonnegative capacities, and $\Phi(Y)$ is the sum of capacities of
edges in $Y$, then $\fg(X)$ is the capacity of the cut. See \cite{LLflow} for
its use in extensions of flow theory to measures.
\end{example}

\begin{example}[Graph homomorphisms]\label{EXA:HOMS}
For two finite graphs $F$ and $G$, let $\hom(F,G)$ denote the number of
homomorphisms (adjacency preserving maps) $V(F)\to V(G)$. For every subset
$X\subseteq E(F)$, $F_X=(V(F),X)$ and $\fg(X)=\hom(F_X,G)$. Then $\fg$ is a
supermodular setfunction. Indeed, let $\Hom(F,G)$ denote the set of maps
$V(F)\to V(G)$ that are homomorphisms. Then
\begin{align*}
\Hom(F_{X\cup Y},G)&= \Hom(F_X,G)\cap\Hom(F_Y,G),\quad\text{and}\\
\Hom(F_{X\cap Y},G)&\supseteq \Hom(F_X,G)\cup\Hom(F_Y,G).
\end{align*}
Taking cardinalities and adding up, we get the supermodularity of $\fg$.

In the limit theory of dense graphs, we often normalize this number to define
the {\it homomorphism density} $t(F,G)=\hom(F,G)/|V(G)|^{|V(F)|}$. Then
$\psi(X)=t(F_X,G)$ is a supermodular setfunction. Since $t(F,G)$ does not
depend on the isolated nodes of $F$, the setfunction $\psi$ can be considered
as a setfunction on the finite subsets of $\binom{J}{2}$ (the set of unordered
pairs of elements of $J$) for any (possibly infinite) set $J$.
\end{example}

\begin{example}[Hitting probabilities]\label{EXA:PROB}
Let $(J,\BB,\pi)$ be a standard Borel probability space, and let $(P_x:~x\in
J)$ be a Markov kernel. Consider a random walk $(v^0,v^1,\ldots)$ defined by
this Markov kernel, where $v^0\in J$ is from a fixed distribution $\sigma$. For
$X\in\BB$, let $A_X$ be the event that $(v^0,v^1,\ldots)$ hits $X$. Then
$\Pr(A_X)$ is an increasing submodular function of $X$. Indeed, clearly
$A_{X\cup Y} = A_X\lor A_Y$ and $A_{X\cap Y}\Rightarrow (A_X\land A_Y)$, so
\begin{equation}\label{EQ:HIT}
\Pr(A_{X\cup Y}) + \Pr(A_{X\cap Y}) \le \Pr(A_X\lor A_Y) + \Pr(A_X\land A_Y) =
\Pr(A_X) + \Pr(A_Y).
\end{equation}

A related setfunction is the expected number $\HH(\pi,S)$ of steps before a
random walk, started (say) at a random point from the stationary distribution,
hits a set $S$. Then $\HH(\pi,S)$ is a supermodular setfunction of $S$, at
least on sets $S$ for which it is finite.

This example has an analogue in analysis, more exactly in the theory of
capacities, which is obtained (roughly speaking) by replacing the random walk
above by Brownian motion in an open domain in the plane. This is the structure
studied by Choquet \cite{Choq} leading to the introduction of the Choquet
integral and perhaps the first work studying submodular setfunctions in an
infinite setting.
\end{example}

\begin{example}[Submodular functions on a lattice]\label{EXA:LATTICE}
Let $(L;\lor,\land)$ be a lattice with top and bottom elements, and let
$(L,2^L)$ be the sigma-algebra consisting of subsets of $L$. Let $\mu$ be an
increasing submodular function on $L$. For $X\subseteq L$, define
$\fg(X)=\inf\{\mu(a):~a\ge x ~\forall x\in X\}$. Then $\fg$ is submodular on
$2^L$; indeed, let $X,Y\in 2^L$, $a\in X$ and $b\in Y$, then $a\lor b\ge X\cup
Y$ and $a\land b\ge X\cap Y$, and hence
\[
\mu(a)+\mu(b)\ge\mu(a\lor b)+\mu(a\land b)\ge \fg(X\cup Y)+\fg(X\cap Y).
\]
Taking the infimum over $a$ and $b$, we get the submodularity of $\fg$.

A notable special case is when $L$ consists of all linear subspaces of $\R^d$
and $\mu=\dim$; then $\fg$ (restricted to a finite family of subspaces) is an
archetypal example of a polymatroid. Another important example is the lattice
defined on either $\Zbb^d$ or $\R^d$ by the partial order $x\le y$ ($x,y\in
[0,1]^d$) meaning that $y$ dominates $x$ entry-by-entry. Submodular
setfunctions on these lattices have been considered as generalizations of the
finite case (see \cite{Mur1,Mur2,Bach}).
\end{example}

We close with two examples of submodular setfunctions which have been studied
and used (so far) in the finite case only.

\begin{example}[Entropy]\label{EXA:SHANNON}
Let $S,V$ be a finite a sets, and let $\eta$ be a probability distribution on
$S^V$. For $U\subseteq V$, let $\eta^U$ denote the marginal of $\eta$ on $S^U$.
If $H(\eta)$ denotes the entropy of $\eta$, then $H(\eta^U)$ is a submodular
setfunction of $U$; this fact is called {\it Shannon's Inequality}.
\end{example}

\begin{example}[Logarithms of determinants]\label{EXA:LOG-DET}
Let $A$ be a (symmetric) positive definite matrix, and for every subset $S$ of
its rows, let $A_S$ denote the submatrix formed by the rows and columns in $S$.
Then $\log\det(A_S)$ is a submodular setfunction of $S$.
\end{example}

\subsection{Basics}\label{SSEC:SUBMOD-BASICS}

In this section we collect some basic properties of submodular setfunctions. We
start with defining some simple operations creating submodular setfunctions
from other ones under appropriate conditions.

\smallskip

$\bullet$ The {\it truncation} of $\fg$ by a real number $c$ is defined by
$X\mapsto\min(c,\fg(X))$.

\smallskip

$\bullet$ For $A\in\BB$, the {\it restriction} $\fg_A:~\BB|_A\to\R$ and the
{\it projection} $\fg^A:~\BB|_A\to\R$ of $\fg$ to $A$ are defined by
$\fg_A(X)=\fg(X)$ and $\fg^A(X) = \fg(A^c\cup X)-\fg(A^c)$, respectively. The
latter operation is also called the {\it contraction} of $A^c$.

\smallskip

$\bullet$ If $(I,\AA)$ is another set-algebra, and $\Pi:~J\to I$ is a
measurable map, then the {\it quotient} of $\fg$ by $\Pi$ is the setfunction
$\fg\circ\Pi^{-1}$ on $\AA$ defined by
$(\fg\circ\Pi^{-1})(X)=\fg(\Pi^{-1}(X))$. We call a quotient {\it finite}, if
$I$ is finite. (In measure-theoretic language, a quotient could be called
``pushforward'', but this could be confusing considering the ``pullback'' to be
defined below.)

\smallskip

$\bullet$ Again, let $(I,\AA)$ is another set-algebra, and let
$\Gamma:~\AA\to\BB$ be a monotone map that preserves unions. The {\it pullback}
of $\fg$ by $\Gamma$ is the setfunction $\fg\circ\Gamma:~\AA\to\R$ defined by
$(\fg\circ\Gamma)(X)=\fg(\Gamma(X))$. The quotient by a measurable map
$\Pi:~J\to I$ is the special case of a pullback by $\Pi^{-1}$. Other special
cases are the measure of closure (Example \ref{EXA:CLOSURE}) and the image
under a relation (Example \ref{EXA:TRANSVERSAL}).

\smallskip

$\bullet$ Let $\fg$ be a submodular setfunction on $(J,\BB)$. Create a new
element $a\notin J$, and define $\BB'=\BB\cup\{U\cup\{a\}:~U\in\BB\}$. We
define a setfunction $\fg':~\BB'\to\R$ by
\[
\fg'(X)=
  \begin{cases}
    \fg(X) & \text{if $a\notin X$}, \\
    \fg((X\setminus\{a\})\cup A) & \text{if $a\in X$}.
  \end{cases}
\]
We call this extension {\it adding a representative of $A$}.

\smallskip

For the next lemma, we omit the proofs as they are straightforward.

\begin{lemma}\label{LEM:SUBMOD-TRIV}
Let $\fg$ be a submodular setfunction on a set-algebra $(J,\BB)$.

\medskip

{\rm(a)} Decreasing the value of $\fg(\emptyset)$ and/or $\fg(J)$ preserves
submodularity.

\smallskip

{\rm(b)} The restriction and projection of $\fg$ to a set $A\in\BB$ are
submodular.

\smallskip

{\rm(c)} For $A\in\BB$, the setfunction $X\mapsto \fg(A\cup X)$ is submodular.

\smallskip

{\rm(d)} The complementary setfunction $\fg^c$ is submodular.

\smallskip

{\rm(e)} Every quotient of $\fg$ is submodular.

\smallskip

{\rm(f)} If $\fg$ is increasing, then truncation by any constant preserves
submodularity.

\smallskip

{\rm(g)} If $\fg$ is increasing, then every pullback of $\fg$ is submodular.

\smallskip

{\rm(h)} If $\fg$ is increasing, then adding a representative of a set
$A\in\BB$ preserves submodularity.
\end{lemma}

We continue with a couple of less trivial lemmas of the same nature. First, a
general fact about extending submodular setfunctions.

\begin{lemma}\label{LEM:STAR}
If $(J,\FF)$ is a lattice family and $\fg:~\FF\to\R$ is submodular, then
$\fg^\li$ and $\fg^\ui$ are submodular (as functions $2^J\to\R$).
\end{lemma}

Recall that $\fg^\li$ (lower-infimum) is decreasing and $\fg^\ui$
(upper-infimum) is increasing.

\begin{proof}
To prove the assertion about $\fg^\li$, let $X,Y\subseteq J$, and let
$U,V\in\FF$ such that $U\subseteq X$ and $V\subseteq Y$. Then $U\cap V\subseteq
X\cap Y$ and $U\cup V\subseteq X\cup Y$, and hence
\[
\fg(U)+\fg(V) \ge \fg(U\cap V)+\fg(U\cup V) \ge \fg^\li(X\cap Y)+\fg^\li(X\cup Y).
\]
Taking the infimum of the left hand side over all choices of $U$ and $V$, we
get the submodularity of $\fg^\li$. Submodularity of $\fg^\ui$ follows
similarly (or by using \eqref{EQ:UP-DOWN}).
\end{proof}

It is well known (and easy to check) that the setfunctions $\alpha\land\beta$
and $\alpha\lor\beta$ are charges if $\alpha$ and $\beta$ are charges. This
implies that they are modular if $\alpha$ and $\beta$ are modular. However,
neither $\fg\land\psi$ nor $\fg\lor\psi$ is submodular in general even if both
$\fg$ and $\psi$ are submodular. But if one of them is modular, then at least
one of these operations preserves submodularity:

\begin{lemma}\label{LEM:LAND1}
Let $\fg:~\BB\to\R_+$ be a submodular setfunction on a set-algebra $(J,\BB)$,
and let $\mu:~\BB\to\R_+$ be a modular setfunction. Then $\fg\land\mu$ is
submodular. Furthermore, if $\fg$ and $\mu$ are increasing, then $\fg\land\mu$
is increasing.
\end{lemma}

We don't claim that $\fg\lor\mu$ is submodular. Note that $\fg^\li=\fg\land 0$,
so Lemma \ref{LEM:STAR} follows from this lemma, at least for the case
$\FF=\BB$.

\begin{proof}
Using the modularity of $\mu$, we can write
\[
(\fg\land\mu)(X) = \inf_{Y\subseteq X} \fg(Y)+\mu(X\setminus Y)
= \inf_{Y\subseteq X} (\fg(Y)+\mu(Y^c) +\mu(X)-\mu(J).
\]
Here $\fg(Y)+\mu(Y^c)$ is a submodular function of $Y$, and hence
$\inf_{Y\subseteq X} (\fg(Y)+\mu(Y^c)$ is a submodular function of $X$ by Lemma
\ref{LEM:STAR}. Since $\mu(X)-\mu(J)$ is a modular function of $X$,
submodularity of $\fg\land\mu$ follows.

To prove monotonicity, let $U,V\in\BB$, $U\subseteq V$, and consider any Borel
set $Z\subseteq V$. Then
\[
\fg(Z)+\mu(V\setminus Z) \ge \fg(U\cap Z) + \mu(U\setminus(U\cap Z)) \ge (\fg\land\mu)(U).
\]
Taking the infimum of the left side over all $Z$, we get $(\fg\land\mu)(V)\ge
(\fg\land\mu)(U)$.
\end{proof}

\begin{remark}\label{REM:MATROID-INTERSECT}
In the finite case, if $\fg$ and $\psi$ are rank functions of two matroids on
the same underlying set, then $\fg\land\psi$ is the rank function of their
intersection (family of common independent sets). Since this intersection is
not a matroid in general, this setfunction is not necessarily submodular even
in this very special case. (We will say more about this ``intersection''
operation in Section \ref{SSEC:COMMON-MEAS}.)

The special case of $\fg\land\mu$, where $\mu(X)=|X|$, is also well known from
matroid theory. If $\fg$ is an increasing integer valued submodular setfunction
on a finite set, satisfying $f(\emptyset)=0$, then $\fg\land\mu$ is a matroid.
For example, if $G$ is a bipartite graph with bipartition $V=U\cup W$, and
$\fg(X)$ denotes the number of neighbors of a set $X\subseteq U$, the
$\fg\land\mu$ is the transversal matroid of $G$ on $U$. In particular,
$(\fg\land\mu)(U)$ is the matching number of $G$.
\end{remark}

Monotonicity of submodular setfunctions is often easy to check using the
following simple lemma:

\begin{lemma}\label{LEM:SUBMOD-MON}
A submodular setfunction $\fg$ on a set-algebra is increasing $[\text{resp.\
decreasing}]$ if and only if $\fg(X)\le\fg(J)$ $[\text{resp.}\
\fg(X)\le\fg(\emptyset)]$ for every $X\in\BB$.
\end{lemma}

\begin{proof}
We prove the first assertion; the second follows similarly. The necessity of
the condition is trivial. To prove the sufficiency, let $X,Y\in\BB$,
$X\subseteq Y$. Then $Y\cap(Y^c\cup X)= X$ and $Y\cup(Y^c\cup X)= J$, whence by
submodularity,
\[
\fg(Y)+\fg(Y^c\cup X) \ge \fg(X)+\fg(J),
\]
or
\[
\fg(Y)-\fg(X)\ge \fg(J)- \fg(Y^c\cup X) \ge 0,
\]
proving that $\fg$ is increasing.
\end{proof}

\subsection{Bounded variation}

The following lemma is important because it implies that we can perform Choquet
integration and define the functional $\wh\fg$ for every bounded (but not
necessarily increasing or decreasing) submodular setfunction $\fg$.

\begin{lemma}\label{LEM:SUBMOD-FINVAR}
Every bounded submodular setfunction on a set-algebra has bounded variation.
\end{lemma}

\begin{proof}
Consider a chain of subsets $\emptyset=X_0\subseteq X_1\subseteq\ldots\subseteq
X_n=J$ $(X_i\in\BB)$. Suppose that there is an index $1\le i\le n-1$ such that
$\fg(X_i)\le \fg(X_{i-1})$ and $\fg(X_i)\le \fg(X_{i+1})$. Let
$X_i'=X_{i-1}\cup(X_{i+1}\setminus X_i)$. Then $X_i\cap X_i'= X_{i-1}$ and
$X_i\cup X_i'= X_{i+1}$, so by submodularity,
\[
\fg(X_i)+\fg(X_i') \ge \fg(X_{i-1})+\fg(X_{i+1}).
\]
This implies that $\fg(X_i')\ge \fg(X_{i-1})$ and $\fg(X_i')\ge \fg(X_{i+1})$,
and also that
\begin{align*}
|\fg(X_{i+1})&-\fg(X_i)| + |\fg(X_i)-\fg(X_{i-1})|
= \fg(X_{i+1})+\fg(X_{i-1})-2\fg(X_i)\\
&\le \fg(X_{i+1})+\fg(X_{i-1})-2(\fg(X_{i-1})+\fg(X_{i+1}) -\fg(X'_i))\\
&= 2\fg(X_i')-\fg(X_{i+1})-\fg(X_{i-1})\\
&=|\fg(X_{i+1})-\fg(X_i')| + |\fg(X_i')-\fg(X_{i-1})|.
\end{align*}
So replacing $X_i$ by $X_i'$ does not decrease the sum $\sum_i
|\fg(X_i)-\fg(X_{i-1})|$. Repeating this exchange procedure a finite number of
times, we get a sequence $\emptyset=Y_0\subseteq Y_1\subseteq\ldots\subseteq
Y_n= J$ for which $\sum_i |\fg(Y_i)-\fg(Y_{i-1})|\ge\sum_i
|\fg(X_i)-\fg(X_{i-1})|$ and there is a $0\le j\le n$ such that
$\fg(Y_0)\le\ldots\le\fg(Y_j)\ge\fg(Y_{j+1})\ge\ldots\ge\fg(Y_n)$. For such a
sequence,
\[
\sum_{i=1}^n |\fg(Y_i)-\fg(Y_{i-1})| = 2\fg(Y_j)-\fg(J)-\fg(\emptyset).
\]
Since $\fg$ is bounded, this proves that $\fg$ has bounded variation.
\end{proof}

The argument above, combined with \eqref{EQ:ALPHA-BETA-DEF} in the proof of
Lemma \ref{LEM:BD-VAR}, gives the decomposition $\fg=\alpha+\beta$, where
$\alpha = \fg\lor 0= \fg^\ls$ and $\beta = \fg-\alpha$. Let us state the main
properties of this decomposition:

\begin{lemma}\label{LEM:VARPHI-DECOMP}
Let $\fg$ be a bounded submodular setfunction on a set-algebra with
$\fg(\emptyset)=0$. Then $\fg^\ls$ is increasing, $\beta=\fg-\fg^\ls$ is
decreasing, and both $\fg^\ls$ and $\fg-\fg^\ls$ are subadditive.
\end{lemma}

\begin{proof}
It is trivial that $\fg^\ls$ is increasing. Let $X_1,X_2\in\BB$, $X_1\cap
X_2=\emptyset$. Set $X=X_1\cup X_2$, and consider any $U\subseteq X$,
$U\in\BB$. Let $U_i=U\cap X_i$. Then
\[
\fg(U) \le \fg(U_1)+\fg(U_2) \le \fg^\ls(X_1)+\fg^\ls(X_2).
\]
Taking the supremum in $U$, it follows that
$\fg^\ls(X)\le\fg^\ls(X_1)+\fg^\ls(X_2)$, so $\alpha$ is subadditive.

Turning to $\beta=\fg-\fg^\ls$, let $X,Y\in\BB$ and $X\subseteq Y$. For any
$U\subseteq X$, submodularity gives that
\[
\fg(Y)-\fg^\ls(Y) \le \fg(Y)-\fg(U\cup(Y\setminus X)) \le \fg(X)-\fg(U),
\]
and taking the infimum in $U$, it follows that $\beta(Y)\le\beta(X)$.

To prove subadditivity, let $X_1,X_2\in\BB$, $X_1\cap X_2=\emptyset$ and
$U_i\subseteq X_i$, $U_i\in\BB$ ($i=1,2$). Set $X=X_1\cup X_2$ and $U=U_1\cup
U_2$. Then by submodularity,
\begin{equation}\label{EQ:FX1U2}
\fg(X_1\cup U_2) + \fg(X_2\cup U_1) \ge \fg(U) + \fg(X).
\end{equation}
Here again by submodularity,
\[
\fg(X_1\cup U_2) = \fg(X_1\cup U)\le \fg(X_1)+\fg(U)-\fg(U_1),
\]
and similarly,
\[
\fg(X_2\cup U_1) \le \fg(X_2)+\fg(U)-\fg(U_2).
\]
Combining with \eqref{EQ:FX1U2}, we get
\[
\fg(X_1)-\fg(U_1) + \fg(X_2)-\fg(U_2)\ge \fg(X)-\fg(U)\ge \beta(X).
\]
Taking the infimum of the left hand side in $U_1$ and $U_2$, we get that
$\beta(X_1)+\beta(X_2)\ge \beta(X)$.
\end{proof}

The setfunctions $\fg^\ls$ and $\fg-\fg^\ls$ in Lemma \ref{LEM:VARPHI-DECOMP}
are not submodular in general. As an example, let $S=\{a,b,c\}$, and
\[
\fg(X)=
  \begin{cases}
    0, & \text{if $X=\emptyset$ or $X=S$}, \\
    2, & \text{if $X=\{a,b\}$}, \\
    1, & \text{otherwise}.
  \end{cases}
\]
Then $\fg$ is submodular, and $\fg^\ls=\fg$ except that $\fg^\ls(J)=2$, and so
$X=\{a,c\}$ and $Y=\{b,c\}$ violate the submodular inequality. An example where
$\fg-\fg^\ls$ is not submodular is similarly easy: with $S=\{a,b,c\}$, let
\[
\fg(X)=
  \begin{cases}
    1, & \text{if $X=\{a\}$ or $X=\{b\}$ or $X=\{a,b\}$}, \\
    0, & \text{otherwise}.
  \end{cases}
\]

In the finite case, it is easy to write a submodular setfunction as the sum of
an increasing modular setfunction and a decreasing submodular setfunction: For
any sufficiently large $K>0$, $\alpha=K|X|$ is increasing and modular, while
the difference $\fg-K|X|$ is decreasing and submodular. In the general case, an
analogous construction works if there is a charge $\alpha\in\ba_+$ such that
$\fg\le\alpha$ (see Section \ref{SSEC:SBOUNDED} for a discussion of this
property). I don't know if more general (perhaps all) bounded submodular
setfunctions can be written as the sum of two submodular setfunctions, one
increasing and one decreasing.

\subsection{Weighting}

We define the {\it weighting} of a setfunction $\fg$ with bounded variation by
a function $w\in\Bd_+$ as the setfunction $(w\cdot\fg)(X)=\wh\fg(w\one_X)$.
Explicitly, assuming $\fg\ge0$,
\begin{equation}\label{EQ:WCDOT}
(w\cdot\fg)(X)=\int_0^\infty \fg(X\cap \{w\ge t\})\,dt.
\end{equation}
It is easy to see that if $\fg$ is a charge, then so is $w\cdot\fg$. Some
important properties of this operation are summarized in the next lemma.

\begin{lemma}\label{LEM:CDOT}
The weighting operation is distributive over its second factor:
\begin{equation}\label{EQ:CDOT-ADD}
w\cdot(\fg+\psi)=w\cdot \fg + w\cdot \psi.
\end{equation}
For charges, it is distributive over the first factor as well:
\begin{equation}\label{EQ:CDOT-CHARGE1}
(u+w)\cdot\alpha=u\cdot\alpha+w\cdot\alpha\qquad(\alpha\in\ba).
\end{equation}
For charges, it is associative in the following sense:
\begin{equation}\label{EQ:CDOT-PROD}
u\cdot(w\cdot\alpha)=(uw)\cdot\alpha\qquad(\alpha\in\ba).
\end{equation}
\end{lemma}

\begin{proof}
Identity \eqref{EQ:CDOT-ADD} follows by \eqref{EQ:WCDOT} immediately. Identity
\eqref{EQ:CDOT-CHARGE1} is immediate by the additivity of $\alpha$. To prove
\eqref{EQ:CDOT-PROD}, first we assume that $\alpha\ge0$ and $u$ and $w$ are
indicator functions:
\begin{align*}
\big(\one_X\cdot(\one_Y\cdot\alpha)\big)(Z)&= \intl_Z \one_X\, d(\one_Y\cdot\alpha) = \intl_J \one_{X\cap Z}\, d(\one_Y\cdot\alpha)\\
&= (\one_Y\cdot\alpha)(X\cap Z) = \intl_{X\cap Z}\one_Y\,d\alpha = \alpha(X\cap Y\cap Z).
\end{align*}
Similarly,
\begin{align*}
\big((\one_X\one_Y)\cdot\alpha \big)(Z)&=\big(\one_{X\cap Y}\cdot\alpha \big)(Z) = \intl_Z \one_{X\cap Y}\,d\alpha)
= \alpha(X\cap Y\cap Z).
\end{align*}
So \eqref{EQ:CDOT-PROD} holds true for indicator functions. Hence by the
distributivity properties \eqref{EQ:CDOT-CHARGE1} and \eqref{EQ:CDOT-ADD}, we
obtain that it holds true if $u$ and $w$ are stepfunctions. Let $u,w$ be
arbitrary functions in $\Bd$ and $\eps>0$, then there are stepfunctions
$f,g\in\Bd$ such that $f\le u\le f+\eps$ and $g\le w \le g+\eps$. Then
(recalling that $\alpha\ge 0$)
\begin{align*}
u\cdot(w\cdot\alpha)&\le (f+\eps)\cdot\big((g+\eps)\cdot\alpha\big)
= f\cdot(g\cdot\alpha) + \eps(g\cdot\alpha) + \eps (f\cdot\alpha) + \eps^2\alpha \\
&= (fg)\cdot\alpha) + \eps(g\cdot\alpha) + \eps (f\cdot\alpha) + \eps^2\alpha\\
&\le (uw)\cdot\alpha) + \eps(w\cdot\alpha) + \eps (u\cdot\alpha) + \eps^2\alpha,
\end{align*}
and by a similar argument
\[
u\cdot(w\cdot\alpha)\ge (uw)\cdot\alpha - \eps(w\cdot\alpha) - \eps (u\cdot\alpha).
\]
Letting $\eps\to 0$ we obtain \eqref{EQ:CDOT-PROD}. Finally, if $\alpha$ is an
arbitrary signed charge, then the identity follows by applying the special case
proved above to $\alpha_+$ and $\alpha_-$, and using the distributivity
\eqref{EQ:CDOT-ADD} once more.
\end{proof}

While identities \eqref{EQ:CDOT-ADD} and \eqref{EQ:CDOT-CHARGE1} do not hold in
general if $\alpha$ is an arbitrary submodular setfunction, we do have the
following simple fact about those:

\begin{lemma}\label{LEM:WEIGHTING}
Let $\fg$ be a bounded submodular setfunction on a set-algebra $(J,\BB)$ with
$\fg(\emptyset)=0$, and let $w\in\Bd_+(\BB)$. Then $w\cdot \fg$ is submodular.
\end{lemma}

\begin{proof}
Using \eqref{EQ:WCDOT},
\begin{align*}
(w\cdot\fg)(X) &+ (w\cdot\fg)(Y) = \int_0^\infty \fg(X\cap \{w\ge t\})+ \fg(Y\cap \{w\ge t\})\,dt\\
&\ge \int_0^\infty \fg((X\cup Y)\cap \{w\ge t\})+ \fg((X\cap Y)\cap \{w\ge t\})\,dt\\
&= (w\cdot\fg)(X\cup Y) + (w\cdot\fg)(X\cap Y).
\end{align*}
\end{proof}

\section{Generalizing some matroid constructions}

Matroid theory has a rich supply of operations on matroids, which lead to new
matroids. Most of the time, formulas for the results of these operations are
nontrivial and have important applications. In this section, we describe
generalizations of three of these basic matroid operations: Strong maps,
Dilworth truncations, and linear extensions.

\subsection{Diverging pairs and strong maps}\label{SSEC:DIVERGE}

The minimum of two submodular setfunctions, or even the minimum of two signed
measures, is not submodular in general. The measures $\fg(X)=\one(0\in
X)-\one(1\in X)$ and $\psi\equiv 0$ provide a counterexample. But we do have
such a conclusion under a monotonicity hypothesis. Let us say that a pair of
setfunctions $(\fg,\psi)$, defined on a set-algebra $(J,\BB)$, is a {\it
diverging pair}, is $\fg-\psi$ is an increasing setfunction. If
$\fg(\emptyset)=\psi(\emptyset)=0$, then this implies that $\psi\le\fg$.

Some easy examples: if $\fg$ is increasing and $\psi$ is decreasing, then
$(\fg,\psi)$ is a diverging pair. If $\fg$ is increasing, then
$(\fg,\min(c,\fg))$ is a diverging pair for every constant $c$. For every
increasing submodular setfunction $\fg$ and $A\in\BB$, the pair $(\fg,\fg_A)$
is diverging; this follows easily by submodularity.

A simple but important property of diverging pairs is the following.

\begin{lemma}\label{LEM:MON-DIFF}
If $(\fg,\psi)$ is a diverging pair of submodular setfunctions, then
$\min(\fg,\psi)$ is submodular.
\end{lemma}

\begin{proof}
Let $\alpha=\min(\fg,\psi)$ and $X,Y\in\BB$. If $\alpha(X)=\fg(X)$ and
$\alpha(Y)=\fg(Y)$, then
\[
\alpha(X\cup Y)+\alpha(X\cap Y) \le \fg(X\cup Y)+\fg(X\cap Y)
\le \fg(X)+\fg(Y) = \alpha(X) + \alpha(Y).
\]
The case when $\alpha(X)=\psi(X)$ and $\alpha(Y)=\psi(Y)$ follows similarly.
Finally, if (say) $\alpha(X)=\fg(X)$ but $\alpha(Y)=\psi(Y)$, then
\begin{align*}
\alpha(X&\cup Y)+\alpha(X\cap Y) \le \psi(X\cup Y)+\fg(X\cap Y)\\
&= \psi(X\cup Y)+\psi(X\cap Y) + [\fg(X\cap Y)-\psi(X\cap Y)]\\
&\le  \psi(X)+\psi(Y) + \fg(X)-\psi(X) = \fg(X)+\psi(Y) = \alpha(X)+\alpha(Y).
\end{align*}
So $\alpha$ is submodular.
\end{proof}

Diverging pairs of increasing submodular setfunctions have a reasonably simple
characterization, generalizing analogous results from matroid theory.
Informally, they can be constructed by extending the first setfunction to a
larger set, and then projecting back to the original set.

\begin{lemma}\label{LEM:DIVERGING}
Let $\fg_1$ and $\fg_2$ be increasing submodular setfunctions on $(J,\BB)$.
Then $(\fg_1,\fg_2)$ is a diverging pair if and only if here is a set-algebra
$(I,\AA)$ and an increasing submodular setfunction $\psi$ on $\AA$ such that
$J\subseteq I$, $\BB=\AA|_J$, $\fg_1(X)=\psi_J(X)$ and $\fg_2(X)=\psi^J(X)$.
\end{lemma}

The proof will show that we could require that $I\setminus J$ is a singleton.
Furthermore, if $(J,\BB)$ is a sigma-algebra, then $(I,\AA)$ can be required to
be a sigma-algebra.

\begin{proof}
({\it The ``if'' part.}) Let $X,Y\in\BB$ with $Y\subseteq X$, then by
submodularity,
\[
\psi(X)+\psi(Y\cup K) \ge \psi(X\cup K) + \psi(Y),
\]
and hence
\begin{align*}
\fg_1(X)-\fg_2(X)&= \psi(X)-\psi(X\cup K)-\psi(K) -\fg_2(\emptyset)\\
&\ge \psi(Y)-\psi(Y\cup K)-\psi(K)-\fg_2(\emptyset)\\
&= \fg_1(Y)-\fg_2(Y).
\end{align*}
So $(\fg_1,\fg_2)$ is a diverging pair.

\medskip

({\it The ``only if'' part.}) Assume that $(\fg_1,\fg_2)$ is a diverging pair.
Create a new element $a$, and a new set-algebra $(I,\AA)$ by
\[
I=J\cup\{a\} \et \AA=\BB\cup\big\{U\cup\{a\}:~U\in\BB\big\}.
\]
Define a setfunction $\psi$ on $(I,\AA)$ by
\[
\psi(X)=
  \begin{cases}
    \fg_1(X), & \text{if $a\notin X$}, \\
    \fg_2(X\setminus a)+\fg_1(J)-\fg_2(J), & \text{if $a\in X$}.
  \end{cases}
\]
Note that $\psi(I)=\fg_1(J)$.

We claim that $\psi$ is submodular: for $X,Y\in\AA$, we have
\begin{equation}\label{EQ:PSI-SUBMOD}
\psi(X)+\psi(Y)\ge \psi(X\cup Y)+\psi(X\cap Y).
\end{equation}
This is trivial if $a\notin X$ and $a\notin Y$ (by the submodularity of
$\fg_1$) and also if $a\in X$ and $a\in Y$ (by the submodularity of $\fg_2$).
Assume that (say) $a\notin X$ but $a\in Y$, which means that $Y=Z\cup \{a\}$
for a set $Z\in\BB$. Then \eqref{EQ:PSI-SUBMOD} is equivalent to
\begin{equation}\label{EQ:PSI-CROSSMOD}
\fg_1(X)+\fg_2(Z) \ge \fg_1(X\cap Z) + \fg_2(X\cup Z).
\end{equation}
By the submodularity of $\fg_1$ and the diverging property of $(\fg_1,\fg_2)$,
we have
\begin{align*}
\fg_1(X)+\fg_2(Z) &\ge  \fg_1(X\cup Z) + \fg_1(X\cap Z) - \fg_1(Z) + \fg_2(Z)\\
&\ge \fg_1(X\cup Z) + \fg_1(X\cap Z) - \fg_1(X\cup Z) + \fg_2(X\cup Z)\\
&= \fg_1(X\cap Z) + \fg_2(X\cup Z),
\end{align*}
proving \eqref{EQ:PSI-CROSSMOD} and the submodularity of $\psi$. To verify that
$\psi$ is increasing, we use lemma \ref{LEM:SUBMOD-MON}: for $X\in\BB$, we have
\[
\psi(I)-\psi(X)=\fg_1(J)-\fg_1(X)\ge 0,
\]
while if $X\in\AA\setminus\BB$, then
\[
\psi(I)-\psi(X)=\fg_1(J) -\fg_2(X\setminus a)-\fg_1(J)+\fg_2(J) = \fg_2(J)-\fg_2(X\setminus a)\ge 0.
\]
The expression for $\fg_1$ and $\fg_2$ in terms of $\psi$ follow by the
definition of $\psi$.

\end{proof}

\begin{remark}\label{REM:S-MAP}
Diverging pairs play an important role in matroid theory. Consider two matroids
$(E,r_1)$ and $(E,r_2)$ on the same underlying set. If $r_1-r_2$ is increasing,
then we say that the identity map $\id_E$, considered as a map from $(E,r_1)$
to $(E,r_2)$ is a {\it strong map} (see e.g.~\cite{Recski}). If we only assume
that $r_1\ge r_2$, then we call $\id_E$ a {\it weak map}. If $(E,r_1)$ is
determined by a multiset of vectors in (say) $\R^d$, then projecting these
vectors to a subspace is a strong map. If $\id_E$ is a strong map, then every
closed set (flat) of $(E,r_2)$ is also closed in $(E,r_1)$. If $\id_E$ is a
weak map, then every independent se of $(E,r_2)$ is also independent in
$(E,r_1)$.

The definition can be extended to maps between matroids on possibly different
underlying sets, but this case is easily reduced to the case described above.
\end{remark}

As an application of the diverging condition, let us describe a construction,
motivated by geometry. Let $\GG$ and $\HH$ be finite families of linear
subspaces in the linear spaces $K$ and $L$, respectively. Let $A\in\GG$ and
$B\in\GG$ with $\dim(A)=\dim(B)=a$. Let us embed $K$ and $L$ into a linear
space $M$ by taking $K\oplus L$, and identifying $A$ and $B$ along a
``generic'' linear bijection $\Gamma:~A\to B$. The families $\GG$ and $\HH$ can
be considered as families of subspaces of $M$, so $\GG\cup\HH$

Due to the ``generic'' nature of $\Gamma$, for $X\in\GG$ and $Y\in\HH$, $X\cap
A$ and $Y\cap B$ will have either $0$ intersection, or they will span $A$ (or
both). Hence the subspace of $M$ spanned by $X\cup Y$ will have dimension
either $\dim(X)+\dim(Y)$, or $\dim(X\cup A)+\dim(Y\cup B)-\dim(A)$, whichever
is smaller.

This formula for generic gluing works for all increasing submodular
setfunctions. Let $\fg$ and $\psi$ be increasing setfunctions on the
set-algebras $(I,\AA)$ and $(J,\BB)$, with $\fg(\emptyset)=\psi(\emptyset)=0$.
Let $a\in I$ and $b\in J$ such that $\fg(a)=\psi(b)$. We define the {\it
splicing of $\fg$ and $\psi$ along $a$ and $b$} as the setfunction defined on
the set-algebra $(I,\AA)\oplus(J,\BB)$ by
\begin{equation}\label{EQ:SPLICING}
\sigma(X\cup Y)=\min\big(\fg(X)+\psi(Y),\,\fg(X\cup\{a\})+\psi(Y\cup\{b\})-\fg(a)\big)
\end{equation}
for $X\in\AA$ and $Y\in\BB$. It is straightforward to check that
$\sigma|_I=\fg$ and $\sigma|_J=\psi$.

\begin{lemma}\label{LEM:BLOW-UP}
If $\fg$ and $\psi$ are increasing submodular setfunctions on the set-algebras
$(I,\AA)$ and $(J,\BB)$ with $\fg(\emptyset)=\psi(\emptyset)=0$, then their
splicing defined by \eqref{EQ:SPLICING} is submodular.
\end{lemma}

\begin{proof}
The setfunctions
\[
\alpha(Z)=\fg(Z\cap I)+\fg(Z\cap J)\quad\text{and}\quad \beta(Z)=\phi((Z\cap I)\cup A)+\psi((Z\cap J)\cup B)-\fg(A)
\]
are submodular on $\AA\oplus\BB$, and $\beta-\alpha$ is increasing. Hence
$\min(\alpha,\beta)$ is submodular by Lemma \ref{LEM:MON-DIFF}.
\end{proof}

Two special cases of splicing are worth mentioning. Let $(I,\AA)$ be a
set-algebra, $A\in\AA$ and $\fg$ an increasing submodular setfunction on
$(I,\AA)$ with $\fg(\emptyset)=0$.

\smallskip

$\bullet$ Let $(J,\BB)$ be a one-point set-algebra, $J=\{b\}$, and let
$\psi(\emptyset)=0$ and $\psi(J)=\fg(A)$. Then $(I,\AA)\oplus(J,\BB)$ endowed
with $\sigma$ can be thought of representing the subset $A$ by a new node $b$.
If $A=\{a\}$ is a singleton, them $a$ and $b$ are ``twins''. If $A=I$, then $b$
is a new ``container'' element, i.e., a point such that
$\sigma(X)=\sigma(\{b\})$ if $b\in X$ and $\sigma(X)\le\sigma(\{b\})$ if
$b\notin X$.

\smallskip

$\bullet$ Let $A=\{a\}$ be a singleton set, $J=B=\{b,c\}$, and
$\psi(\{b\})=\psi(\{b\})=\fg(A)$. In $((I,\AA)\oplus(J,\BB),\sigma)$, $c$ can
be thought of as a new ``generic'' point below $a$ with $\sigma(c)=\psi(c)$.
This means that adding $c$ to a set $X\in\AA$ increases $\sigma(X)$ by as much
as possible, namely by $\min(\psi(c),\fg(X\cup A)-\fg(X)$.

\subsection{Positive part and Dilworth truncation}

For a setfunction $\fg:~\BB\to\R$ defined on a set-algebra $(J,\BB)$, and for
$U\in\BB$, define
\begin{equation}\label{EQ:FGSTAR}
\fg_{\circ}(U)=\inf_\XX \sum_{X\in\XX} |\fg(X)|_+,
\end{equation}
where $\XX$ ranges over all partitions of $U$ into a finite number of sets in
$\BB$. We define $\fg_{\circ}(\emptyset)=|\fg(\emptyset)|_+$ (note that we have
not assumed that $\fg(\emptyset)=0$). Clearly $\fg_{\circ}\ge0$, and
$\fg_{\circ}(U)\le|\fg(U)|_+$. If $\fg$ is increasing, then so is
$\fg_{\circ}$. We may allow in $\XX$ the empty set, or not, without changing
$\fg_{\circ}$.

\begin{lemma}\label{LEM:DILW}
If $\fg$ is an increasing submodular setfunction, then $\fg_{\circ}$ is
submodular.
\end{lemma}

\begin{proof}
If $\fg(\emptyset)\ge 0$, then $\fg$ is subadditive, and hence
$\fg_{\circ}=\fg$. So we may assume that $\fg(\emptyset)<0$. Then
$\fg_{\circ}(\emptyset)=0$.

Let $U,V\in\BB$, $\eps>0$, and let $\XX$ and $\YY$ be finite partitions of $U$
and $V$, respectively, into nonempty sets in $\BB$ such that
\[
\sum_{X\in\XX} |\fg(X)|_+ < \fg_{\circ}(U)+\eps,\et  \sum_{Y\in\YY} |\fg(Y)|_+ < \fg_{\circ}(V)+\eps.
\]
Let $\RR$ be the (finite) set-algebra generated by $\XX\cup\YY$, let $\AA$ be
the set of its atoms, and let
\[
\AA_0=\{X\in\AA:~\phi(X)<0\}, \et \AA_1=\{X\in\AA:~\phi(X)\ge 0\}.
\]
Let $J_i=\cup\AA_i$ and $\RR_1=\RR|_{J_1}$. For $X\in\RR$, let $X'=X\cap J_1$.

If an atom $A\in\AA_0$ is a proper subset of $X\in\XX$, then splitting $X$ into
$A$ and $X\setminus A$ can only decrease the sum $\sum_{X\in\XX} |\fg(X)|_+$.
So we may assume that every $X\in\XX$ is either an atom in $\AA_0$, or a union
of atoms in $\AA_1$, and similarly for $\YY$.

On the set-algebra $\RR_1$, the setfunction $\fg$ is nonnegative on nonempty
sets.
Let $\XX'=\XX\cap\RR_1$, $\YY'=\YY\cap\RR_1$, $U'=\cup\XX'$ and $V'=\cup\YY'$.
Let $\ZZ_0=\XX'\cup\YY'$, considered as a multiset. Then $\ZZ_0$ covers every
point in $U'\cap V'$ exactly twice, and every other point in $U'\cup V'$
exactly once, and
\[
\fg_{\circ}(U)+\fg_{\circ}(V) \ge \sum_{Z\in\ZZ_0} \fg(Z)-2\eps.
\]
Let $Z_1$ and $Z_2$ be two sets in $\ZZ_0$ such that $Z_1\not\subseteq Z_2$,
$Z_2\not\subseteq Z_1$, and $Z_1\cap Z_2\not=\emptyset$. Deleting $Z_1$ and
$Z_2$ from $\ZZ_0$, but adding $Z_1\cup Z_2$ and $Z_1\cap Z_2$, we get a
multiset $\ZZ_1$ of nonempty sets in $\RR_1$, which covers every point in
$U'\cap V'$ exactly twice, and every other point in $U'\cup V'$ exactly once.
By submodularity, we have $\sum_{Z\in\ZZ_1} \fg(Z)\le \sum_{Z\in\ZZ_0} \fg(Z)$.

We repeat this procedure, to get a sequence of multisets $\ZZ_1,\ZZ_2,\ldots$.
In a finite number of steps we must get stuck, since $\RR_1$ is finite, each
set occurs at most twice in $\ZZ_i$, and $\sum_{Z\in\ZZ_i}|Z|^2$ is increasing
all the time. So we end up with a multiset $\ZZ_n$ of nonempty sets in $\RR_1$,
in which every pair of sets $Z_1,Z_2\in\ZZ_n$ is either disjoint or comparable,
$\ZZ_n$ covers every point in $U'\cap V'$ exactly twice, and every other point
in $U'\cup V'$ exactly once, and
\[
\fg_{\circ}(U)+\fg_{\circ}(V) \ge \sum_{Z\in\ZZ_n} \fg(Z)-2\eps.
\]
Let $\ZZ_n'$ be the family of maximal sets in $\ZZ_n$, and $\ZZ_n''$, the rest.
(If a set occurs twice in $\ZZ_n$, then we put one copy in $\ZZ_n'$, and one
copy in $\ZZ_n''$.) It is easy to see that $\cup\ZZ_n'=U'\cup V'$,
$\cup\ZZ_n''=U'\cap V'$, and the families $\ZZ_n'$ and $\ZZ_n''$ consist of
disjoint sets. Hence $\ZZ_n'$, together with some atoms in $\AA_0$, forms a
partition of $U\cup V$, and so
\[
\fg_{\circ}(U\cup V)\le \sum_{X\in \ZZ_n'} \fg(X),
\]
and similarly,
\[
\fg_{\circ}(U\cap V)\le \sum_{X\in \ZZ_n''} \fg(X).
\]
Hence
\[
\fg_{\circ}(U)+\fg_{\circ}(V) \ge \sum_{Z\in\ZZ_n} \fg(Z)-2\eps \ge \fg_{\circ}(U\cup V)+\fg_{\circ}(U\cap V)-2\eps.
\]
Since this holds for every $\eps>0$, this proves the lemma.
\end{proof}

The last proof is an application of the ``uncrossing method'' of combinatorial
optimization; see \cite{Schr}, Theorem 48.2, and also the proof of Theorem
\ref{THM:SUBMOD-UNCROSS2}.

If $\psi$ is a submodular setfunction such that
\begin{equation}\label{EQ:PSI-TRUNC}
0\le \psi(X) \le |\fg(X)|_+\qquad(X\in\BB),
\end{equation}
then $\psi\le \fg_{\circ}$. Indeed, $\psi$ is nonnegative, and hence
subadditive, so for any $U\in\BB$, and any finite partition $\XX$ of $U$ into
sets in $\BB$, we have
\[
\psi(U)\le \sum_{X\in\XX} \psi(X) \le  \sum_{X\in\XX} |\fg(X)|_+,
\]
and taking the infimum  on the right, we get that $\psi(U)\le\fg_{\circ}(U)$.
Thus $\fg_{\circ}$ is the unique largest submodular setfunction with properties
\eqref{EQ:PSI-TRUNC}. This justifies calling it the {\it positive part} of
$\fg$. (Note: this is different from the ``positive part'' $\alpha_+$ of a
charge.)

\begin{example}\label{EXA:EDGE2GRAPH}
Let $G=(V,E)$ be a finite graph, and define $\fg(X)=|\cup X|$ for $X\subseteq
E$. Then $\fg$ is submodular (a finite version of Example
\ref{EXA:TRANSVERSAL}). The setfunction $\fg-1$ is negative on $\emptyset$ and
$1$ on singletons. Its positive part $(\fg-1)_{\circ}$ is the rank function of
the cycle matroid of $G$.

More generally, let $(E,r)$ be a matroid, and let $E_2$ be the set of its flats
of rank $2$. For $X\subseteq E_2$, let $\rho(X) = r(\cup X)$. Then $\rho-1$ is
an increasing submodular setfunction, with $\rho(\emptyset)=-1$ and $\rho(x)=1$
for $x\in E_2$. Thus $(E_2,(\rho-1)_{\circ})$ is a matroid, which is called the
{\it Dilworth truncation} of $(E,r)$. For the trivial matroid in which
$r(X)=|X|$, its Dilworth truncation is the cycle matroid of the complete graph.
\end{example}

\begin{example}\label{EXA:INTERSECT}
Consider a finite family $\FF$ of linear subspaces in $\R^d$, along with a
``generic'' $d-k$-dimensional subspace $K$, and family of intersections
$\FF'=\{A\cap K:~A\in \FF\}$. If $(\FF,\rho)$ is the polymatroid defined by
$\FF$, then the polymatroid defined by $\FF'$ is $(\FF',(\rho-k)_{\circ})$.
(The proof of this fact is somewhat involved; see e.g.~\cite{GeomBook}, Lemma
19.20.)
\end{example}

\subsection{Convexity and uncrossing}\label{SEC:UNCROSS}

In this section, we describe an application of a basic technique in
combinatorial optimization called ``uncrossing'', generalizing some results in
\cite{Sipos,LL83,Gal}.

We start with an easy observation: for every submodular setfunction $\fg$ on a
set-algebra $(J,\BB)$, the functional $\wh\fg$ is submodular on the lattice of
functions in $\Bd(\BB)$. Defining
\[
(f\lor g)(x)=\max\{f(x),g(x)\} \et (f\land g)(x)=\min\{f(x),g(x)\},
\]
we have the following inequality:
\begin{equation}\label{EQ:FUNC-LAT}
\wh\fg(f\lor g)+\wh\fg(f\land g) \le \wh\fg(f)+\wh\fg(g)
\end{equation}
Indeed, by \eqref{EQ:WH-DEF} we may assume that $f,g\ge 0$. The identities
$\{f\lor g \ge t\} = \{f\ge t\}\cup \{g\ge t\}$ and $\{f\land g \ge t\} =
\{f\ge t\}\cap \{g\ge t\}$ imply that
\begin{align*}
\wh\fg(f\lor g)+\wh\fg(f\land g) &= \int_0^\infty \fg\{f\lor g\ge t\}\,dt
+ \int_0^\infty \fg\{f\land g\ge t\}\,dt\\
&=\int_0^\infty \fg\{f\lor g\ge t\}+\fg\{f\land g\ge t\}\,dt\\
&\le \int_0^\infty \fg\{f\ge t\}+\fg\{g\ge t\}\,dt= \wh\fg(f)+\wh\fg(g).
\end{align*}

A less trivial inequality asserts that $\wh\fg$ is a subadditive (or,
equivalently, convex) functional. For the increasing case, this was proved in
\cite{Sipos}; see also \cite{Denn}, Chapter 6. Two functions $f,g:~I\to\R$
defined on the same set $I$ are called {\it comonotonic}, if
$(f(x)-f(y))(g(x)-g(y))\ge 0$ for all $x,y\in I$. This is equivalent to saying
that all level sets $\{f\ge t\}$ and $g\ge s$ are comparable.

\begin{theorem}\label{THM:SUBMOD-UNCROSS2}
Let $\fg$ be a bounded submodular setfunction on a set-algebra $(J,\BB)$ with
$\fg(\emptyset)=0$, and let $f,g\in\Bd$. Then
\[
\wh\fg(f+g) \le  \wh\fg(f) + \wh\fg(g).
\]
If $f$ and $g$ are comonotonic, then equality holds.
\end{theorem}

Of course, the inequality generalizes to the sum of any finite number of
functions in $\Bd$. It will be useful to state and prove the following special
case first:

\begin{lemma}\label{LEM:SUBMOD-UNCROSS0}
Let $\fg$ be a bounded submodular setfunction on a set-algebra $(J,\BB)$ with
$\fg(\emptyset)=0$, let $H_1,\ldots,H_n\in\BB$, $a_1,\ldots a_n\in\R_+$, and
$h=\sum_i a_i\one_{H_i}$. Then
\[
\wh\fg(h) \le \sum_{i=1}^n a_i \fg(H_i).
\]
If the sets $H_i$ form a chain, then equality holds.
\end{lemma}

\begin{proof}
Let $\HH$ be the (finite) set-algebra generated by the sets $H_i$. Then we can
express $h$ as
\begin{equation}\label{EQ:FBF}
h(x) = \sum_{H\in \HH} b_H \one_H(x)\qquad(b_H\ge 0)
\end{equation}
in several ways. Let $\HH_b$ be the family of sets $H\in\HH$ with $b_H>0$.
Starting with the representation defining $h$ (where $b_{H_i}=a_i$), we modify
it repeatedly as follows. Suppose that we find two sets $H_1,H_2\in\HH_b$ such
that neither one of them contains the other. Let (say) $b_{H_1}\le b_{H_2}$.
For $H\in\HH$, define
\[
b'_H=
  \begin{cases}
    b_{H_1}+b_{H_1\cup H_2}, & \text{if $H=H_1\cup H_2$}, \\
    b_{H_1}+b_{H_1\cap H_2}, & \text{if $H=H_1\cap H_2$},\\
    0, & \text{if $H=H_1$},\\
    b_{H_2}-b_{H_1}, & \text{if $H=H_2$},\\
    b_H, & \text{otherwise}.
  \end{cases}
\]
(Informally, we replace as much as possible of $H_1$ and $H_2$ by $H_1\cup H_2$
and $H_1\cap H_2$.) Let $\HH'=\{H\in\HH:~b'_H>0\}$. Then it is easy to verify
that
\[
\sum_{H\in\HH'} b'_H\one_H(x) = h(x),
\]
and
\[
\sum_{H\in\HH'} b'_H\fg(H) \le \sum_{H\in\HH} b_H\fg(H)
\]
by submodularity.

At each step of the above procedure we have removed one set from $\HH_b$, but
may have added two. Nevertheless, it was proved in \cite{HLST} that repeating
this procedure a finite number of times (choosing $H_1$ and $H_2$
appropriately), we get a representation $b$ for which the sets in $\HH_b$ form
a chain. Then \eqref{EQ:FBF} is just the ``Layer Cake Representation'' of $h$,
showing that
\[
\wh\fg(h) = \sum_{H\in\HH} b_H\fg(H) \le \sum_i a_i\fg(H_i).
\]
This proves the lemma.
\end{proof}

\begin{proof*}{Theorem \ref{THM:SUBMOD-UNCROSS2}}
In the special case when $f$ and $g$ are nonnegative stepfunctions, we express
them in their layer cake representation, and apply Lemma
\ref{LEM:SUBMOD-UNCROSS0} to get the inequality as stated. The more general
case of nonnegative $f$ and $g$ follows via a routine approximation by
stepfunctions. Let $\eps>0$, and let $f',g'$ be stepfunctions such that $f'\le
f\le f'+\eps$ and $g'\le g\le g'+\eps$. Using Lemma \ref{LEM:WHPHI-CONT}, we
get
\[
\wh\fg(f+g) \le \wh\fg(f'+g')+ 2 c_\fg \eps \le \wh\fg(f')+ \wh\fg(g')+ 2c_\fg\eps
\le \wh\fg(f)+ \wh\fg(g)+ 4c_\fg\eps
\]
Since $\eps>0$ is arbitrary, it follows that $\wh\fg(f+g) \le  \wh\fg(f) +
\wh\fg(g)$.

If $f$ and $g$ may have negative values, then for any sufficiently large $C>0$,
using \eqref{EQ:WH-DEF},
\begin{align*}
\wh\fg(f+g)&= \wh\fg(f+g+2C)-2C\fg(J) \le \wh\fg(f+C)+\wh\fg(g+C) -2C\fg(J)\\
&= \wh\fg(f)+\wh\fg(g).
\end{align*}
This completes the proof of the first assertion.

The second one follows for stepfunctions from the observation that for two
comonotonic stepfunctions, their level sets form a single chain, and hence
their layer cake representations can be added up. The general case follows by a
routine approximation argument as above.
\end{proof*}

This theorem easily implies the following fact:

\begin{corollary}\label{COR:SUBMOD-CONV}
Let $\fg$ be a setfunction with bounded variation on a set-algebra. Then
$\wh\fg$ is a convex functional if and only if $\fg$ is submodular.
\end{corollary}

\begin{proof}
The ``if'' part follows immediately by the homogeneity of $\wh\fg$ and Theorem
\ref{THM:SUBMOD-UNCROSS2}. To prove the ``only if'' part, suppose that $\wh\fg$
is convex. Since it is also homogeneous, we have
\[
\fg(S\cup T)+\fg(S\cap T) =
\wh\fg(\one_S+\one_T) \le \wh\fg(\one_S)+\wh\fg(\one_T) = \fg(S)+\fg(T).
\]
proving that $\fg$ is submodular.
\end{proof}

\section{Minorizing and separating measures}\label{SEC:MATROID}

Our goal in this section is to generalize (as far as possible) the notions of
bases and independence polytopes from matroid theory. Matroid independent sets
will correspond (in some sense) to charges minorizing the given submodular
setfunction. Such minorizing measures have been considered in the
measure-theoretic literature on submodularity as well, so here the interaction
between the two areas is quite rich.

\subsection{Separation by modular setfunctions}

Using the convexity of the functional $\wh\fg$ (Corollary
\ref{COR:SUBMOD-CONV}), it is easy to prove the following separation theorem.

\begin{theorem}\label{THM:SEPARATE}
Let $\fg$ be a submodular setfunction on a set-algebra $(J,\BB)$, and let $\xi$
be a supermodular setfunction on the same set-algebra. Assume that $\xi\le\fg$,
then there is a modular setfunction $\mu$ on $(J,\BB)$ such that
$\xi\le\mu\le\fg$. More generally, given $a,b\in\R$ such that
$\xi(\emptyset)\le a\le\fg(\emptyset)$ and $\xi(J)\le b\le\fg(J)$, there is a
modular setfunction $\mu$ on $(J,\BB)$ such that $\xi\le\mu\le\fg$,
$\mu(\emptyset)=a$ and $\mu(J)=b$.
\end{theorem}

\begin{proof}
We may assume that $a=0$ by subtracting $a$ from $\fg$, $\xi$ and $b$. We may
also assume that $\fg(\emptyset)=\xi(\emptyset)=0$ by lowering resp.\ raising
the original values of $\fg(\emptyset)$ and $\xi(\emptyset)$. Similarly, we may
assume that $\fg(J)=\xi(J)=b$. Consider the Banach space $\Bd\oplus\Bd$, and
the functional $\Phi(f,g)=\wh\fg(f)-\wh\xi(g)$. Clearly $\Phi$ is bounded,
convex, and on the linear subspace $L=\{(f,f):~f\in\Bd\}$ it is nonnegative.
Hence by the Hahn--Banach Theorem, there is a continuous linear functional
$\Lambda:~\Bd\oplus\Bd\to\R$ such that $\Lambda|_L=0$ and $\Lambda\le\Phi$. The
first condition implies that $\Lambda(f,g)=-\Lambda(g,f)$.

Consider the linear functional $\Lambda_1:~\Bd\to\R$ defined by
$\Lambda_1(f)=\Lambda(f,0)$. Since $\Bd^*=\ba$, there is an $\mu\in\ba(\Bd)$
such that $\Lambda_1(f)=\mu(f)$ for $f\in\Bd$. Then
\[
\wh\fg(f) =\Phi(f,0) \ge \Lambda(f,0) = \mu(f),
\]
and similarly $\wh\xi(f) \le \mu(f)$. Trivially $\mu(\emptyset)=0$ and
$\mu(J)=b$.
\end{proof}

\begin{remark}\label{REM:FRANK}
The finite version of Theorem \ref{THM:SEPARATE} has the (substantially deeper)
supplement, due to Frank \cite{Frank}, that if $\xi$ and $\fg$ are integer
valued, then there is an integer valued modular setfunction $\mu$ separating
them. A generalization to the infinite case would be interesting to obtain.
\end{remark}

\subsection{Minorizing and the Greedy Algorithm}\label{SSEC:MINORIZE}

Given a submodular setfunction $\fg$, we say that a charge $\alpha\in\ba(\BB)$
is {\it minorizing $\fg$}, if $\alpha\le\fg$. A minorizing charge $\alpha$ is
called {\it basic}, if $\alpha(J)=\fg(J)$. Minorizing charges and basic
minorizing charges form convex sets $\matp(\fg) \subseteq \ba$ and
$\basp(\fg)\subseteq \ba$. Often we are interested in nonnegative charges, and
study the sets $\matp_+(\fg)=\matp(\fg)\cap\ba_+$ and
$\basp_+(\fg)=\basp(\fg)\cap\ba_+$.

It will be useful to notice that charges $\alpha\in\basp(\fg)$ cannot be too
wild: from above we have $\alpha\le\fg$, while from below,
\[
\alpha(X)=\alpha(J)-\alpha(X^c) = \fg(J)-\alpha(X^c) \ge \fg(J)-\fg(X^c).
\]
This implies that $\|\alpha\|\le 2\|\fg\|$.

Minorizing nonnegative charges will be natural generalizations of fractional
independent sets (convex combinations of indicator functions of independent
sets) of a matroid. Fractional bases of a matroid correspond to basic
minorizing charges.

Theorem \ref{THM:SEPARATE} has some corollaries about minorizing measures.

\begin{corollary}\label{COR:FIN-ADDX}
Let $\fg$ be a bounded submodular setfunction with $\fg(\emptyset)=0$. {\rm(a)}
The set $\basp(\fg)$ is nonempty. {\rm(b)} For every $\beta\in\matp(\fg)$ there
is a charge $\alpha\in\basp(\fg)$ such that $\alpha\ge \beta$.
\end{corollary}

In particular, part (b) can be applied with $\beta=0$ if $\fg$ is nonnegative,
to get that $\basp_+(\fg)$ is nonempty.

\begin{proof}
(a) Let $K\le 0$ be a lower bound on $\fg$, and consider setfunction
\[
\psi(X)=
  \begin{cases}
    K, & \text{if $X\notin\{\emptyset, J\}$},\\
    0, & \text{if $X=\emptyset$},\\
    \fg(J), & \text{if $X=J$}.
  \end{cases}
\]
Clearly $\psi$ is supermodular, and $\psi\le\fg$. Hence by Theorem
\ref{THM:SEPARATE}, there is a modular setfunction $\alpha$ such that
$\psi\le\alpha\le\fg$. It follows that $\alpha(\emptyset)=0$, so $\alpha$ is a
charge, and trivially $0\le\alpha\le\fg$ and $\alpha(J)=\fg(J)$.

The proof of (b) is essentially the same, replacing the definition of $\psi$ by
\[
\psi(X)=
  \begin{cases}
    \beta(X), & \text{if $X\not=\emptyset$},\\
    \fg(J), & \text{if $X=J$}.
  \end{cases}
\]
\end{proof}

Corollary \ref{COR:FIN-ADDX} says nothing if $\fg\ge0$ and $\fg(J)=0$ (we can
take $\alpha=0$), but we can prove a related result for this case.

\begin{lemma}\label{LEM:FIN-ADD1}
Let $\fg$ be a submodular setfunction with $\fg(\emptyset)=\fg(J)=0$. Then for
every $A\in\BB$ there is an $\alpha\in\ba(\BB)$ such that $-\fg^c \le \alpha
\le \fg$ and $\alpha(A)=\fg(A)$.
\end{lemma}

Note that $\fg(X)+\fg(X^c)\ge \fg(\emptyset)+\fg(J)=0$ implies that
$-\fg^c\le\fg$, so the lower and upper bounds on $\alpha$ are not
contradictory. If $\fg\not=0$, then applying the assertion to any $A$ with
$\fg(A)\not=0$, we obtain a nonzero charge $\alpha$.

\begin{proof}
Let $\mu_1$ be a charge on $\BB|_A=\{X\in\BB:~X\subseteq A\}$ such that
$\mu_1\le \fg$ and $\mu_1(A)=\fg(A)$. Let $\mu_2$ be a charge on $\BB|_{A^c}$
such that
$\mu_2\le \fg^c$ and $\mu_2(A^c)=\fg^c(A^c)=\fg(A)$. Such measures $\mu_1$ and
$\mu_2$ exist by Corollary \ref{COR:FIN-ADDX}. Define
\[
\alpha(X)=\mu_1(A\cap X) - \mu_2(A^c\cap X).
\]
Clearly $\alpha$ is a charge, which satisfies
\[
\alpha(\emptyset)=\alpha(J)=0,\qquad \alpha(A)=-\alpha(A^c)=\fg(A)\not=0.
\]
Let $X\in\BB$, then
\begin{align*}
\alpha(X)&= \mu_1(A\cap X)-\mu_2(A^c\cap X) = \mu_1(A\cap X)-\mu_2(A^c) + \mu_2(A^c\cap X^c)\\
&\le \fg(A\cap X) - \fg(A) + \fg^c(A^c\cap X^c) = \fg(A\cap X) - \fg(A) + \fg(A\cup X)\\
&\le \fg(X)
\end{align*}
by submodularity. The lower bound on $\alpha$ follows easily:
\begin{align*}
\alpha(X)&=-\alpha(X^c)\ge -\fg(X^c)=-\fg^c(X).
\end{align*}
\end{proof}

One can prove substantially more than Corollary \ref{COR:FIN-ADDX}: A
minorizing measure can be required to agree with $\fg$ on many more sets.

\begin{theorem}\label{THM:FIN-ADD-LAT}
Let $(J,\BB)$ be a set-algebra, let $\fg$ be a bounded submodular setfunction
on $\BB$ with $\fg(\emptyset)=0$, and let $\LL\subseteq\BB$ be a lattice family
containing $\emptyset$ and $J$. Assume that $\fg$ is modular on $\LL$. Then
there is an $\alpha\in\basp(\fg)$ such that $\alpha(S)=\fg(S)$ for $S\in\LL$.
If $\fg$ is increasing, then we can require that $\alpha\in \basp_+$.
\end{theorem}

\begin{proof}
Let $\NN$ be the set-algebra generated by $\LL$. By Proposition
\ref{PROP:EXTEND}, there is a charge $\psi$ on $\NN$ such that $\psi=\fg$ on
$\LL$. (Note that $\psi\not=\fg$ on $\NN\setminus\LL$ in general.) By Lemma
\ref{LEM:FF-LINEAR}, the set $\Bd(\LL)$ of bounded $\LL$-measurable functions
is a linear space. For $f\in \Bd(\LL)$, using Lemma \ref{LEM:SET-ALG-INT}, we
have
\[
\wh\fg(f) = \intl_0^\infty \fg\{f\ge t\} \,dt
= \intl_0^\infty \psi\{f\ge t\} \,dt = \intl_J f \,d\psi,
\]
and the integral on the right is a linear functional. This implies that the
functional $\wh\fg$ is linear on $\Bd(\LL)$.

The Hahn--Banach Theorem applies, and gives a continuous linear functional
$\Lambda$ on $\Bd(\BB)$ such that $\Lambda\le\wh\fg$, and $\Lambda=\wh\fg$
on $\Bd(\LL)$.

By Proposition \ref{PROP:DUAL-BA} in the Appendix, the functional $\Lambda$ can
be represented by a charge $\alpha$ in the sense that
\[
\Lambda(h)=\intl_J h \,d\alpha=\wh\alpha(h)
\]
for all $h\in\Bd(\BB)$. In particular, we have
\[
\alpha(X)=\intl \one_X \,d\alpha = \Lambda(\one_X)\le \wh\fg(\one_X)=\fg(X)
\]
for $X\in\BB$, with equality for $X\in\LL$. Since $J\in\LL$, we have
$\alpha(J)=\fg(J)$.

If $\fg$ is increasing, then nonnegativity of $\alpha$ follows by a simple
computation:
\[
\alpha(X) = \alpha(J)-\alpha(X^c) \ge \fg(J)-\fg(X^c)\ge 0\qquad(X\in\BB)
\]
by the monotonicity of $\fg$.
\end{proof}

An important special case is the following. A family of sets $\SS\subseteq\BB$
totally ordered by inclusion will be called a {\it chain} (in $\BB$). A chain
that is maximal will be called a {\it full chain} (in $\BB$). Clearly every
full chain contains $\emptyset$ and $J$.

If $\BB$ is the family of Borel subsets of $(0,1]$, then initial intervals
(open or closed from the right) form a full chain in $\BB$. Let us call this
the {\it standard full chain}. A non-standard example of a full chain of Borel
sets is the following: Assume the Continuum Hypothesis, and order the points in
$J$ to the initial ordinal of $\aleph_1$. Then all beginning sections are
countable, and hence Borel, and they form a full chain.

Since a chain is trivially a lattice, and every function on a chain is modular,
we have the following corollary \cite{Choq}, \cite{Sipos}, \cite{Denn}.

\begin{corollary}\label{COR:FIN-ADD-LAT}
Let $\fg$ be a submodular setfunction on a set-algebra $(J,\BB)$ with
$\fg(\emptyset)=0$, and let $\SS\subseteq\BB$ be a full chain of sets. Then
there is a charge $\alpha\le\fg$ such that $\alpha(S)=\fg(S)$ for $S\in\SS$. If
$\fg$ is increasing, then $\alpha\ge 0$.
\end{corollary}

In a sense, this corollary is a generalization of the Greedy Algorithm in
combinatorial optimization. In the ordering of the points defining a full chain
$\SS$, every beginning section $S$ has as large an $\alpha$-measure as it
possibly can, namely $\fg(S)$. We can think of building up $\alpha$ in
infinitesimally small steps starting from $\emptyset$, adding always as much to
it as possible.

\begin{lemma}\label{LEM:PHI-ALPHA}
Let $\fg$ and $\psi$ be increasing setfunctions on a set-algebra $(J,\BB)$ with
$\fg(\emptyset)=\psi(\emptyset)=0$, and let $\SS\subseteq\BB$ be a chain of
sets. Assume that $\fg(S)=\psi(S)$ for every $S\in\SS$. Then
$\wh\fg(f)=\wh\psi(f)$ for every $\SS$-measurable function $f\in\Bd$.
\end{lemma}

\begin{proof}
If $(J,\BB)$ is a sigma-algebra, then this lemma is trivial by the integral
formula defining $\wh\fg$ and $\wh\psi$. In the general case, however, we have
to work with the more involved definition of $\SS$-measurability. It suffices
to consider that case when $f\ge 0$, because \eqref{EQ:WH-DEF} implies then the
general case. Let $t,\eps\ge 0$. Since $f$ is $\SS$-measurable, there is a set
$S\in\SS$ such that $\{f\ge t\}\subseteq S\subseteq \{f\ge t-\eps\}$. Then
\[
\fg^\ui\{f\ge t\} \le \fg(S) = \psi(S)
\le \psi^\ls\{f\ge t-\eps\},
\]
and so
\[
\wh{\fg^\ui}(f) \le \wh{\psi^\ls}(f+\eps) = \wh{\psi^\ls}(f)+\eps\psi(J).
\]
Since $\eps>0$ is arbitrary, this implies that $\wh{\fg^\ui}(f)
\le\wh{\psi^\ls}(f)$. Combining with the trivial inequality
$\psi^\ls\le\psi^\ui$, we get that $\wh{\fg^\ui}(f) \le\wh{\psi^\ui}(f)$. The
reverse inequality follows by the same argument, and so Lemma
\ref{LEM:SET-ALG-INT} completes the proof.
\end{proof}

In general, we cannot require in Theorem \ref{THM:FIN-ADD-LAT}, or even in
Corollary \ref{COR:FIN-ADD-LAT}, that the charge $\alpha$ be countably
additive; for example, the setfunction $\fg$ can be any finitely but not
countably additive nonnegative measure, then the only choice for $\alpha$ is
$\alpha=\fg$ (see, however, Corollary \ref{COR:FIN-ADD2}). Nor can we claim
that this charge $\alpha$ is unique, as the following example shows.

\begin{example}\label{EXA:NON-UNIQUE}
Let $(J,\BB)=(0,1]$, and consider the submodular setfunction
$\overline{\lambda}(X)=\lambda(\overline{X})$, as defined in Example
\ref{EXA:CLOSURE}. Let $\SS$ be the standard full chain. Trivially,
$\lambda(S)=\overline{\lambda}(S)$ for $S\in\SS$, and
$\lambda\le\overline{\lambda}$. Thus $\alpha=\lambda$ satisfies the conclusion
of Theorem \ref{THM:FIN-ADD-LAT}.

To construct another solution, let $Q=\Qbb\cap(0,1]$ and
$Q_n=\{\frac1n,\frac2n,\ldots,\frac nn\}$ ($n=1,2,\ldots$). We use the Banach
limit $\blim$. For $X\in\BB$, define
\[
\alpha(X)= \blim_n\, \frac1n|Q_n\cap X|.
\]
It is easy to see that $\alpha$ is nonnegative and finitely additive. If $X$ is
an interval, then the sequence $|Q_n\cap X|/n$ is convergent, and it converges
to $\lambda(X)$, so $\alpha(X)=\lambda(X)$ for intervals; in particular,
$\alpha(S)=\lambda(S)=\overline{\lambda}(S)$ for $S\in\SS$.

To verify that $\alpha\le\overline{\lambda}$, notice that for any $X\in\BB$,
the set $J\setminus \overline{X}$ is the union of countably many disjoint open
intervals: $J\setminus \overline{X}=I_1\cup I_2\cup\ldots$. Hence for any $n\ge
1$,
\[
\alpha(X)\le \alpha(\overline{X})\le \alpha(J\setminus I_1\setminus\ldots\setminus I_n)
= 1-\sum_{j=1}^n \alpha(I_j) = 1-\sum_{j=1}^n \lambda(I_j).
\]
Letting $n\to\infty$, the right hand side tends to
$\lambda(\overline{X})=\overline{\lambda}(X)$, proving that
$\alpha\le\overline{\lambda}$.

It is clear that $\alpha\not=\lambda$: the set $Q$ is countable, so
$\lambda(Q)=0$, but $\alpha(Q)=1$.
\end{example}

\begin{corollary}\label{COR:FINADD-EXTEND}
Let $(J,\BB)$ be a sigma-algebra, and let $\LL\subseteq\BB$ be a
lattice family containing $\emptyset$ and $J$. Let $\alpha:~\LL\to\R$ be an
increasing modular setfunction with $\alpha(\emptyset)=0$. Then there is a
charge $\beta\in\ba_+(J,\BB)$ such that $\beta(S)=\alpha(S)$ for $S\in\LL$.
\end{corollary}

\begin{proof}
By Lemma \ref{LEM:STAR}, the setfunction $\alpha^\li:~2^J\to\R$ is an extension
of $\alpha$, increasing, and submodular; we consider it only on $\BB$. By
Theorem \ref{THM:FIN-ADD-LAT}, there is a measure $\beta\in\ba_+(\BB)$
extending $\alpha$.
\end{proof}

As a further application of Theorem \ref{THM:FIN-ADD-LAT}, we prove following
theorem; for increasing submodular setfunction this was proved by \v{S}ipo\v{s}
\cite{Sipos} (see also \cite{Denn}, Chapter 10):

\begin{corollary}\label{COR:SIPOS}
Let $\fg$ be a submodular setfunction on a set-algebra $(J,\BB)$ with
$\fg(\emptyset)=0$, and let $w\in\Bd_+$. Then
\[
\max_{\alpha\in\matp(\fg)} \wh\alpha(w)  = \wh\fg(w).
\]
In particular, for every $X\in\BB$, we have
\[
\max_{\alpha\in\matp(\fg)} \alpha(X) = \fg(X).
\]
If $\fg$ is increasing, then $\alpha$ can be required to be nonnegative.
\end{corollary}

\begin{proof}
If $\beta\in\matp(\fg)$, then $\beta\le\fg$ and hence
$\wh\beta(w)\le\wh\fg(w)$. On the other hand, let $\SS$ be the chain of level
sets $\{w\ge t\}$. Then Theorem \ref{THM:FIN-ADD-LAT} yields a charge
$\alpha\le\fg$ such that $\alpha(S)=\fg(S)$ for every $S\in\SS$, and so
$\wh\alpha(w)=\wh\fg(w)$. In the increasing case, \ref{THM:FIN-ADD-LAT} yields
a nonnegative charge $\alpha$.
\end{proof}

\subsection{Exchange properties}\label{SSEC:EXCHANGE}

We can generalize some fundamental properties of maximal fractional independent
sets (bases) of a matroid: they have the same cardinality (equal to the rank of
the underlying set), and they satisfy the Steinitz Exchange Property. Corollary
\ref{COR:FIN-ADDX} can be viewed as an infinite version of the fact that every
independent set can be extended to a basis, which all have the same
cardinality.

Similarly, we can generalize the Steinitz Augmentation Property of independent
sets of a matroid: if $A$ and $B$ are independent sets and $|B|>|A|$, then
there is an independent set $C$ such that $A\subseteq C\subseteq A\cup B$ and
$|C|\ge|B|$. Let us say that a family $\Hf$ of charges on $(J,\BB)$ has the
{\it exchange property}, if for all $\alpha,\beta\in\Hf$ there is a
$\gamma\in\Hf$ such that $\alpha\le \gamma\le \alpha\lor\beta$ and
$\gamma(J)\ge\beta(J)$. (Of course, this condition is interesting only if
$\alpha(J)<\beta(J)$, else $\gamma=\alpha$ is a trivial solution.)

\begin{theorem}\label{THM:M-AUGMENT}
Let $\fg$ be an increasing submodular setfunction on a set-algebra $(J,\BB)$
with $\fg(\emptyset)=0$. Then $\matp_+(\fg)$ has the exchange property.
\end{theorem}

\begin{proof}
Consider the setfunctions
\[
\fg_0 = ((\fg-\alpha)^c\land 0)^c \et \beta_0 = (\alpha\lor \beta)-\alpha.
\]
Clearly both of them are nonnegative and increasing, and
$\fg_0(\emptyset)=\beta_0(\emptyset)=0$. Lemma \ref{LEM:LAND1} implies that
$\fg_0$ is submodular. Furthermore, $\beta_0$ is a charge. Hence by Lemma
\ref{LEM:LAND1}, $\psi=\fg_0\land\beta_0$ is a bounded increasing submodular
setfunction with $\psi(\emptyset)=0$.

By Corollary \ref{COR:FIN-ADD-LAT}, there is a charge $\delta\in\matp_+(\psi)$
such that $\delta(J)=\psi(J)$. Let $\gamma=\alpha+\delta$. Then
\[
\gamma \le \alpha+\psi\le \alpha + \fg_0 \le \alpha + (\fg-\alpha)=\fg,
\]
and
\[
\gamma \le \alpha  + \beta_0 = \alpha + (\alpha\lor\beta)-\alpha = \alpha\lor\beta.
\]
To conclude, we show that $\gamma(J)\ge \beta(J)$. Let $x,Z\in\BB$ such that
$Z\supseteq X$. Then we have
\begin{align*}
\fg(Z)-\alpha(Z) &\ge \beta(Z)-\alpha(Z) = \beta(J) - \alpha(J)
- (\beta(Z^c)-\alpha(Z^c)\\
&\ge \beta(J) - \alpha(J) - \beta_0(X^c).
\end{align*}
The infimum of the left hand side over all $Z\supseteq X$ for a fixed $X$ is
$\fg_0(X)$, so we get
\[
\fg_0(X) + \beta_0(X^c)\ge \beta(J) - \alpha(J)
\]
for all $X$. The supremum of the left hand side is just $\psi(J)$.
\end{proof}

The last theorem above motivates the following converse to Corollary
\ref{COR:SIPOS}:

\begin{theorem}\label{THM:INV-SIPOS}
Let $\Hf$ be a family of nonnegative charges on a set-algebra $(J,\BB)$ with
the exchange property. Then the setfunction defined by
$\fg(X)=\sup_{\alpha\in\Hf} \alpha(X)$ is submodular.
\end{theorem}

\begin{proof}
Let $X,Y\in\BB$ and $\eps>0$. Let $\alpha,\beta\in\Hf$ be chosen so that
\[
\alpha(X\cap Y)\ge \fg(X\cap Y)-\eps \et \beta(X\cup Y)\ge \fg(X\cup Y)-\eps.
\]
We may assume that $\alpha$ is zero outside $X\cap Y$ and $\beta$ is zero
outside $X\cup Y$. By the exchange property, there is a $\gamma\in\Hb$ such
that $\alpha\le \gamma\le \alpha\lor\beta$ and $\gamma(J)\ge\beta(J)$. Then
$\gamma$ is zero outside $X\cup Y$, and so $\gamma(X\cup Y)=\gamma(J)\ge
\beta(J)\ge \beta(X\cup Y)\ge \fg(X\cup Y)-\eps$ and $\gamma(X\cap Y) \ge
\alpha(X\cap Y)\ge \fg(X\cap Y)-\eps$, and so
\[
\fg(X)+\fg(Y) \ge \gamma(X)+\gamma(Y)=\gamma(X\cup Y)+\gamma(X\cap Y)
\ge \fg(X\cup Y)+\fg(X\cap Y)-2\eps.
\]
This holds for every $\eps$, proving the submodular inequality.
\end{proof}

\subsection{Common minorizing measures}\label{SSEC:COMMON-MEAS}

Let $\fg$ and $\psi$ be two increasing submodular setfunctions with
$\fg(\emptyset)=\psi(\emptyset)=0$. We want to study the convex set
$\matp_+(\fg)\cap\matp_+(\psi)$, i.e., those charges $\alpha\in\ba_+$ that
satisfy both $\alpha\le\fg$ and $\alpha\le\psi$. Recall that neither
$\min(\fg,\psi)$ nor $\fg\land\psi$ is submodular in general; nevertheless, the
formulas in the previous section have extensions to this more general problem.
Our results can be considered as generalizations of the fractional version of
the Matroid Intersection Theorem.

We start with a couple of simple observations.

\begin{lemma}\label{LEM:TRIV}
Let $\fg$ and $\psi$ be two bounded submodular setfunctions on a set-algebra
$(J,\BB)$ such that $\fg(\emptyset)=\psi(\emptyset)=0$.

\smallskip

{\rm(a)} A charge $\alpha\in\ba$ satisfies $\alpha\le\min(\fg,\psi)$ if and
only if it satisfies $\alpha\le\fg\land\psi$.

\smallskip

{\rm(a)} For every charge $\alpha\le\fg\land\psi$ there is a charge $\beta$
satisfying $\alpha\le\beta\le\fg\land\psi$ such that $\|\beta\|\le
2\|\fg\|+2\|\psi\|$.
\end{lemma}

\begin{proof}
(a) We have $(\fg\land\psi)(X) \le \fg(X) + \psi(\emptyset)=\fg(X)$ and
similarly $(\fg\land\psi)(X) \le \psi(X)$, and so
$\fg\land\psi\le\min(\fg,\psi)$. Hence the ``if'' part is trivial. Conversely,
assume that $\alpha\le\min(\fg,\psi)$. Then for every $X,Y\in\BB$ with
$Y\subseteq X$, we have
\[
\alpha(X)=\alpha(Y)+\alpha(X\setminus Y)\le \fg(Y)+\psi(X\setminus Y).
\]
Taking the infimum of the right hand side over all $Y$, we get $\alpha(X)\le
(\fg\land\psi)(X)$.

\medskip

(b) By Corollary \ref{COR:FIN-ADDX}, there are charges $\beta,\gamma$ such that
$\alpha\le\beta\le\fg$, $\alpha\le \gamma\le\psi$, $\beta(J)=\phi(J)$ and
$\gamma(J)=\psi(J)$. Then the charge $\beta\land\gamma$ satisfies
$\alpha\le\beta\land\gamma\le\fg\land\psi$. As remarked in Section
\ref{SSEC:MINORIZE}, we have $\|\beta\|\le2\|\fg\|$ and
$\|\gamma\|\le\|\psi\|$. Hence $\|\beta\land\gamma\|\le \|\beta\|+\|\gamma\|\le
2\|\fg\|+2\|\psi\|$.
\end{proof}

With a slight modification of the proof of Corollary \ref{COR:FIN-ADDX}, we get
the following.

\begin{lemma}\label{LEM:INTERSECT1}
Let $\fg$ and $\psi$ be submodular setfunctions on a set-algebra $(J,\BB)$ with
$\fg(\emptyset)=\psi(\emptyset)=0$. Then there is a $\alpha\in \ba(\BB)$ such
that $\alpha\le\fg\land\psi$ and $\alpha(J)=(\fg\land\psi)(J)$.
\end{lemma}

\begin{proof}
Set $b=(\fg\land\psi)(J)$ and define $\psi':~\BB\to\R$ by
\[
\psi'(X)=
  \begin{cases}
    0, & \text{if $X=J$}, \\
    b-\psi(X^c) & \text{otherwise}.
  \end{cases}
\]
Clearly $b-\psi(X^c)$ is a supermodular function of $X$, and $n-\psi(J)\le 0$,
hence $\psi'$ is supermodular. Furthermore, $\psi'\le\fg$ by the definition of
$\fg\land\psi$. Hence by Theorem \ref{THM:SEPARATE}, there is a modular
setfunction $\alpha$ such that $\psi'\le\alpha\le\fg$. It follows that
$\alpha(\emptyset)=0$, so $\alpha$ is a charge, and trivially
$0\le\alpha\le\fg$ and $\alpha(J)=b$.
\end{proof}

The following generalization of this lemma, and also of Corollary
\ref{COR:SIPOS}, can be considered as the infinite version of the fractional
Matroid (or Polymatroid) Intersection Theorem.

\begin{theorem}\label{THM:INTERSECT}
Let $\fg$ and $\psi$ be submodular setfunctions on a set-algebra $(J,\BB)$ with
$\fg(\emptyset)=\psi(\emptyset)=0$. Then for every $X\in\BB$,
\begin{equation}\label{EQ:INTERSECT0}
\max\{\alpha(X):~\alpha\in\basp(\fg)\cap\basp(\psi)\}=(\fg\land\psi)(X).
\end{equation}
If $\fg$ and $\psi$ are increasing, then
\begin{equation}\label{EQ:INTERSECT1}
\max\{\alpha(X):~\alpha\in\basp_+(\fg)\cap\basp_+(\psi)\}=(\fg\land\psi)(X).
\end{equation}
\end{theorem}

In particular, the maxima on the left are attained. Note that by Corollary
\ref{COR:FIN-ADDX}(b), the condition $\alpha\in\basp(\fg)\cap\basp(\psi)$ could
be replaced by $\alpha\in\matp(\fg)\cap\matp(\psi)$, which in turn is
equivalent to $\alpha\le\min(\fg,\psi)$ and with $\alpha\le\fg\land\psi$.
Similar equivalences hold for the second assertion.

\begin{proof}
First we prove that
\begin{equation}\label{EQ:INTERSECT2}
\sup\{\alpha(X):~\alpha\in\ba(\BB),\,\alpha\le\fg\land\psi\}\le (\fg\land\psi)(X).
\end{equation}
For any $Y\in\BB$ with $Y\subseteq X$, and $\alpha\le \fg\land\psi$, we have
\[
\alpha(X)=\alpha(Y)+\alpha(X\setminus Y)\le \fg(Y)+\psi(X\setminus Y),
\]
and taking the supremum over all $\alpha$ on the left and infimum over $Y$ on
the right, we get \eqref{EQ:INTERSECT2}.

Second, we prove that there exists a charge $\alpha\le\fg\land\psi$ attaining
equality in \eqref{EQ:INTERSECT2}; this will also prove that the supremum is
attained.

\medskip

{\bf Case 1:} $X=J$. Let $b=(\fg\land\psi)(J)$. We may assume that
$\fg(J)=\psi(J)=b$. Indeed, clearly $\fg(J),\psi(J)\ge b$, and if we lower the
values $\fg(J)$ and $\psi(J)$ to $b$, then the setfunctions $\fg$ and $\psi$
remain  submodular and $\fg\land\psi$ does not change.

The condition defining $b$ means that $\fg(Y)\ge b-\psi(Y^c)$ for $Y\in\BB$.
Here $b-\psi(Y^c)$ is a supermodular setfunction, so by Theorem
\ref{THM:SEPARATE} there is a modular setfunction $\alpha$ such that
\[
\fg(Y)\ge\alpha(Y)\ge b-\psi(Y^c)\qquad(Y\in\BB).
\]
Substituting $Y=\emptyset$ and $Y=J$, we see that $\alpha(\emptyset)=0$ and
$\alpha(J)=b$. So $\alpha$ is a charge. Applying the condition on $\alpha$ to
$Y^c$, we get $b-\alpha(Y)=\alpha(Y^c)\ge b-\psi(Y)$, and hence
$\alpha\le\psi$.

Now assume that $\fg$ and $\psi$ are increasing, and let again
$b=(\fg\land\psi)(J)$. We cannot assume any more that $\fg(J)=b$ or
$\psi(J)=b$, but we already know that there is an $\alpha\in\ba(\BB)$ such that
$\alpha\le\min(\fg,\psi)$ and $\alpha(J)=b$. For every $\eps>0$ there is a set
$U\subseteq J$ such that $b\ge\fg(U)+\psi(U^c)-\eps$. Then the computation
\begin{align*}
\alpha(Y)&=\alpha(J)-\alpha(Y^c) = b-\alpha(Y^c) = b-\alpha(Y^c\cap U)-\alpha(Y^c\cap U^c)\\
&\ge b-\fg(Y^c\cap U)-\psi(Y^c\cap U^c) \ge b-\fg(U)-\psi(U^c) \ge -\eps
\end{align*}
shows that $\alpha$ is nonnegative.

\medskip

{\bf Case 2:} $X\not=J$. We restrict $\BB$, $\fg$ and $\psi$ to the subsets of
$X$, and apply the special case proved above. Clearly $(\fg\land\psi)(X)$ does
not change, so we get a charge $\beta$ on $\BB|_X$ such that
$\beta\le\fg\land\psi$ and $\beta(X)=(\fg\land\psi)(X)$.

Consider the projections of $\fg$ and $\psi$ on $X^c$:
\[
\fg_1(Z)=\fg(Z\cup X)-\fg(X) \et \psi_1(Z)=\psi(X\cup Z)-\psi(X).
\]
Clearly $\fg_1(\emptyset)=\psi_1(\emptyset)=0$, so by the special case proved
above, there is a charge $\gamma$ on $\BB|_{X^c}$ such that
$\gamma\le\fg_1\land\psi_1$ and $\gamma(X^c)=(\fg_1\land\psi_1)(X^c)$.

We define $\alpha(Y)=\beta(Y\cap X)+\gamma(Y\cap X^c)$. Then $\alpha$ is a
signed measure, and $\alpha(X)=\beta(X)=(\fg\land\psi)(X)$. For $Y\in\BB$,
submodularity yields that
\begin{align*}
\alpha(Y)&=\beta(Y\cap X)+\gamma(Y\cap X^c) \le \fg(Y\cap X) + \fg_1(Y\cap X^c)\\
&= \fg(Y\cap X) + \fg(Y\cup X)-\fg(X) \le\fg(Y).
\end{align*}
Similarly $\alpha\le\psi$, which implies by Lemma \ref{LEM:TRIV} that
$\alpha\le\fg\land\psi$.

If $\fg$ and $\psi$ are increasing, then, as remarked above, $\beta\ge 0$, and
we can take $\gamma=0$, so we can require that $\alpha\ge 0$.
\end{proof}

The weighted matroid intersection also generalizes.

\begin{theorem}\label{THM:WEIGHT-INTERSECT}
Let $\fg$ and $\psi$ be bounded submodular setfunctions with
$\fg(\emptyset)=\psi(\emptyset)=0$. For every function $w\in\Bd_+(\BB)$,
\begin{equation}\label{EQ:INTERSECT-W1}
\max\{\wh\alpha(w):~\alpha\in\basp(\fg)\cap\basp(\psi)\}  = \inf_{0\le h\le w} \wh\fg(h)+\wh\psi(w-h).
\end{equation}
If $\fg$ and $\psi$ are increasing, then
\begin{equation}\label{EQ:INTERSECT-W2}
\max\{\wh\alpha(w):~\alpha\in\basp_+(\fg)\cap\basp_+(\psi)\}  = \inf_{0\le h\le w} \wh\fg(h)+\wh\psi(w-h).
\end{equation}
\end{theorem}

\begin{proof}
The inequality $\le$ is trivial in both formulas. To prove the reverse
inequality in \eqref{EQ:INTERSECT-W1}, we may assume that $w<1$. First we
assume that there is a $c>0$ such that $w\ge c$. Then the function
$u(x)=1/w(x)$ is bounded.

Consider the setfunctions $\fg_0=w\cdot\fg$ and $\psi_0=w\cdot\psi$. By Lemma
\ref{LEM:SUBMOD-TRIV}(h) these are submodular, so by Theorem
\ref{THM:INTERSECT}, we have
\begin{equation}\label{EQ:INTERSECT-W3}
\max\{\beta(J):~\beta\in\ba(\BB),\,\beta\le\min(\fg_0,\psi_0)\}=(\fg_0\land\psi_0)(J).
\end{equation}
Let $\beta$ attain the maximum on the left, and let $\alpha=u\cdot \beta$.
Lemma \ref{LEM:CDOT} implies that $w\cdot\alpha=(wu)\cdot\beta=\beta$. In
particular,
\[
\beta(J) = \intl_J w\,d\alpha = \wh\alpha(w).
\]
Using the definition of $w\cdot\fg$, we have
\[
(\fg_0\land\psi_0)(J) = \inf_{X\in\BB} \big(\fg_0(X)+\psi_0(X^c)\big) = \inf_{X\in\BB} \big(\wh\fg(w\one_X)
+ \wh\psi(w\one_{X^c})\big).
\]
Since $0\le w\one_X\le w$ and $w\one_{X^c}=w-w\one_X$, this implies that
\[
(\fg_0\land\psi_0)(J) \ge \inf_{0\le h\le w} \wh\fg(h)+\wh\psi(w-h).
\]
Since we already know the reverse inequality, this proves
\eqref{EQ:INTERSECT-W1} for the case when $w\ge c>0$.

In the case when $w$ is not separated from $0$, we apply the special case above
to the function $w+\eps$, and let $\eps\to 0$. To make this precise, the case
we have proved implies that there is a charge
$\alpha_\eps\in\basp(\fg)\cap\basp(\psi)$ such that
\begin{equation}\label{EQ:EPS2NULL}
\wh\alpha_\eps(w+\eps) = \inf_{0\le h\le w+\eps} \wh\fg(h)+\wh\psi(w+\eps-h).
\end{equation}
Let $h'=\min(h,w)$. Then $0\le h-h'\le\eps$, and hence
\[
\wh\fg(h) \ge \wh\fg(h')-\wh\fg(h'-h) =\wh\fg(h')+O(\eps),
\]
and similarly,
\[
\wh\psi(w+\eps-h) \ge \wh\psi(w-h')-\wh\psi(h-h'-\eps) = \wh\psi(w-h')+O(\eps),
\]
Furthermore,
\[
\wh\alpha_\eps(w+\eps) = \wh\alpha_\eps(w) + \eps\alpha_\eps(J) \le \wh\alpha_\eps(w) + \eps\fg(J).
\]
The right hand side of \eqref{EQ:EPS2NULL} satisfies
\[
\inf_{0\le h\le w+\eps} \wh\fg(h)+\wh\psi(w+c-h) \ge \inf_{0\le h'\le w} \wh\fg(h')+\wh\psi(w-h') + O(\eps).
\]
Letting $\eps\to 0$, we get
\begin{equation}\label{EQ:INTERSECT-W4}
\sup\{\wh\alpha(w):~\alpha\in\ba,\, \alpha\in \basp(\fg)\cap\basp(\psi)\}  = \inf_{0\le h\le w} \wh\fg(h)+\wh\psi(w-h).
\end{equation}
By Lemma \ref{LEM:TRIV}, in the supremum we may also require that
$\|\alpha\|\le2\|\fg\|+2\|\psi\|$.

To prove that the supremum is attained, let $C=2\|\fg\|+2\|\psi\|$, and
consider the space $T=\R^\BB$  with the product topology and its subspace
$K=[-C,C]^\BB$. This is compact by Tikhonov's Theorem\footnote{I am grateful to
Boglárka Gehér for suggesting this shortcut in my original proof.}. By the
remarks in section \ref{SSEC:CHOQ-INT}, the sets $\basp(\fg)$ and $\basp(\psi)$
are closed, and hence $\basp(\fg)\cap\basp(\psi)$ is closed and thus compact.
By Lemma \ref{LEM:WHPHI-CONT}, the map $\alpha\mapsto\wh\alpha(f)$ is
continuous, and so it attains its maximum on $\basp(\fg)\cap\basp(\psi)$.

Finally, if (say) $\psi$ is increasing, then so is $w\cdot\psi$, and so in
\eqref{EQ:INTERSECT-W3} we have a maximizing $\beta\ge 0$ by the second
statement of Theorem \ref{THM:INTERSECT}.
\end{proof}

\begin{remark}\label{REM:GEN-M-INTERSECT}
There are two possible natural generalizations of Theorem \ref{THM:INTERSECT},
both of which fail rather fundamentally even in the finite case.

First, we may want to find a charge $\alpha\le\fg\land\psi$ that coincides with
$\fg\land \psi$ not just on one set but on a whole chain of sets (as in
Corollary \ref{COR:FIN-ADD-LAT}). As remarked, this would mean that an optimal
minorizing measure can be found by the Greedy Algorithm, which does not work
even for very simple combinatorial examples of matroid intersection, like for
matchings (see Section \ref{SEC:MATCH}).

Second, it would be natural to consider three or more submodular setfunctions
instead of two, and construct common minorizing measures. In matroid theory,
this would amount to finding a common basis of three or more matroids, and no
min-max formula is known for the rank of this intersection. In fact, to
determine the rank of the intersection of three almost trivial matroids can be
NP-hard (it contains the matching problem for $3$-uniform hypergraphs).
\end{remark}

\section{Smoothness properties}\label{SEC:CONT}

Several of the results about submodular setfunctions can be strengthened if the
setfunctions have some smoothness properties. Continuity from below and from
above seem to be the most useful, but other smoothness properties also come up.
(For a charge, most of these conditions imply countable additivity.) In the
last three sections we describe applications of these conditions.

In this section $(J,\BB)$ is a sigma-algebra.

\subsection{Continuity}

Let $\fg$ be a setfunction on a sigma-algebra $(J,\BB)$. We say that $\fg$ is
{\it continuous from above}, if for every sequence $A_1\supseteq
A_2\supseteq\ldots$ of sets $A_i\in\BB$ and $A=\cap_i A_i$, we have
$\fg(A_n)\to \fg(A)$ $(n\to\infty)$. {\it Continuous from below} is defined
analogously, considering sequences $A_1\subseteq A_2\subseteq\ldots$ and
$A=\cup_i A_i$.

If $\fg$ is increasing, subadditive and $\fg(\emptyset)=0$, then in the
definition of continuity from above, it suffices to consider the case
$A=\emptyset$. Indeed, let $A=\cap_i A_i$ any set in $\BB$, then
$\fg(A_i)\ge\fg(A)$ by monotonicity, and $\fg(A_i) \le \fg(A_i\setminus
A)+\fg(A)$ by subadditivity. Here $\fg(A_i\setminus A)\to 0$ by continuity from
above, which implies that $\fg(A_i)\to \fg(A)$.

Furthermore, for an increasing and subadditive setfunction $\fg$, continuity
from above implies continuity from below. Indeed, let $A_1\subseteq
A_2\subseteq\ldots$ and $A=\cup_i A_i$. Then $\fg(A_i)\le\fg(A)$ by
monotonicity, and $\fg(A_i) \ge \fg(A)-\fg(A\setminus A_i)$ by subadditivity.
Here $\fg(A\setminus A_i)\to 0$ by continuity from above, which implies that
$\fg(A_i)\to \fg(A)$.

For charges, continuity from below or from above are equivalent properties (and
equivalent to countable additivity); but for the more general case of
submodular setfunctions, these properties are not equivalent. As in Example
\ref{EXA:IDEAL}, let $\chi(X)=\one(X~\text{is uncountable})$. Since the union
of a countable family of countable sets is countable, $\chi$ is continuous from
below, but the intersection of a countable chain of uncountable sets can be
empty, showing that $\chi$ is not continuous from above. Another example is the
rank function of subsets of a linear space $L$ (Example \ref{EXA:LIN-DEP}):
this setfunction is continuous from below, but not continuous from above, as
shown by any sequence $A_1\supseteq A_2\supseteq\ldots$ of subsets of $L$ with
$\cap_i A_i=\emptyset$, such that each $A_i$ contains a nonzero vector (so
$r(A_i)\ge 1$).

We note that subadditivity and continuity from below imply that $\fg$ is an
outer measure in the following sense: If $A_1,A_2,\ldots\in\BB$, then
\begin{equation}\label{EQ:OUTER}
\fg(\cup_i A_i)\le \sum_i\fg(A_i).
\end{equation}
Indeed, using the finite subadditivity of $\fg$,
\begin{align*}
\fg(A_1\cup\ldots\cup A_n) \le \fg(A_1)+\ldots+\fg(A_n).
\end{align*}
Here the left hand side tends to $\fg(\cup_i A_i)$ as $n\to\infty$ by
continuity from below, and the right hand side tends to
$\sum_{i=1}^\infty\fg(A_i)$.

From an increasing submodular setfunction $\fg$, we can ``extract'' a part that
is continuous from below by the following construction:
\[
\fg_\lc(X) = \inf_\XX \lim_{n\to\infty} \fg(X_n),
\]
where $\XX=(X_1,X_2,\ldots)$ ranges over all sequences of sets in $\BB$ such
that $X_1\subseteq X_2\subseteq\ldots$ and $\cup_n X_n=X$.

Clearly $\fg_\lc\le\fg$, and $\fg_\lc=\fg$ if and only if $\fg$ is continuous
from below. Warning: $\fg_\lc$ may not be continuous from below for general
increasing setfunctions (but it will turn out continuous from below if $\fg$ is
increasing and submodular).

\begin{example}\label{EXA:NOT-LSC}
Consider the space $(J,\BB)=(0,1)^\Nbb$. For $i,n\in\Nbb$, let
\[
X_n = \Big\{u\in J:~u_1\le 1-\frac1n\Big\} \et X_{n,i}
=  \Big\{u\in X_n: u_n\le 1-\frac1i\Big\}.
\]
Then $\cup_n X_n = J$ and $\cup_i X_{n,i}=X_n$. Define
\[
\fg(Y)=
  \begin{cases}
    0, & \text{if $Y\subseteq X_{n,i}$ for some $n$ and $i$}, \\
    1, & \text{otherwise}.
  \end{cases}
\]
Then $\fg$ is an increasing setfunction (not submodular). Clearly
$\fg(X_{n,i})=0$, which implies that $\fg_\lc(X_n)=0$.

We claim that $\fg_\lc(J)=1$, which shows that $\fg_\lc$ is not continuous from
below. It suffices to show that $\cup_k Y_k\not=J$ for any sequence
$Y_1\subseteq Y_2\subseteq\ldots$ of sets in $\BB$ with $\fg(Y_n)=0$. Each
$Y_k$ is contained in a set $X_{n_k,i_k}$. The sequence $(n_1,n_2,\ldots)$
cannot be bounded, since then all $Y_k$ would be contained in $X_n$ for a
sufficiently large $n$, and their union could not be $J$. We can delete sets
$Y_k$ from the sequence without changing their union, as long as we keep
infinitely many; hence we may assume that the indices $n_k$ are distinct.

Now consider the point $u\in J$ for which
\[
u_m=
  \begin{cases}
    1-\frac1{2i_k}, & \text{if $m=n_k$}, \\
    1/2, & \text{otherwise}.
  \end{cases}
\]
Then $u\notin X_{n_k,i_k}$ for any $k$, and hence $\cup_k Y_k\not=J$.
\end{example}

Let us look at a couple of submodular examples.

\begin{example}\label{EXA:LEBESGE2}
Consider the setfunction $\overline{\lambda}$ defined in Example
\ref{EXA:CLOSURE}. We claim that
\begin{equation}\label{EQ:OVL}
\overline{\lambda}_\lc=\lambda.
\end{equation}
Indeed, let $X_1\subseteq X_2\subseteq\ldots\in\BB$ and $\cup_n X_n=X$. Then
\[
\lim_n \overline{\lambda}(X_n) = \lim_n \lambda(\overline{X}_n) \ge \lim_n \lambda(X_n) =\lambda(X),
\]
which implies that $\overline{\lambda}_\lc \ge \lambda$. On the other hand, let
$\eps>0$, and let $\{I_1,I_2,\ldots\}$ be a countable family of closed
intervals covering $X$ such that $\sum_n \lambda(I_n)\le \lambda(X)+\eps$. Let
$X_n=(\cup_{k=1}^n I_k)\cap X$. Clearly the sequence $(X_1,X_2,\ldots)$ is
increasing and $\cup_nX_n=X$. Furthermore, $\overline{X}_n\subseteq
\cup_{k=1}^n I_k$, and hence
\[
\overline{\lambda}(X_n)\le \sum_{k=1}^n \lambda(I_k) \le \sum_{k=1}^\infty \lambda(I_k)\le \lambda(X) +\eps.
\]
Letting $n\to\infty$ and $\eps\to0$, we get that $\overline{\lambda}_\lc
\le\lambda$, proving \eqref{EQ:OVL}.
\end{example}

\begin{example}\label{EXA:COUNT-UNCOUNT}
As in Example \ref{EXA:IDEAL}, let $\chi(X)=\one(X~\text{is infinite})$. Then a
jump of $\chi$ is any infinite set that is the union of a countable chain of
finite sets. This means that $\chi_\lc=\one(X~\text{is uncountable})$. This
setfunction is submodular but not modular. Let us note that
$\chi(X)-\chi_\lc(X)=\one(X~\text{is countably infinite})$ is neither
submodular nor monotone.
\end{example}

If $\alpha$ is a charge, then $\alpha_\lc$ is a countably additive measure;
this is not hard to see. In this case $\alpha-\alpha_\lc$ is purely finitely
additive, and $\alpha=\alpha_\lc+(\alpha-\alpha_\lc)$ is the canonical
decomposition of $\alpha$ mentioned in Section \ref{SSEC:FIN-ADD} of the
Appendix.

For the more general case of an increasing submodular but not modular
setfunction, $\fg_\lc$ may be modular (Example \ref{EXA:LEBESGE2}), but it is
not always so (Example \ref{EXA:COUNT-UNCOUNT}). But we are going to show that
it is submodular. The difference $\fg-\fg_\lc$ may be a submodular (as in
Example \ref{EXA:LEBESGE2}), but it is neither monotone nor submodular in
general, as shown by Example \ref{EXA:COUNT-UNCOUNT}. But at least it is
subadditive.

\begin{lemma}\label{LEM:SEMIC-PART}
Let $\fg$ be an increasing submodular setfunction. Then the setfunction
$\fg_\lc$ is increasing, submodular, and continuous from below.
\end{lemma}

\begin{proof}
Let $Y\subseteq X$. For every $\eps>0$ there are sets $X_1\subseteq
X_2\subseteq\ldots$ in $\BB$ such that $\cup_n X_n=X$ and $\fg(X_n)\le
\fg_\lc(X)+\eps$. Let $Y_n=Y\cap X_n$, then $Y_1\subseteq Y_2\subseteq\ldots$
and $\cup_n Y_n=Y$, and moreover, $\fg(Y_n)\le\fg(X_n)\le \fg_\lc(X)+\eps$.
This implies that $\fg_\lc(Y)\le \fg_\lc(X)+\eps$. Since this holds for every
$\eps>0$, it follows that $\fg_\lc(Y)\le \fg_\lc(X)$, so $\fg_\lc$ is
increasing.

Next, to prove submodularity, let $X,Y\in\BB$. For every $\eps>0$ there are
sets $X_1\subseteq X_2\subseteq\ldots$ and $Y_1\subseteq Y_2\subseteq\ldots$ in
$\BB$ such that $\cup_n X_n=X$, $\cup_n Y_n=Y$, $\fg(X_n)\le \fg_\lc(X)+\eps$
and $\fg(Y_n)\le \fg_\lc(Y)+\eps$. Then the sets $U_n=X_n\cup Y_n$ and
$W_n=X_n\cap Y_n$ have the property that $U_1\subseteq U_2\subseteq\ldots$,
$W_1\subseteq W_2\subseteq\ldots$, $\cup_n U_n=X\cup Y$ and $\cup_n W_n=X\cap
Y$. Furthermore,
\[
\fg(U_n)+\fg(W_n)\le\fg(X_n)+\fg(Y_n)\le \fg_\lc(X)+\fg_\lc(Y)+2\eps,
\]
and letting $n\to\infty$, we get that $\fg_\lc (U_n)+\fg_\lc(W_n)\le
\fg_\lc(X)+\fg_\lc(Y)+2\eps$. Again, $\eps$ is arbitrary, so the submodularity
of $\fg_\lc$ follows.

The hardest part is to prove that $\fg_\lc$ is continuous from below. We want
to prove that for any sequence $X_1\subseteq X_2\subseteq\ldots$ of sets on
$\BB$ with $\cup_n X_n=X$, we have
\begin{equation}\label{EQ:CONT-BELOW}
\fg_\lc(X)\le\lim_{n\to\infty}\fg_\lc(X_n).
\end{equation}
Let $\eps>0$ and $\eps_n=\eps/3^n$. By the definition of $\fg_\lc$, for every
$n$ there is a sequence $X_{n,1}\subseteq X_{n,2}\subseteq\ldots$ of sets in
$\BB$ such that $\cup_i X_{n,i}=X_n$ and $\fg(X_{n,i})\le \fg_\lc(X_n)+\eps_n$.

Define $Y_{n,i}=X_{1,i}\cup \ldots \cup X_{n,i}$ and $Z_{n,i}=Y_{n,i}\cap
X_{n+1,i}$. Clearly $Y_{n,i+1}\supseteq Y_{n,i}$ and $Z_{n,i+1}\supseteq
Z_{n,i}$. Furthermore,
\begin{equation}\label{EQ:YNI}
\bigcup_i Y_{n,i} = X_n.
\end{equation}
Indeed, $Y_{n,i}\subseteq X_1\cup\ldots\cup X_n=X_n$, but $\cup_i Y_{n,i}
\supseteq \cup_i X_{n,i}= X_n$. Hence
\begin{equation}\label{EQ:XNI-YNI}
\bigcup_i Z_{n,i} = X_n,
\end{equation}
since every $x\in X_n$ is contained in $Y_{n,i}$ as well as in $X_{n+1,i}$ if
$i$ is large enough. By the definition of $\fg_\lc$, it follows that
\begin{equation}\label{EQ:ZNI-LARGE}
\fg(Z_{n,i})\ge \fg_\lc(X_n)-\eps_n
\end{equation}
if $i$ is large enough, say $i\ge f(n)$. We may assume that $f(n)$ is a
strictly increasing function.

We prove by induction on $n$ that if $i\ge f(n)$, then
\begin{equation}\label{EQ:YNI-SMALL}
\fg(Y_{n,i})\le\fg_\lc(X_n)+\eps-2\eps_n.
\end{equation}
For $n=1$, we have $\fg(Y_{1,i})=\fg(X_{1,i}) \le \fg_\lc(X_1) + \eps_1 =
\fg_\lc(X_1) +\eps-2\eps_1$. For the induction step, we use submodularity. For
$i\ge f(n+1)$, we have
\begin{align*}\label{EQ:Y-SUBMOD}
\fg(Y_{n+1,i}) &= \fg(Y_{n,i}\cup X_{n+1,i}) \le \fg(Y_{n,i}) +\fg(X_{n+1,i})-\fg(Z_{n,i})\\
&\le (\fg_\lc(X_n)+\eps-2\eps_n) + (\fg_\lc(X_{n+1}) + \eps_{n+1}) - (\fg_\lc(X_n)-\eps_n)\\
&= \fg_\lc(X_{n+1}) +\eps-\eps_n+\eps_{n+1} = \fg_\lc(X_{n+1}) +\eps-2\eps_{n+1}.
\end{align*}
This proves \eqref{EQ:YNI-SMALL}.

Consider that sequence of sets $(Y_{n,f(n)}:~n=1,2,\ldots)$. Clearly it is
increasing. We claim that
\begin{equation}\label{EQ:YNI-X}
\bigcup_n Y_{n,f(n)} = X.
\end{equation}
Indeed, $Y_{n,f(n)}\subseteq X$, so the $\subseteq$ containment is trivial. To
prove equality, let $x\in X$. Then $x\in X_k$ for some $k$, and hence $x\in
X_{k,i}$ for some $i$. Let $n$ be large enough such that $i\le f(n)$, then
$X_{k,i}\subseteq X_{k,f(n)}\subseteq Y_{n,f(n)}$, so $x\in Y_{n,f(n)}$.

To conclude, note that \eqref{EQ:YNI-X} and \eqref{EQ:YNI-SMALL} imply that
\[
\fg_\lc(X)\le \lim_{n\to\infty}\fg_\lc(Y_{n,f(n)})\le
\lim_{n\to\infty} \fg_\lc(X_n) +\eps-2\eps_n = \lim_{n\to\infty} \fg_\lc(X_n) +\eps.
\]
This holds for every $\eps>0$, proving \eqref{EQ:CONT-BELOW}.
\end{proof}

We supplement this lemma with a weaker statement about the difference.

\begin{lemma}\label{LEM:SEMIC-PART2}
Let $\fg$ be an increasing submodular setfunction. Then the difference
$\fg-\fg_\lc$ is subadditive.
\end{lemma}

\begin{proof}
Let $\eps>0$, $X,Y\in\BB$, and let $Z_1\subseteq Z_2\ldots\in\BB$ be a chain of
sets such that $\cup_n Z_n=X\cup Y$ and $\fg(Z_n)\le\fg_\lc(X\cup Y)+\eps$. Let
$U_n=X\cap Z_n$, $V_n=X\cup Z_n$, and $W_n=Y\cap V_n$. Then $\cup_n U_n =X$,
$\cup_n W_n=Y$, and $V_n\cup Y = X\cup Y$. By submodularity,
\[
\fg(X)+\fg(Z_n) \ge \fg(V_n)+\fg(U_n),
\]
and
\[
\fg(Y) + \fg(V_n) \ge \fg(Y\cup V_n) + \fg(W_n) = \fg(X\cup Y) + \fg(W_n).
\]
Adding up these inequalities, canceling what we can, and rearranging, we obtain
\[
\fg(X)- \fg(U_n) + \fg(Y) - \fg(W_n) \ge  \fg(X\cup Y) - \fg(Z_n) \ge \fg(X\cup Y) - \fg_\lc(X\cup Y)-\eps .
\]
By definition,
\[
\fg(X)- \lim_n\fg(U_n) \le \fg(X)- \fg_\lc(X) \et \fg(Y)- \lim_n\fg(W_n) \le \fg(Y)-\fg_\lc(Y),
\]
so letting $n\to\infty$ and $\eps\to 0$, we get
\[
\fg(X)- \fg_\lc(X) + \fg(Y)-\fg_\lc(Y) \ge \fg(X\cup Y) - \fg_\lc(X\cup Y),
\]
proving the lemma.
\end{proof}

Let $\fg$ be an increasing submodular setfunction. If $\psi\le\fg$ is another
increasing setfunction, then clearly $\psi_\lc\le\fg_\lc$. If, in particular,
$\psi$ is continuous from below, then $\psi\le\fg_\lc$. So by lemma
\ref{LEM:SEMIC-PART}, $\fg_\lc$ is the unique largest increasing setfunction
minorizing $\fg$ that is continuous from below, and it is also submodular.

\subsection{Absolute continuity}

Let $\pi$ be a measure on a sigma-algebra $(J,\BB)$. We say that a setfunction
$\fg$ is {\it absolutely continuous with respect to $\pi$}, if for every
$\eps>0$ there is a $\delta>0$ such that $\pi(S)<\delta$ implies that
$|\fg(S)|<\eps$ for every $S\in\BB$. Let us say that $\fg$ is {\it weakly
absolutely continuous} with respect to $\pi$ if $\fg(S)=0$ for all $S\in\BB$
with $\pi(S)=0$. Absolute continuity implies weak absolute continuity but
(unlike in the case of countably additive measures), it is not equivalent to
it. But we have the following fact.

\begin{lemma}\label{LEM:ABS-CONT}
An increasing submodular setfunction $\fg$ on a sigma-algebra $(J,\BB)$ with
$\fg(\emptyset)=0$ is absolutely continuous with respect to a measure $\pi$ if
and only if it is continuous from above and weakly absolutely continuous with
respect to $\pi$.
\end{lemma}

\begin{proof}
First, suppose that $\fg$ is absolutely continuous with respect to a countably
additive measure $\pi$. Weak absolute continuity with respect to $\pi$ is
trivial. To prove that $\fg$ is continuous from above, let $A_i\in\BB$,
$A_1\supseteq A_2\supseteq\ldots$, and $\cap_i A_i=\emptyset$. Given $\eps>0$,
there is a $\delta>0$ such that $|\fg(S)|<\eps$ for every $S\in\BB$ with
$\pi(S)<\delta$. Then $\pi(A_n)\to 0$ (since $\pi$ is countably additive), so
$\pi(A_n)<\delta$ if $n$ is large enough. Hence $\fg(A_n) <\eps$ for every $n$
that is large enough.

Second, suppose that $\fg$ is continuous from above and weakly absolutely
continuous with respect to $\pi$, and assume, by way of contradiction, that
$\fg$ is not absolutely continuous with respect to $\pi$. Then there is a
sequence of sets $A_n\in\BB$ such that $\pi(A_n)\to 0$ but $\fg(A_n)>c$ for
some $c>0$. We may assume that $\pi(A_n)<2^{-n}$. Let $B_n=\cup_{m\ge n} A_m$,
then $\pi(B_n)<2^{1-n}$ and $\fg(B_n)>c$. Then $\pi(\cap_nB_n)=0$ (since $\pi$
is countably additive) but $\fg(\cap_nB_n)>c$ (since $\fg$ is continuous from
above). This contradicts the weak absolute continuity of $\fg$.
\end{proof}

\subsection{Continuity of the Choquet extension}

Lemma \ref{LEM:WHPHI-CONT} implies that if $\fg$ is any setfunction with
bounded variation, $g,f_n\in\Bd$ and $f_n\to g$ uniformly, then
$\wh\fg(f_n)\to\wh\fg(g)$. Under submodularity conditions on $\fg$, we can
relax the uniform convergence condition. (Under smoothness conditions for
$\fg$, we get further improvements.)

Let $f_n$ $(n=1,2,\ldots)$ and $g$ be measurable functions $J\to[0,1]$. We say
that {\it $f_n \to g$ in $\fg$-measure}, if $\fg\{x:~|f_n(x)-g(x)|>\eps\}\to 0$
as $n\to\infty$ for every $\eps>0$. We say that {\it $f_n \to g$ $\fg$-almost
everywhere}, if $\fg\{x:~f_n(x)\not\to g(x)\}=0$.

\begin{lemma}\label{LEM:LEBESGUE}
Let $\fg$ be an increasing submodular setfunction on $\BB$ with
$\fg(\emptyset)=0$, let $f_n,g\in\Bd$ with $\|f_n\|,\|g\|\le 1$, and let
$f_n\to g$ in $\fg$-measure. Then $\wh{\fg}(f_n)\to\wh\fg(g)$ $(n\to\infty)$.
\end{lemma}

\begin{proof}
Let $h_n=|f_n-g|$, then $0\le h_n\le 2$, and $f_n\le g+h_n$. Hence by Theorem
\ref{THM:SUBMOD-UNCROSS2}, we have
$\wh\fg(f_n)\le\wh\fg(g+h_n)\le\wh\fg(g)+\wh\fg(h_n)$. Let $\eps>0$, then
\begin{align*}
\wh\fg(h_n) &= \intl_0^2 \fg\{h_n\ge t\}\,dt \le  \intl_0^\eps \fg\{h_n\ge t\}\,dt +  \intl_\eps^2 \fg\{h_n\ge t\}\,dt\\
&\le \eps\|\fg\| + 2\fg\{h_n\ge \eps\}.
\end{align*}
The last term is less than $\eps$ if $n$ is sufficiently large, so
$\wh\fg(h_n)\le 3\eps$ and $\wh\fg(f_n)\le\wh\fg(g)+3\eps$. The reverse
inequality follows similarly.
\end{proof}

\begin{lemma}\label{LEM:C-CONV}
Let $\fg$ be an increasing submodular setfunction on $\BB$ continuous from
above with $\fg(\emptyset)=0$. Let $f_n,g\in\Bd$ with $\|f_n\|,\|g\|\le 1$, and
assume that $f_n\to g$ $\fg$-almost everywhere. Then
$\wh{\fg}(f_n)\to\wh\fg(g)$ $(n\to\infty)$.
\end{lemma}

\begin{proof}
By Lemma \ref{LEM:LEBESGUE}, it suffices to prove that $f_n\to g$ in
$\fg$-measure. Let $Z=\{x\in J:~f_n(x)\not\to g(x)\}$ and $\eps>0$. Define
$X_{n,\eps}=\{|f_n-g|>\eps\}$ and $Y_{n,\eps}=\cup_{m\ge n} X_{m,\eps}$. We
claim that $\cap_n Y_{n,\eps}=Z$. Indeed, if $x\in\cap_nY_{n,\eps}$, then there
are infinitely many indices $m$ such that $x\in X_{m,\eps}$, so
$|f_m(x)-g(x)|>\eps$ for infinitely many indices $m$, and hence $f_m(x)\not\to
g(x)$.

By continuity from above, this implies that $\fg(Y_{n,\eps})\to \fg(Z)=0$ as
$n\to\infty$. Since $X_{n,\eps}\subseteq Y_{n,\eps}$, the monotonicity of $\fg$
implies that $\fg(X_n)\to 0$, and so $f_n\to g$ in $\fg$-measure.
\end{proof}

\subsection{Majorizing measures and strong boundedness}\label{SSEC:SBOUNDED}

Let $\fg$ be an increasing subadditive setfunction on a sigma-algebra $(J,\BB)$
with $\fg(\emptyset)=0$. We define $\fgx:~\BB\to\R\cup\{\infty\}$ by
\begin{equation}\label{EQ:MU-SUP-FG}
\fgx(X) = \sup \sum_{i=1}^n \fg(X_i),
\end{equation}
where $\{X_1,\ldots,X_n\}$ ranges over all partitions of $X$ into sets in
$\BB$. Clearly $\fgx$ is increasing (but it may take infinite values). Let us
also note that in the definition of $\fgx$, we could allow countable
partitions. Indeed, for any countable partition $X=\cup_{i=1}^\infty X_i$, we
have
\[
\fgx(X)\ge \fgx(X_1\cup \ldots\cup X_n)\ge \sum_{i=1}^n \fg(X_i),
\]
and letting $n\to\infty$, we get
\[
\fgx(X)\ge \sum_{i=1}^\infty \fg(X_i).
\]
This means that including countable partitions does not increase the supremum
in \eqref{EQ:MU-SUP-FG}.

\begin{lemma}\label{LEM:UPENV}
The setfunction $\fgx$ is additive. If $\fg$ is continuous from below, then
$\fgx$ is countably additive.
\end{lemma}

\begin{proof}
We prove the second assertion. The proof of the first one is almost the same,
except that continuity from below is not needed because of the finiteness of
the partition.

Let $X=\cup_{i=1}^\infty X_i$ be a partition of $X$ ($X,X_i\in\BB$). Let
$\eps>0$, and let $X_i=\cup_{j=1}^\infty X_{i,j}$ partition of $X_i$ such that
$X_{i,j}\in\BB$) and
\[
\fgx(X_i) \le \sum_{j=1}^\infty \fg(X_{i,j}) + \frac\eps{2^i}.
\]
Then $\{X_{i,j}:~i,j\in\Nbb\}$ is a partition of $X$ into sets in $\BB$, and
hence
\begin{equation}\label{EQ:FGXBIG}
\fgx(X)\ge \sum_{i,j} \fg(X_{i,j}) \ge \sum_i \Big(\fgx(X_i)-\frac\eps{2^i}\Big) = \sum_i\fgx(X_i) - \eps.
\end{equation}

To show the reverse inequality, let $\eps>0$, and let $X=Z_1\cup\ldots\Z_n$ be
a measurable partition of $X$ such that
\[
\fgx(X)\le \sum_{i=1}^n \fg(Z_i) + \eps.
\]
Continuity from below implies, by \eqref{EQ:OUTER}, that
\begin{equation}\label{EQ:FG-LC}
\fg(Z_i) \le \sum_{j=1}^\infty \fg(Z_i\cap X_j).
\end{equation}
Hence
\begin{equation}\label{EQ:FGXSMALL}
\fgx(X)\le \sum_{i=1}^n \fg(Z_i) + \eps = \sum_{i=1}^n\sum_{j=1}^\infty \fg(Z_i\cap X_j) + \eps \le \sum_i\fgx(X_i) + \eps.
\end{equation}
Since $\eps>0$ was arbitrary, \eqref{EQ:FGXBIG} and \eqref{EQ:FGXSMALL} prove
the countable additivity of $\fgx$.
\end{proof}

The previous considerations are most relevant when $\fgx$ is finite. We call
such a setfunction $\fg$ {\it strongly bounded}. This means that there is a
$C\in\R$ such that $\sum_i |\fg(A_i)| \le C$ for every family of disjoint sets
$A_1,\ldots,A_n\in\BB$. It is clear that every nonnegative charge has this
property. The following lemma establishes a certain converse.

\begin{lemma}\label{LEM:ST-BOUNDED}
A nonnegative bounded subadditive setfunction $\fg$ with $\fg(\emptyset)=0$ is
strongly bounded if and only if there is a charge $\alpha\in\ba_+(\BB)$ such
that $\fg\le\alpha$.
\end{lemma}

\begin{proof}
The necessity of the condition follows by Lemma \ref{LEM:UPENV}, considering
$\alpha=\fgx$. Sufficiency follows by the following computation: For for every
family of disjoint sets $A_1,\ldots,A_n\in\BB$, we have
\[
\sum_{i=1}^n \fg(A_i) \le \sum_{i=1}^n \alpha(A_i) = \alpha\big(\bigcup_i A_i\big) \le \alpha(J).
\]
\end{proof}

If a nonnegative submodular setfunction $\fg$ is continuous from below and
strongly bounded, then it bounded from above by a (finite) countably additive
measure $\mu$; this implies that $\fg$ is also continuous from above. It also
follows that $\fg$ is absolutely continuous with respect to the measure $\fgx$.

\section{Applications of smoothness properties}

In this section we show how previous results can be improved or applied under
smoothness conditions.

\subsection{Countably additive minorizing measures}

In Theorem \ref{THM:FIN-ADD-LAT} and its corollaries, we have constructed
minorizing charges. We can do better for setfunctions continuous from above,
for which minorizing measures are automatically countably additive:

\begin{lemma}\label{LEM:MINOR-CONT}
Let $\fg$ be a nonnegative setfunction with $\fg(\emptyset)=0$  continuous from
above, and let $\alpha\in\matp_+(\fg)$. Then $\alpha$ is countably additive.
\end{lemma}

\begin{proof}
Let $B_i\in\BB$ ($i=1,2,\ldots$) be disjoint sets, and define $A_n=\cup_{i\ge
n}B_i$. Then by the finite additivity of $\alpha$,
\[
\alpha(A_1)=\sum_{i=1}^{n-1} \alpha(B_i) + \alpha(A_n).
\]
Since $\cap_n A_n=\emptyset$, continuity of $\fg$ from above implies that
$\fg(A_n)\to 0$ as $n\to\infty$. Since $\alpha\le\fg$, it follows that
$\alpha(A_n)\to 0$. So letting $n\to\infty$, we get that
$\alpha(A_1)=\sum_{i=1}^\infty \alpha(B_i)$.
\end{proof}

\begin{corollary}\label{COR:FIN-ADD2}
Let $\fg$ be an increasing submodular setfunction continuous from above. Let
$\SS\subseteq\BB$ be a chain of sets. Then there is a countably additive
measure $0\le\alpha\le\fg$ on $\BB$ such that $\alpha(S)=\fg(S)$ for $S\in\SS$.
If $\SS$ generates $\BB$ as a sigma-algebra, then $\alpha$ is uniquely
determined.
\end{corollary}

The natural full chain $\{[0,a):~0<a\le 1\} \cup \{[0,a]:~0\le a\le 1\}$ of
Borel subsets of $[0,1]$ is an example when the last conclusion applies.

\begin{proof}
The charge $\alpha$ constructed in Theorem \ref{THM:FIN-ADD-LAT} is countably
additive by Lemma \ref{LEM:MINOR-CONT}. To prove uniqueness of $\alpha$ it
suffices to note that $\SS$ is a set-algebra that generates the Borel sets as a
sigma-algebra, and the restriction of $\fg$ to $\SS$ is a finite countably
additive measure, so its extension to all Borel sets is unique by
Caratheodory's Theorem.
\end{proof}

\subsection{Lopsided Fubini Theorem}

Assuming continuity of $\fg$, we can prove the following generalization of
Theorem \ref{THM:SUBMOD-UNCROSS2}. The continuity condition cannot be omitted,
even if $\fg$ is a charge, as shown by Example \ref{EXA:NOT-FUBINI2} in the
Appendix.

\begin{theorem}\label{THM:SUBMOD-UNCROSS3}
Let $(I,\AA,\lambda)$ and $(J,\BB,\pi)$ be probability spaces. Let $\fg$ be an
increasing submodular setfunction on $(J,\BB)$ absolutely continuous with
respect to $\pi$, and let $F:~ J\times I\to\R$ be a bounded measurable
function. Define $ F_y(x)=F(x,y)$ and $g(x)=\int_I F(x,y)\,d\lambda(y)$. Then
\begin{equation}\label{EQ:MAIN1}
\wh\fg(g) \le \intl_I \wh\fg(F_y)\,d\lambda(y).
\end{equation}
\end{theorem}

Using the integral notation, we can write this inequality as
\[
\intl_J \intl_I F(x,y)\,d\lambda(y)\,d\fg(x) \le \intl_I \intl_J F(x,y)\,d\fg(x)\,d\lambda(y).
\]
So inequality \eqref{EQ:MAIN1} is a ``lopsided'' version of Fubini's Theorem.

\begin{proof}
We may assume that $0\le F\le 1$. Let $\XX$ be the set of pairs $(x,\yb)$ such
that $x\in J$, $\yb=(y_1,y_2,\ldots)\in I^\Nbb$, and
\begin{equation}\label{EQ:FT-CONV}
\frac1n\sum_{i=1}^n F_{y_i}(x) \to g(x).
\end{equation}
For every $x\in J$, this happens to $\lambda^\Nbb$-almost all sequences $\yb$.
By Fubini's Theorem, for almost all sequences $\yb$, \eqref{EQ:FT-CONV} holds
for $\pi$-almost all $x\in J$.

Let $\yb=(y_1,y_2,\ldots)$ be an infinite sequence of independent uniform
random points in $I$ from the distribution $\lambda$, and let
\[
f_n =\frac1n\sum_{i=1}^n F_{y_i}.
\]
Then
\begin{equation}\label{EQ:FIN-INEQ}
\wh\fg(f_n) \le \frac1n\sum_{i=1}^n\wh\fg(F_{y_i})
\end{equation}
by Theorem \ref{THM:SUBMOD-UNCROSS2}. By the Law of Large Numbers,
\begin{equation}\label{EQ:WHPHI-CONV}
\frac1n\sum_{i=1}^n\wh\fg(F_{y_i})\to \intl_I \wh\fg(F_y)\,d\lambda(y)\qquad(n\to\infty)
\end{equation}
almost surely. So we have
\begin{equation}\label{EQ:INEQ-N}
\limsup_{n\to\infty} \wh\fg(f_n) \le \intl_I \wh\fg(F_y)\,d\lambda(y)
\end{equation}
for almost all $\yb$. In addition, we have \eqref{EQ:FT-CONV} for almost all
pairs $(x,\yb)$. So we can fix a sequence $\yb$ for which $f_n\to g$
$\pi$-almost everywhere, and \eqref{EQ:INEQ-N} holds. Since $\fg$ is absolutely
continuous with respect to $\pi$, it follows that $f_n\to g$ $\fg$-almost
everywhere. By Lemma \ref{LEM:ABS-CONT} the setfunction $\fg$ is continuous
from above, and so Lemma \ref{LEM:C-CONV} implies that $\wh\fg(f_n) \to
\wh\fg(g)$ as $n\to\infty$. Combined with \eqref{EQ:INEQ-N}, this proves the
theorem.
\end{proof}

\subsection{Matchings and coupling measures}\label{SEC:MATCH}

An important area of applications of matroids to graph theory is matching
theory. Let $(J_i,\BB_i,\pi_i)$ be probability spaces, and let
$E\in\BB_1\times\BB_2$ be a set. We can think of $(J_1, J_2, E)$ as a bigraph
$G$ with bipartition $J_1\cup J_2$ and edge set $E$. Recall that for
$X\in\BB_1$, we denote by $E(X)\subseteq J_2$ the set of neighbors of $X$ in
$G$. We would like to find a measure $\alpha$ on
$\AA=2^E\cap(\BB_1\times\BB_2)$ whose marginals are $\pi_1$ and $\pi_2$.

If $J_1$ and $J_2$ are finite, $|J_1|=|J_2|$, and $\pi_1,\pi_2$ are uniform
distributions, then this just the Fractional Perfect Matching Problem. In the
general case, let $(J,\BB,\pi)$ be a probability space. A nonnegative measure
on $\BB^2$ with marginals $\pi$ will be called a {\it coupling measure}. (If
$J=[0,1]$ and $\pi$ is the Lebesgue measure, then it is also called a {\it
doubly stochastic measure}, or a {\it permuton}; the last name refers to the
fact that these measures represent limits of permutations of finite sets; see
\cite{HKMRS,GGKK}.)

We are interested in coupling measures (both countably and finitely additive
ones) supported on a given set. The necessary condition generalizes the
condition in the Frobenius--Kőnig Theorem.

\begin{theorem}\label{THM:COUPLING}
{\rm(a)} A finitely additive coupling measure supported on $E\in\BB\times\BB$
exists if and only if $\lambda(X)\le \lambda(E(X))$ for every $X\in\BB$.

\smallskip

{\rm(b)} If, in addition, $(J,\BB)$ is the sigma-algebra of Borel sets of a
compact Hausdorff space, and $E$ is closed, then even a countably additive
coupling measure supported on $E$ exists.
\end{theorem}

Before the proof, some remarks. The set $E(X)$ is not necessarily Borel, but
(as remarked in Example \ref{EXA:TRANSVERSAL}), it is Lebesgue measurable, and
so $\lambda(E(X))$ is well-defined. Part (b) of the Theorem is well known: it
follows by a theorem of Strassen \cite{Str}. It is also a rather
straightforward generalization of Theorem 2.5.2 in \cite{Mbook}. Part (a) is
closely related to recent work of Rigo \cite{Rigo}. We give a proof of (a) as
an application of the Submodular Intersection Theorem \ref{THM:INTERSECT}, and
derive (b) from (a).

Part (b) remains valid under the weaker hypothesis that $[0,1]^2\setminus E$ is
the union of a countable number of product sets $A_i\times B_i$
$(A_i,B_i\in\BB)$ (Proposition 3.8 of Kellerer \cite{Kell}).

\begin{proof*}{Theorem \ref{THM:COUPLING}}
The necessity of the condition is easy. Let $\alpha$ be a coupling measure
supported on $E$. For every $X\in\BB$, the set $X\times (J\setminus E(X))$ is
disjoint from $E$, and hence $\alpha(X\times (J\setminus E(X)))=0$. Hence
\[
\lambda(X) = \alpha(X\times J)
= \alpha(X\times E(X)) \le \alpha(J\times E(X))=\lambda(E(X)).
\]

To prove the ``if'' part of (a), consider two setfunctions on $\AA$: for
$U\in\AA$, let $\Pi_i(U)$ denote its projection onto the $i$-th coordinate, and
define $\fg_i(U)=\lambda(\Pi_i(U))$ ($i=1,2$). The setfunctions $\fg_i$ are
both increasing and submodular, as shown in Example \ref{EXA:TRANSVERSAL}.
Furthermore, for any $U\in \AA$,
\[
\fg_1(U) + \fg_2(E\setminus U) = \lambda(\Pi_1(U))+\lambda(\Pi_2(E\setminus U)).
\]
Let $X=[0,1]\setminus\Pi_1(U)$, then by for any $x\in X$ and $(x,y)\in E$ we
have $(x,y)\notin U$, and hence $y\in \Pi_2(E\setminus U)$. So $E(X)\subseteq
\Pi_2(E\setminus U)$, and hence by the hypothesis in (a),
\[
\lambda(\Pi_2(E\setminus U)) \ge \lambda(E(X))\ge \lambda(X) = 1-\lambda(\Pi_1(U)).
\]
In other words, $\fg_1(U) + \fg_2(E\setminus U)\ge 1$. Taking the infimum on
the left, we get that
\[
(\fg_1\land\fg_2)(E) \ge 1.
\]
So by Theorem \ref{THM:INTERSECT}, there is a charge
$\alpha\in\matp_+(\fg_1)\cap\matp(\fg_2)$ such that $\alpha(E)=1$. The measure
$\alpha$ is defined on Borel subsets of $E$, but we can extend it to all sets
$X\in\BB\times\BB$ by setting $\alpha(X)=\alpha(X\cap E)$.

By construction, $\alpha$ is supported on $E$. To show that $\alpha$ is a
coupling measure, let $X\in\BB$. Then $\alpha(X\times J)\le \fg_1(X\times
J)=\lambda(X)$, and similarly $\alpha(X^c\times J)\le\lambda(X^c)$. But
$1=\alpha(J\times J) = \alpha(X\times J) + \alpha(X^c\times J) \le \lambda(X) +
\lambda(X^c) = \lambda(J)=1$, which proves that the first marginal of $\alpha$
is $\lambda$. The same conclusion follows for the second marginal similarly.

To prove the ``if'' part of (b), consider a charge $\alpha$ constructed above,
and let $\overline{\alpha}$ be the countably additive measure constructed in
Proposition \ref{PROP:SMOOTHING}. Let $f(x)$ denote the distance of $x\in
I\times I$ from $E$, then $f$ is continuous and positive on $E^c$. Hence
\[
\intl_{E^c} f\,d\overline{\alpha}= \intl_{E^c} f\,d\alpha =0,
\]
which implies that $\overline{\alpha}$ is supported on $E$.
\end{proof*}

\begin{remark}\label{REM:LIND}
The Birkhoff--von Neumann Theorem asserts that in the finite case, the extremal
points of the set of coupling measures are just the perfect matchings. The
question of determining these extremal points in the infinite case (at least in
the Borel case) is open. Complicated extremal points exist (see e.g.~Losert
\cite{Los}), and an interesting characterization was given by Lindenstrauss
\cite{Lind}, but no direct generalization of the Birkhoff--von Neumann Theorem
seems to be known.
\end{remark}

\section{Björner distance and flats of submodular
setfunctions}\label{SEC:BJ-DIST}

The results in Section \ref{SSEC:EXCHANGE} represent an attempt to generalize
the notions of independent sets and bases from matroids to submodular
setfunctions based on charges. Here we generalize the notion of flats (closed
subsets) of matroid theory, using a simple but very useful distance of subsets
defined by an increasing submodular setfunction.

\subsection{Definition}

The following definition is motivated by a construction of Björner
\cite{Bj87,Bj19}. Let $\fg$ be a increasing submodular setfunction on a
set-algebra $(J,\BB)$. For two sets $X,Y\in\BB$, we define their {\it Björner
distance} by
\begin{equation}\label{EQ:D-DEF}
d(X,Y)=d_\fg(X,Y)=2\fg(X\cup Y)-\fg(X)-\fg(Y).
\end{equation}
In particular,
\begin{equation}\label{EQ:D-SUB}
d(X,Y)=\fg(Y)-\fg(X)\quad\text{if}\quad X\subseteq Y.
\end{equation}
Hence for any $X,Y\in\BB$, we have
\begin{equation}\label{EQ:D-SPLIT}
d(X,Y)=d(X,X\cup Y)+d(Y,X\cup Y).
\end{equation}
If $\fg$ is a measure, then $d(X,Y)=\fg(X\triangle Y)$ is the $L^1$-distance of
$\one_X$ and $\one_Y$. The metric $d_\fg$ does not change if a constant is
added to $\fg$. So we may assume that $\fg(\emptyset)=0$, and we will do so in
the rest of this section. It follows that $\fg(X)=d(X,\emptyset)$.

\begin{lemma}\label{LEM:BJ-MAETRIC}
The function $d$ is a pseudometric. If $\fg$ is strictly increasing, then $d$
is a metric.
\end{lemma}

\begin{proof}
Nonnegativity of $d$ is implied by the monotonicity of $\fg$. The equations
$d(X,X)=0$ and $d(X,Y)=d(Y,X)$ are trivial. The triangle inequality follows by
an easy computation:
\begin{align*}
d(X,Y)&+d(Y,Z)-d(X,Z) = 2\fg(X\cup Y)-\fg(X)-\fg(Y)\\
&+ 2\fg(Y\cup Z)-\fg(Y)-\fg(Z) - (2\fg(X\cup Z)-\fg(X)-\fg(Z))\\
&=2\fg(X\cup Y) + 2\fg(Y\cup Z) - 2\fg(X\cup Z) - 2\fg(Y).
\end{align*}
Here $X\cup Z\subseteq (X\cup Y)\cup(Y\cup Z)$ and $Y\subseteq (X\cup
Y)\cap(Y\cup Z)$, and so
\begin{align*}
d(X,Y)&+d(Y,Z)-d(X,Z)\ge 2\fg(X\cup Y) + 2\fg(Y\cup Z)\\
&- 2\fg((X\cup Y)\cup(Y\cup Z)) - 2\fg((X\cup Y)\cap(Y\cup Z)) \ge 0
\end{align*}
by submodularity. The second assertion follows trivially from the first.
\end{proof}

It is straightforward to check that if $\fg$ and $\psi$ are increasing
submodular setfunctions, then $d_\psi\le d_\fg$ if and only if the pair
$(\fg,\psi)$ is diverging. Since the pair $(\fg,\fg(A\cup.))$ is diverging for
any $A\in\BB$, it follows that the map $X\mapsto X\cup A$ is contractive:
\begin{align}\label{EQ:UNION-CONTR}
d(X\cup A,Y\cup A) \le d(X,Y).
\end{align}

\begin{remark}\label{REM:GRAPH-METRIC}
The Björner distance is a weighted graph distance, where the underlying graph
is the (undirected) comparability graph of sets in $\BB$, and the weight of an
edge $(X,Y)$ is $|\fg(X)-\fg(Y)|$. This can be proved similarly as Lemma
\ref{LEM:SUBMOD-FINVAR}.
\end{remark}

\subsection{Roofs and flats}

For the rest of this section, let us assume that $(J,\BB)$ is a sigma-algebra
and $\fg$, an increasing submodular setfunction on $\BB$. The distance function
$d_\fg$ is not a metric in general, because two different sets can have
distance $0$. Two sets $X,Y\in\BB$ have $d(X,Y)=0$ if and only if
$\fg(X)=\fg(Y)=\fg(X\cup Y)$. If $X\subseteq Y$, then this condition simplifies
to $\fg(X)=\fg(Y)$. The condition $d(X,Y)=0$ defines an equivalence relation,
which we denote by $X\equiv_\fg Y$, or simply by $X\equiv Y$. The classes of
this equivalence relation will be called {\it roofs}. Contracting the
equivalence classes to single elements, we get a metric space.

If $\fg=\lambda$ is the Lebesgue measure, then this equivalence is the standard
identification of sets $X$ and $Y$ if $\lambda(X\triangle Y)=0$. In the case of
a matroid $(E,r)$, two sets are equivalent if and only if they span the same
flat, and the roof corresponding to this flat consists of all sets spanning
this flat. The flat consists of all elements of $E$ whose equivalence class is
under this roof.

It is clear that for every roof $\RR$, all sets $X\in\RR$ have the same
$\fg(X)$, which we can denote by $\fg(\RR)$. It is also clear that every roof
is closed under (finite) union. In the finite case of matroids, every roof
contains a unique largest set, and it is closed under all unions. This does not
remain true in the infinite case, even for countable unions and under
smoothness conditions. If $\fg$ is the Lebesgue measure on $[0,1]$, then all
null-sets among Borel sets form a roof, but there is no largest null-set. For
the setfunction $\one(X~\text{is infinite})$ (cf. Example \ref{EXA:IDEAL})
there are two roofs (finite sets and infinite sets), but finite sets are not
closed under countable union. However, the following smoothness condition
helps:

\begin{lemma}\label{LEM:COUNT-UNION}
If $\fg$ is an increasing submodular setfunction continuous from below, then
every roof $\RR$ is closed under countable union.
\end{lemma}

\begin{proof}
Let $X_1,X_2,\ldots\in\RR$ and $X=X_1\cup X_2\cup\ldots$. Since $\RR$ is closed
under finite union, we may replace $X_n$ by $X_1\cup\ldots\cup X_n$; in other
words, we may assume that $X_1\subseteq X_2\subseteq\ldots$. Then $\fg(X_n)\to
\fg(X)$ by continuity from below. Since $\fg(X_n)=\fg(\RR)$, it follows that
$\fg(X)=\fg(X_n)=\fg(\RR)$. Since $X_n\subseteq X$, this implies that
$X_n\equiv X=0$, and so $X\in\RR$.
\end{proof}

For any increasing submodular setfunction $\fg$, we can define a partial order
on roofs as follows. Let $\RR$ and $\QQ$ be two roofs. Let $\RR\le\QQ$ mean
that there is an $A\in\RR$ and $B\in\QQ$ such that $A\subseteq B$. Let us note
right away that if this holds, then for every $A'\in\RR$ there is a $B'\in\QQ$
such that $A'\subseteq B'$. Indeed, let $B'=B\cup A'$. To prove that $B\equiv
B'$, it suffices to show that $\fg(B')=\fg(B)$. Clearly $\fg(B')\ge\fg(B)$. On
the other hand,
\[
\fg(B')=\fg(B\cup (A\cup A')) \le \fg(B)+\fg(A\cup A')-\fg((A\cup A')\cap B)
\]
by submodularity. Here $\fg(A\cup A')=\fg(A)$ since $A\equiv A'$, but also
$\fg((A\cup A')\cap B)=\fg(A)$ since $A\subseteq (A\cup A')\cap B \subseteq
A\cup A'$. Thus we get that $\fg(B')\le \fg(B)$, and so $\fg(B')=\fg(B)$.

This argument also shows that if $\RR$ and $\QQ$ are two roofs such that
$\QQ\le\RR$ and $\RR\le\QQ$, then $\RR=\QQ$.

Motivated by the case of matroids, we define a {\it flat} associated with a
roof $\RR$ as the union of all roofs $\QQ$ such that $\QQ\le\RR$. Every flat is
a set ideal in $\BB$, but not necessarily a principal ideal, by the example
above. Nevertheless, the flat determines its roof: if two roofs $\RR$ and $\QQ$
define the same flat, then $\RR\le\QQ$ and $\QQ\le\RR$ by the definition of
flats,, and hence $\RR=\QQ$.

Inequality \eqref{EQ:UNION-CONTR} implies that if $A_1\equiv A_2$, then
$A_1\cup B\equiv A_2\cup B$ for every $B$. Applying this twice, we get that if
$A_1\equiv A_2$ and $B_1\equiv B_2$, then $A_1\cup B_1\equiv A_2\cup B_2$. In
other words, for two roofs $\RR$ and $\QQ$, if $A\in\RR$ and $B\in\QQ$, then
the roof containing $A\cup B$ depends only on $\RR$ and $\QQ$, and we can
denote it by $\RR\lor\QQ$. It is trivial to check that this operation is
commutative, associative and idempotent.

So we have a join semilattice on the roofs. It is also easy to verify that
$\RR\le\QQ$ if and only if $\RR\lor\QQ=\QQ$, so the semilattice structure
conforms with the partial order. However, this semilattice is not a lattice in
general, only under a smoothness condition:

\begin{lemma}\label{LEM:LATTICE}
If $\fg$ is an increasing submodular setfunction continuous from below, then
the semilattice structure on the roofs is a lattice.
\end{lemma}

\begin{proof}
We want to prove that if $\RR$ and $\QQ$ are two roofs, then there is a unique
largest roof $\HH$ such that $\HH\le\RR$ and $\HH\le\QQ$. This will allow us to
define the meet operation by $\RR\land\QQ=\HH$.

Let $c=\sup\{\fg(X\cap Y):~X\in\RR, Y\in\QQ\}$. We claim that this supremum is
a maximum, i.e., it is attained. Let $X_1,X_2,\ldots\in\RR$ and
$Y_1,Y_2,\ldots\in\QQ$ be chosen so that $\fg(X_n\cap Y_n)\to c$ as
$n\to\infty$. Let $X=\cup_n X_n$ and $Y=\cup_n Y_n$, then $X\in\RR$ and
$Y\in\QQ$ by Lemma \ref{LEM:COUNT-UNION}. By monotonicity and the definition of
$c$, $\fg(X_n\cap Y_n)\le \fg(X\cap Y) \le c$. Since $\fg(X_n\cap Y_n)\to c$,
this implies that $\fg(X\cap Y)=c$.

Let $\HH$ be the roof containing $X\cap Y$. Then by the definition of the
partial order, $\HH\le\RR$ and $\HH\le\QQ$. Let $\KK$ be any other roof such
that $\KK\le\RR$ and $\KK\le\QQ$, and let $Z\in\KK$. Then there are sets
$U\in\RR$ and $W\in\QQ$ such that $Z\subseteq U$ and $Z\subseteq W$. Then
$Z\subseteq U\cap W$, which implies that $\fg(Z)\le\fg(U\cap W)\le c$. Let
$X'=U\cup X$ and $Y'=W\cup Y$, then $X'\in\RR$, $Y'\in\QQ$, and
\[
c=\fg(X\cap Y)\le \fg(X'\cap Y')\le c.
\]
Thus $\fg(X'\cap Y')= c = \fg(X\cap Y)$. Since $X\cap Y\subseteq X'\cap Y'$,
this implies that $X'\cap Y'\in\HH$. Since $Z\subseteq X'\cap Y'$, it follows
that $\KK\le\HH$. So $\HH$ is the unique largest roof below $\RR$ and $\QQ$.
\end{proof}

\subsection{Completeness}

\begin{theorem}\label{THM:METRIC-COMPLETE}
If $\fg$ is an increasing submodular setfunction continuous from above, then
the pseudometric $d=d_\fg$ is complete.
\end{theorem}

\begin{proof}
Let $X_1,X_2,\ldots$ be a Cauchy sequence of sets in $\BB$. Set
\[
a(n)=\sup_{i,j\ge n} d(X_i,X_j),
\]
then $a(n)\to0$ as $n\to\infty$.

We claim that for any $1\le n_1\le\ldots\le n_r$,
\begin{equation}\label{EQ:XNXK}
d(X_{n_1},X_{n_1}\cup \ldots \cup X_{n_r}) \le a(n_1)+\ldots +a(n_{r-1}).
\end{equation}
Indeed, by \eqref{EQ:D-SUB} and \eqref{EQ:UNION-CONTR},
\begin{align*}
d(X_{n_1},\,&X_{n_1}\cup \ldots \cup X_{n_r}) = \sum_{j=0}^r d(X_{n_1}\cup \ldots \cup X_{n_{j+1}},X_{n_1}\cup \ldots \cup X_{n_j})\\
&\le\sum_{j=1}^{r-1} d(X_{n_j}\cup X_{n_{j+1}}, X_{n_j})
\le \sum_{j=1}^{r-1} d(X_{n_{j+1}}, X_{n_j})
\le\sum_{j=1}^{r-1} a(n_j).
\end{align*}

Let $n_1<n_2<\ldots $ be chosen so that $a(n_k)\le 2^{-k}$. Let $Y_k=\cup_{j\ge
k} X_{n_j}$ and $Z=\cap_k Y_k$. By \eqref{EQ:XNXK},
\[
d(X_{n_k},X_{n_k}\cup \ldots \cup X_{n_{k+r}}) \le 2^{-k}+\ldots+2^{-k-r}< 2^{1-k},
\]
and so by continuity from below (implied by continuity from above),
\[
d(X_{n_k},Y_k) \le 2^{1-k}.
\]
Furthermore, continuity from above implies that $\fg(Y_{n_k})\to\fg(Z)$ as
$k\to\infty$. Thus
\[
d(X_{n_k},Z)\le d(X_{n_k},Y_k)+d(Y_k,Z) \le 2^{1-k}+d(Y_k,Z) \to 0\qquad(k\to\infty).
\]
Now let $m\to\infty$, and let $k$ be the largest $j$ for which $n_j\le m$. Then
\[
d(X_m,Z)\le d(X_m,X_{n_k})+d(X_{n_k},Z) \le 2^{1-k}+d(X_{n,k},Z) \to 0,
\]
This shows that $X_m\to Z$ in the metric $d$, proving the theorem.
\end{proof}

\section{Strong submodularity}\label{SEC:STRONG-SUBMOD}

\subsection{Definition and examples}

Choquet \cite{Choq} introduced a sequence of inequalities generalizing
submodularity. Let $\FF$ be a lattice family of sets. For
$A_0,\ldots,A_n\in\FF$ and $K\subseteq\{0,\ldots,n\}$, set $A_K=\cup_{i\in
K}A_i$. For $n\ge 1$ and $A_0,\ldots,A_n\in\FF$, consider the following
property:
\begin{equation}\label{EQ:SSUBMOD}
\sum_{K\subseteq\{1,\ldots,n\}} (-1)^{|K|} \fg(A_0\cup A_K)\le 0.
\end{equation}
For $n=1$, this means that $\fg$ is increasing; for $n=2$, this means that
$\fg$ is increasing and submodular. We call a setfunction {\it strongly
submodular}, if it satisfies inequality \eqref{EQ:SSUBMOD} for every $n\ge1$
and every family $A_0,\ldots,A_n\in\FF$. The property is not changed when we
scale $\fg$ by any positive number, and add a constant to it. For example, we
can make $\fg(\emptyset)=0$ (if this is useful); then $\fg$ is nonnegative.
(This is why we don't require \eqref{EQ:SSUBMOD} to hold for $n=0$.) It is easy
to see that every modular setfunction on a lattice family is strongly
submodular.

Every increasing modular setfunction $\alpha$ is strongly submodular. This can
be seen by shifting it to satisfy $\alpha(\emptyset)=0$, when it becomes a
charge. Let $C_1,\ldots,C_r$ be the atoms of the (finite) set-algebra generated
by the sets $A_0$ and $A_i$ $(i\in N=\{1,\ldots,n\})$, and let us compute the
contribution of an atom $C_j$ to the left hand side of \eqref{EQ:SSUBMOD}. Let
$M=\{i:~C_j\subseteq A_i\}$. If $0\in M$, then this contribution is
\begin{align*}
\alpha(C_j) \sum_{K\subseteq N} (-1)^{|K|}=0
\end{align*}
(here we need that $n\ge 1$). If $0\notin M$, then
\begin{align*}
\alpha(C_j) \sum_{K\subseteq N \atop K\cap M\not=\emptyset} (-1)^{|K|}
&= - \alpha(C_j) \sum_{K\subseteq N\setminus M} (-1)^{|K|}.
\end{align*}
This sum is zero unless $N=M$, in which case it is $-\alpha(C_J)\le 0$.

Clearly, for a strongly submodular setfunction, its negative $\psi=-\fg$
satisfies the opposite inequality:
\begin{equation}\label{EQ:SSUPMOD}
\sum_{K\subseteq\{1,\ldots,n\}} (-1)^{|K|} \psi(A_0\cup A_K)\ge 0
\end{equation}
for all $A_0,\ldots,A_n\in\FF$. If we add the condition $\psi\ge 0$, then
\eqref{EQ:SSUPMOD} will hold for $n=0$ as well. The case $n=1$ means that
$\psi$ is decreasing; the cases $n=1,2$ together mean that $f$ is decreasing
and supermodular. In this case, we cannot require that $\psi(\emptyset)=0$; the
condition $\psi(J)=0$ can be achieved, but it is not always useful. We call a
setfunction satisfying inequality \eqref{EQ:SSUPMOD} for every $n\ge0$ {\it
strongly supermodular}\footnote{Choquet calls strongly submodular setfunctions
{\it alternating of infinite order}; I feel, however, that this expression is
overused. He calls the setfunction $-\fg^c$ {\it monotone of infinite order}}.

The following lemma gives a very general construction for strongly submodular
functions.

\begin{lemma}\label{LEM:R-INDUCED}
Every pullback of a strongly submodular setfunction is strongly submodular.
\end{lemma}

\begin{proof}
Let $\fg$ be a strongly submodular setfunction on the set-algebra $(J,\BB)$,
let $(I,\AA)$ be another set-algebra, and let $\Gamma:~\AA\to\BB$ be a monotone
increasing union-preserving map. We claim that $\psi=\fg\circ\Gamma$ is
strongly submodular. Indeed, let $A_0,\ldots,A_n\in\AA$. Then
\begin{align*}
\sum_{0\in K\subseteq I} (-1)^{|K|} \psi(\cup_{i\in K} A_i)
&= \sum_{0\in K\subseteq I} (-1)^{|K|} \fg(\Gamma(\cup_{i\in K} A_i))\\
&= \sum_{0\in K\subseteq I} (-1)^{|K|} \fg(\cup_{i\in K} \Gamma(A_i))\le 0.
\end{align*}
\end{proof}

Among our previous examples, \ref{EXA:SUP}, \ref{EXA:PROB}, \ref{EXA:IDEAL} and
\ref{EXA:CLOSURE} are also strongly submodular; they are all pullbacks of
charges. Example \ref{EXA:CONCAVE} is strongly submodular only for special
concave functions $f$ (called totally monotone). Examples \ref{EXA:CUT} and
\ref{EXA:LIN-DEP} are not strongly submodular in general. It will be useful to
generalize Example \ref{EXA:TRANSVERSAL} as follows.

Consider two set-algebras $(I,\AA)$ and $(J,\BB)$ and a relation $R\subseteq
J\times I$ that is measurable in the sense that if $X\in\BB$ then $R(X)\in\AA$.
Let $\alpha$ be a charge on $\AA$. We say that the setfunction
$\psi(X)=\alpha(R(X))$ is {\it induced by} $R$ and $\alpha$. Since the map
$X\mapsto R(X)$ is monotone increasing and union-preserving, the setfunction
$\psi$ is a pullback of $\alpha$. This implies that $\psi$ is strongly
submodular.

Examples \ref{EXA:SUP} and \ref{EXA:PROB} are easy special cases. Let us show
that Example \ref{EXA:CLOSURE} is a special case as well.

\begin{lemma}\label{LEM:CLOSURE-CHARGE}
Let $(J,\BB)$ be the sigma-algebra of Borel sets of a compact metric space
$(J,d)$, and let $\mu$ be a measure on $(J,\BB)$. Then there is a sigma-algebra
$(I,\AA)$, and a charge $\alpha$ on it, and a measurable relation $R\subseteq
J\times I$, so that
\[
\mu(\overline{X})=\alpha(R(X))
\]
for every $X\in\BB$.
\end{lemma}

\begin{proof}
Let $I=J\times[0,1]$, let $\AA$ be the family of Borel subsets of $I$, and let
\[
R=\{(x,(y,z)):~x,y\in J,\, z\in[0,1],\,d(x,y)< z\}.
\]
Let $\alpha_n=\mu\times\pi_n$, where $\pi_n$ is the uniform probability measure
on $[0,1]\times[0,1/n)$ ($n\in\Nbb$). Fix a nontrivial ultrafilter $\FF$ on
$\Nbb$, and for every set $X\in\AA$, define the ultralimit
\begin{equation}\label{EQ:A-ULTRA}
\alpha(X)=\lim_\FF \alpha_n(X).
\end{equation}
Clearly $\alpha$ is a charge on $\AA$. We claim that
\begin{equation}\label{EQ:ARULU}
\alpha(R(U))=\mu(\overline{U})
\end{equation}
for $U\in\BB$. For $0<z\le 1$, let $U_z$ denote the $z$-neighborhood of $U$.
Then the set $R(U)\cap (J\times\{z\})$ is congruent with $U_z$, and hence
\[
\alpha_n(R(U)) =n\intl_0^{1/n} \mu(U_z)\,dz.
\]
For $0<z\le 1/n$, the inclusion $\overline{U}\subseteq U_z\subseteq U_{1/n}$
implies that $\mu(\overline{U})\le \mu(U_z)\le \mu(U_{1/n})$. Hence
$\mu(\overline{U})\le \alpha_n(R(U)) \le \mu(U_{1/n})$. Since $\lim_{n\to
\infty} \mu(U_{1/n}) = \mu(\overline{U})$, this inequality implies that
$\lim_{n\to \infty} \alpha_n(R(U)) = \mu(\overline{U})$. If this limit exists,
then it is also equal to the ultralimit $\alpha(R(U))$.
\end{proof}

Two apparent gaps below should be explained. First, due to the simple
transformation between strongly submodular and supermodular setfunctions,
results about one property can easily be translated to results about the other,
and we will not describe these parallel assertions. Second, it is clear that a
setfunction $\fg$ is strongly submodular if and only if every finite quotient
of it is strongly submodular. So every result giving equivalent conditions for
strong submodularity in the finite case (to be discussed in the next section)
provides a characterization in the infinite case, by adding ``for every finite
quotient''; we don't state these separately.

\subsection{The finite case}

The results below about the finite case would follow from the much more general
results of Choquet about strongly submodular setfunctions on compact subsets of
a compact set, but the treatment here is (naturally) much simpler and more
explicit.

Before turning to strongly submodular setfunctions, we recall some well-known
results about the linear algebra of subsets of a finite set. Let $V$ be a
finite set with $|V|=n$, and consider the spaces $\R^{2^V}$ of setfunctions on
$(V,2^V)$ and $\R^{2^V\times 2^V}$ of $2^n\times 2^n$ matrices whose rows and
columns are indexed by the subsets of $V$. Let $\Zb$ and $\Mb$ be $2^V\times
2^V$ matrices, defined by
\begin{equation}\label{EQ:ZB-DEF}
\Zb_{X,Y}=\one(X\subseteq Y) \et \Mb_{X,Y} = \one(X\subseteq Y) (-1)^{|Y\setminus X|},
\end{equation}
The matrices $\Mb$ and $\Zb$ are triangular, with $1$'s in the diagonal. It is
easy to check that $\Mb=\Zb^{-1}$.

Consider a setfunction $\fg:~2^V\to\R$, which we can view as a vector in
$\R^{2^V}$. The setfunctions $\Zb \fg$ and $\Zb\T \fg$ are called its {\it
upper} and {\it lower summation functions}, and $\Mb \fg$ and $\Mb\T \fg$ are
its {\it upper} and {\it lower Möbius inverses}. Explicitly,
\[
(\Zb \fg)_X=\sum_{Y\supseteq X} \fg(Y),\qquad (\Zb\T \fg)_X=\sum_{Y\subseteq X} \fg(Y)
\]
and
\[
(\Mb \fg)_X=\sum_{Y\supseteq X}  (-1)^{|Y\setminus X|} \fg(Y),\qquad (\Mb\T \fg)_X=\sum_{Y\subseteq X} (-1)^{|X\setminus Y|} \fg(Y)
\]

Introducing the permutation matrix $\Cb$ which interchanges a set with its
complement, explicitly $\Cb_{X,Y}= \one(X^c=Y)$, we define $\Pb=\Jb-\Zb\Cb$ and
$\Nb=-\Cb\Mb$. The matrices $\Pb$ and $\Nb$ are symmetric. The matrix $\Pb$ has
a zero row and column corresponding to $\emptyset$, so it is not invertible.
But we can delete this row and column, to get the matrix $\Pb'$, which is
invertible; it is easy to check that its inverse is the matrix $\Nb'$, obtained
from $\Nb$ by deleting the row and the column corresponding to $\emptyset$. If,
instead, we replace the only nonzero element of this row and column of $\Nb$ by
$0$ (this is a $1$ in the position $V$), then we get a matrix which is a
Penrose inverse of $\Pb$, and we may denote it by $\Pb^{-1}$. The product
$\Pb\Pb^{-1}$ is obtained from the identity matrix by replacing its
$(\emptyset,\emptyset)$-entry by $0$.

\begin{lemma}\label{LEM:INDUCE}
For every setfunction $\fg:~2^V\to\R$ with $\fg(\emptyset)=0$ there is a unique
setfunction $\alpha:~2^V\to\R$ with $\alpha(\emptyset)=0$ such that
$\fg=\Pb\alpha$.
\end{lemma}

\begin{proof}
We want to solve the equation $\Pb\alpha=\fg$. Here $\Pb$ is singular, but the
value $\alpha(\emptyset)$ plays no role, and we have the equation
$\alpha(\emptyset)=0$. Since $\fg(\emptyset)=0$, the equation $\Pb\alpha=\fg$
is equivalent with $\Pb'\alpha'=\fg'$, where $\alpha'$ and $\fg'$ denote the
restrictions of $\alpha$ and $\fg$ to nonempty subsets. Equivalently,
$\alpha'=\Nb'\fg'$, which we can write as $\alpha=\Pb^{-1}\fg$.
\end{proof}

We define the (symmetric) {\it union matrix} $\Ub_\fg\in \R^{2^V\times 2^V}$ of
a setfunction $2^V\to\R$ by
\begin{equation}\label{EQ:UB-DEF}
(\Ub_\fg)_{X,Y}=\fg(X\cup Y) \qquad (X,Y\subseteq V).
\end{equation}
Defining the diagonal matrix $\Db_\fg=\diag(\fg(X):~X\subseteq V)$, we have the
simple but very useful {\it Lindstr\"om--Wilf Formula} \cite{Li,Wi}:
\begin{equation}\label{EQ:LI-WI}
\Ub_\fg=\Zb\Db_{\Mb \fg}\Zb\T.
\end{equation}
As an immediate consequence, we see that the number of positive [negative]
eigenvalues of the union matrix $\Ub_\fg$ is the number of positive [negative]
values of the upper Möbius inverse $\Mb \fg$.

\begin{theorem}\label{THM:STRONG-SUBMOD}
For a setfunction $\fg:~2^V\to\R$ on a finite set $V$ with $\fg(\emptyset)=0$,
the following are equivalent:

\smallskip

{\rm(i)} $\fg$ is strongly submodular;

\smallskip

{\rm(ii)} $\fg$ can be induced by the relation $\inb$ and a setfunction
$\alpha:~2^V\to\R_+$.

\smallskip

{\rm(iii)} $\fg$ can be induced by a relation $R\subseteq V\times W$ and a
charge $\alpha:~2^W\to\R_+$.

\smallskip

{\rm(iv)} The upper Möbius inverse satisfies $\Mb\fg(X) \le 0$ for all $X
\subseteq V$, $X\not=V$.

\smallskip

{\rm(v)} The union matrix $\Ub_\fg$ has at most one positive eigenvalue.
\end{theorem}

\begin{proof}
Some implications are easy: (iii)$\Rightarrow$(ii) is trivial,
(ii)$\Rightarrow$(i) follows by Lemma \ref{LEM:R-INDUCED}, and
(iv)$\Leftrightarrow$(v) follows immediately from the Lindström--Wilf Formula
\eqref{EQ:LI-WI}.

\medskip

(i)$\Rightarrow$(iv): The inequality $\Mb\fg(X) \le 0$ for $X\not=V$ means that
\[
\sum_{Y\supseteq X} (-1)^{|Y\setminus X|} \fg(Y) = \sum_{Z\subseteq X^c} (-1)^{|Z|} \fg(X\cup Z) \le 0.
\]
This is a special case of Choquet's condition \eqref{EQ:SSUBMOD} with $A_0=X$
and $A_1,\ldots,A_n$, the singleton subsets of $X^c$. (The condition $X\not=V$
guarantees the bound $n\ge 1$ in \eqref{EQ:SSUBMOD}.)

\medskip

(iv)$\Rightarrow$(iii): By Lemma \ref{LEM:INDUCE}, there is a unique
$\alpha:~2^V\to\R$ such that $\alpha(\emptyset)=0$ and $\Pb\alpha=\fg$, which
can be expressed as $\Pb^{-1}\fg$. This $\alpha$ and the relation $\inb$ induce
$\fg$. We have $\alpha\ge 0$ if and only if $\Pb^{-1}\fg\ge 0$. Here $\Pb^{-1}$
agrees with $\Nb$ except on the row and column indexed by $\emptyset$, and
hence (using that $\fg(\emptyset)=0$) we have $\Nb\fg=\Pb^{-1}$ on nonempty
sets. Using that $\Nb=-\Cb\Mb$, we see that $\Pb^{-1}\fg\ge 0$ is equivalent to
$\Mb\fg(X)\le0$ for $X\not=V$.
\end{proof}

This theorem implies that the extreme rays of the convex cone of strongly
submodular setfunctions on the subsets of a finite set $V$ are given by
$\psi_T(X)=\one(X\cap T\not=\emptyset)$ ($T\subseteq V$). The argument in the
proof also gives that every setfunction $\fg$ with $\fg(\emptyset)=0$
(submodular or not) can be induced by the relation $\inb$ and a function
$\alpha:~2^V\to\R$. Furthermore, $\alpha$ is uniquely determined, and the
theorem implies that $\alpha\ge 0$ if and only if $\fg$ is strongly submodular.

A property of a symmetric matrix $M$ stronger than having at most one positive
eigenvalue is that the matrix is of {\it negative type}: this means that the
quadratic form $x\T Mx$ is nonnegative on the linear hyperplane $\one^\perp$
(see e.g.~\cite{DL}, Section 6.1 and \cite{HLMT} for more on this notion). In
general, negative type is a strictly stronger property than having at most one
positive eigenvalue; however, for union matrices, they are equivalent.

\begin{corollary}\label{COR:UNION-POS}
A setfunction $\fg:~2^V\to\R$ with $\fg(\emptyset)=0$ is strongly submodular if
and only if its union matrix $\Ub_\fg$ is of negative type.
\end{corollary}

\begin{proof}
The ``if'' part follows from the obvious fact that negative type matrices have
at most one positive eigenvalue. To prove that $\Ub_\fg$ is of negative type
for a strongly submodular setfunction $\fg:~2^V\to\R$ with $\fg(\emptyset)=0$,
we have to prove that for every $x\in\R^{2^V}$ such that $\sum_{U\subseteq V}
x_U=0$,
\[
\sum_{U,W\subseteq V} x_Ux_W\fg(U\cup W)\le 0.
\]
Using the representation given by Theorem \ref{THM:STRONG-SUBMOD}(ii), we have
\begin{align*}
\sum_{U,W\subseteq V} x_Ux_W\fg(U\cup W) = \sum_{U,W\subseteq V} x_Ux_W \sum_{Y\cap (U\cup W)\not=\emptyset} \alpha(Y).
\end{align*}
Let us determine the coefficient of $\alpha(Y)$ for any given $Y\subseteq V$ in
the sum above. Either $U$ or $W$ must meet $Y$, so we can split the sum into
three terms:
\begin{align*}
\sum_{U,W\in \inb(Y)}x_Ux_W &+ \sum_{U\in \inb(Y), W\notin \inb(Y)}x_Ux_W +\sum_{U\notin \inb(Y), W\in \inb(Y)}x_Uy_W\\
&=\Big(\sum_{U\in 2^V} x_U\Big)^2 - \Big(\sum_{U\notin \inb(Y)} x_U\Big)^2.
\end{align*}
The first term is $0$, proving that the coefficient of $\alpha(Y)$ is
nonpositive. Since $\alpha\ge 0$ if $\fg$ is strongly submodular, this proves
Corollary \ref{COR:UNION-POS}.
\end{proof}

Most of the time, this property is studied for distance matrices of finite
(pseudo)metric spaces, where $(M_{ij})_{i,j=1}^n$ is the distance of $i$ and
$j$ in some metric space on $\{1,\ldots,n\}$. We say that a (possibly infinite)
pseudometric space has negative type, if all finite subspaces have negative
type. Many important metric spaces have this property.

We can turn Corollary \ref{COR:UNION-POS} to a statement about a (pseudo)metric
space introduced in Section \ref{SEC:BJ-DIST}.

\begin{corollary}\label{COR:BJ-NEG-TYPE}
A setfunction on the subsets of a finite set is strongly submodular if and only
if it defines a Björner distance of negative type.
\end{corollary}

\begin{proof}
First, assume that $J$ is finite and $\BB=2^V$. The matrix of the Björner
distance is given by
\[
(\Bb_\fg)_{U,W}=2\fg(U\cup W)-\fg(U)-\fg(W),
\]
and hence for any vector $x\in\R^{2^V}$ with $\sum_U x_U=0$, we have
\[
x\T\Bb_\fg x = 2 \sum_{U,W\subseteq V} \fg(U\cup W) x_Ux_W - \sum_{U,W\subseteq V} \fg(U) x_Ux_W
- \sum_{U,W\subseteq V} \fg(W) x_Ux_W.
\]
The last two terms are zero, so the quadratic form is nonpositive is and only
if the first term is.

For the case of a general set-algebra $(J,\BB)$, it is clear that $\fg$ is
strongly submodular if and only if every finite quotient of it is strongly
submodular. It is also easy to see that the corresponding Björner distance
defined by $\fg$ has negative type if an only if this holds for every finite
quotient.
\end{proof}

We conclude with describing the behavior of $\fg$ and $\alpha$ in Theorem
\ref{THM:STRONG-SUBMOD}(ii) on quotients. Let $U$ and $V$ be two finite sets,
and let $\Psi:~V\to U$ be a surjective map. Let $\fg$ be a setfunction on
$(V,2^V)$, then we can consider the quotient $\fg_1=\fg\circ\Psi^{-1}$ of $\fg$
by $\Psi$, as defined in Section \ref{SSEC:SUBMOD-BASICS}. Let $\Pb$ be the
$2^V\times 2^V$ matrix defined above, and let $\alpha=\Pb^{-1}\fg$. Let $\Pb_1$
be the $2^U\times 2^U$ matrix analogous to $\Pb$, and let
$\alpha_1=\Pb_1^{-1}\fg_1$. We want to express $\alpha_1$ in terms of $\alpha$.

Let us define two $2^U\times 2^V$ matrices $\Qb$, $\Rb$ by
\[
\Qb_{X,Y}=\one(Y=\Psi^{-1}(X)), \et \Rb_{x,y} = \one(X=\Psi(Y)).
\]
Then it is easy to check that $\fg_1=\Qb\fg$, and also the identity
$\Pb_1\Rb=\Qb\Pb$. Clearly $(\Rb\alpha)(\emptyset)=0$, and
\[
\Pb_1\Rb\alpha = \Qb\Pb\alpha = \Qb\fg = \fg_1,
\]
showing that $\Rb\alpha=\alpha_1$. Explicitly,
\begin{equation}\label{EQ:ALPHA-QUOTIENT}
\alpha_1(X)=\sum_{Y\subseteq V\atop \Psi(Y)=X} \alpha(Y).
\end{equation}

\subsection{Finite subsets of an infinite set}\label{SSEC:LOCFIN}

We prove an extension of Theorem \ref{THM:STRONG-SUBMOD} for the case when the
underlying ring $(J,J_\fin)$ consists of the finite subsets of an infinite set.
This is still a very special ring of sets, but it does occur in the important
example of homomorphism numbers and densities (Example \ref{EXA:HOMS}).

We consider the sigma-algebra $(I,\AA)$, where $I=[0,1]^\Nbb$ and
$\AA=\BB^\Nbb$. Let $R$ consist of those pairs $(x,y)\in I\times J$, where
$x=(x_1,x_2,\ldots)$ and $y=x_n$ for some $n$. It is easy to see that $R$ is a
Borel set. Furthermore, this relation is {\it countably separating}, which
means that for any two disjoint countable sets $A,B\subseteq J$ there is an
$x\in I$ such that $A\subseteq R(x)$ but $B\subseteq J\setminus R(x)$. Indeed,
any sequence $x\in I$ using all elements of $A$ but no elements of $B$ has this
property.

Recall the notation $\QR(X)=R(X^c)^c=\{y\in J:~(x,y)\in R~\forall x\in X\}$.
For a setfunction $\fg:~J_\fin\to\R$, define the function
$\overline{\fg}:~J_\fin\times J_\fin\to\R$ by
\begin{equation}\label{EQ:OLF-INF}
\overline{\fg}(A,B) = \sum_{X\subseteq B} (-1)^{|X|} \fg(A\cup X).
\end{equation}
We can also write this as
\begin{equation}\label{EQ:OLF-INF1}
\overline{\fg}(A,B) = \sum_{A\subseteq Y\subseteq A\cup B} (-1)^{|Y\setminus A|} f(Y).
\end{equation}

\begin{theorem}\label{THM:KOLM-BOR}
Let $\fg$ be a setfunction on $J_\fin$ with $\fg(\emptyset)=0$. Then the
following are equivalent:

\smallskip

{\rm(a)} $\fg$ is strongly submodular;

\smallskip

{\rm(b)} $\overline{\fg}(A,B)\le 0$ for all $A,B\in J_\fin$, $B\not=\emptyset$.

\smallskip

{\rm(c)} there exists a measure $\alpha$ on $(I,\CC)$ such that
$\alpha(R(A))=\fg(A)$ for every $A\in J_\fin$.
\end{theorem}

\begin{proof}
All three conditions imply easily that $\fg$ is increasing, and therefore
nonnegative. We may assume that $\fg\le 1$. The implication (a)$\Rightarrow$(b)
is easy, since the relation $\overline{f}(A,B)\le 0$ for $B\not=\emptyset$ is a
special case of \eqref{EQ:SSUBMOD}. (c)$\Rightarrow$(a) follows by Lemma
\ref{LEM:R-INDUCED}.

To complete the cycle, we prove that (b)$\Rightarrow$(c). We need some
preparation. For $A\in J_\fin$, the set $R(A)$ consists of those sequences from
$J^\Nbb$ which contain at least one point of $A$, and $\QR(A)$ consists of
those sequences from $J^\Nbb$ which contain all points of $A$. For
$A,B\subseteq J$, let
\begin{equation}\label{EQ:HAB}
H(A,B)=\QR(B)\setminus R(A).
\end{equation}
Clearly $H(\emptyset,\emptyset)=I$, more generally, $H(\emptyset,B)=\QR(B)$,
and $H(A,\emptyset)= I\setminus R(A)$. We also have $H(A,B)=\emptyset$ if
$A\cap B\not = \emptyset$. The fact that $R$ is countably separating implies
that $H(A,B)\not=\emptyset$ if $A$ and $B$ are countable and $A\cap B =
\emptyset$. Trivially $H(A,B)\in\CC$ if $A,B\subseteq I$ are countable. Hence
the sigma-algebra generated by the sets $H(A,B)$ $(A,B\in J_\fin)$ is just
$\CC$.

\begin{claim}\label{CLAIM:3.1}
The family $\Rf=\{H(A,B):~A,B\in J_\fin\}$ is a semiring.
\end{claim}

It is easy to check that if $A_i,B_i\subseteq J$ ($i\in S$), $A=\cup_{i\in S}
A_i$ and $B=\cup_{i\in S} B_i$, then
\begin{equation}\label{EQ:INTERSECTION}
\bigcap_{i\in S} H(A_i,B_i) =H(A,B).
\end{equation}
Here $A$ and $B$ are finite if $S$ is finite. In particular,
\begin{equation}\label{EQ:INTERSECTION2}
H(A_1,B_1)\cap H(A_2, B_2) =  H(A_1\cup A_2,B_1\cup B_2) \qquad (A_i,B_i\in J_\fin),
\end{equation}
so $\Rf$ is closed under intersections. We also need the identity that if
$A,B\subseteq U\subseteq J$, then
\begin{equation}\label{EQ:EXT}
H(A,B) = \bigcup_{A\subseteq X\subseteq U\setminus B} H(X,U\setminus X).
\end{equation}
Note that the sets $H(X,U\setminus X)$ on the right are mutually disjoint, and
so they partition $H(A,B)$.

Using this, it is easy to express the difference of two sets of the form
$H(A,B)$. Let $A,B,C,D\subseteq U\subseteq I$. Then $H(A,B)\setminus H(C,D)$
can be obtained by deleting those terms in \eqref{EQ:EXT} satisfying
$C\subseteq X\subseteq U\setminus D$, to get
\begin{equation}\label{EQ:DIFF}
H(A,B)\setminus H(C,D)= \bigcup_{A\subseteq X\subseteq U\setminus B\atop(C\setminus X)
\cup (D\cap X)\not=\emptyset} H(X,U\setminus X).
\end{equation}
In \eqref{EQ:DIFF} we can take $U=A\cup B\cup C\cup D$. With this choice, if
$A,B,C,D$ are finite, then the union has a finite number of terms. This proves
the Claim.

\begin{claim}\label{CLAIM:3.2}
If $A,B,C,D\in J_\fin$ and $H(A,B)=H(C,D)\not=\emptyset$, then $A=C$ and $B=D$.
\end{claim}

Indeed, assume that $A\not=C$, and (say) $C\not\subseteq A$. Then $A$ and
$B'=(C\setminus A)\cup B$ are disjoint finite sets, and hence
$H(A,B')\not=\emptyset$. Let $y\in H(A,B')$, then $y\in H(A,B)$, but $y\in
R(C)$ and so $y\notin H(C,D)$.

Define
\[
\alpha(H(A,B))=
  \begin{cases}
    -\overline{\fg}(A,B), & \text{if $B\not=\emptyset$}, \\
    1-\fg(A), & \text{if $B=\emptyset$}.
  \end{cases}
\]
We have shown that if $H(A,B)$ is nonempty, then it determines $A$ and $B$, so
$\alpha$ is well-defined.

\begin{claim}\label{CLAIM:3.3}
$\alpha$ is a measure on $\Rf$.
\end{claim}

We argue first that additivity holds for the partition \eqref{EQ:EXT}, i.e.,
for $A,B,U\in J_\fin$ with $A,B\subseteq U$, we have
\begin{equation}\label{EQ:EXT1}
\alpha(H(A,B)) = \sum_{A\subseteq X\subseteq U\setminus B} \alpha(H(X,U\setminus X)).
\end{equation}
For $B\not=\emptyset$, this means that
\begin{equation}\label{EQ:EXT1A}
\overline{\fg}(A,B) = \sum_{A\subseteq X\subseteq U\setminus B} \overline{\fg}(X,U\setminus X).
\end{equation}
This can be verified by direct computation. Starting with the right hand side
and using \eqref{EQ:OLF-INF1}, we have
\begin{align*}
\sum_{A\subseteq X\subseteq U\setminus B} \overline{\fg}(X,U\setminus X)
&= \sum_{A\subseteq X\subseteq U\setminus B} \sum_{X\subseteq Z\subseteq U} (-1)^{|Z\setminus X|} \fg(Z)\\
&= \sum_{A\subseteq Z\subseteq U} \fg(Z) \sum_{A\subseteq X\subseteq Z\cap(U\setminus B)} (-1)^{|Z\setminus X|}.
\end{align*}
The inner sum is
\[
  \begin{cases}
    (-1)^{|Z\setminus A|}, & \text{if $A\subseteq Z\subseteq A\cup B$}, \\
    0, & \text{otherwise}.
  \end{cases}
\]
Hence
\[
\sum_{A\subseteq X\subseteq U\setminus B} \overline{\fg}(X,U\setminus X)
= \sum_{A\subseteq Z\subseteq A\cup B} \fg(Z) (-1)^{Z\setminus A} = \overline{\fg}(A,B).
\]
This proves \eqref{EQ:EXT1}. For $B=\emptyset$, $1$ is added to both sides of
equation \eqref{EQ:EXT1A}, and so it can be verified by the same computation.

Next we show finite additivity on the semiring $\Rf$.
\begin{equation}\label{EQ:X-PART}
H(C,D)=\bigcup_{i=1}^m H(A_i,B_i)
\end{equation}
where $A_i,B_i,X.Y\in J_\fin$ and the sets $H(A_i,B_i)$ are disjoint. Choose
$U\in J_\fin$ large enough so that $A_i,B_i,C,D\subseteq U$. Then each
$H(A_i,B_i)$ as well as $H(C,D)$ can be expressed as the union of some of the
sets $H(X,U\setminus X)$ ($X\subseteq U$). Since the sets $H(A_i,B_i)$ are
disjoint, the sets $H(X,U\setminus X)$ in these representations are all
different, and the union of all these sets is $H(C,D)$. Applying
\eqref{EQ:EXT1} to each of these representations, we get that
\begin{equation}\label{EQ:AX-PART}
\alpha(H(C,D))=\sum_{i=1}^m \alpha(H(A_i,B_i)),
\end{equation}
proving Claim \ref{CLAIM:3.3}.

It follows by basic measure theory that $\alpha$ can be extended to a
nonnegative charge on the ring $\Rf'$ generated by $\Rf$. Next we prove that
$\alpha$ is countably additive on the ring $\Rf'$. We need a combinatorial
fact.

\begin{claim}\label{CLAIM:KOL3}
If $X_1\supseteq X_2\supseteq\ldots$ is a countable sequence of sets in
$X_i\in\Rf$ and $\cap_i X_i=\emptyset$, then $X_i=\emptyset$ for every
sufficiently large $i$.
\end{claim}

Let $X_i=H(A_i,B_i)$ with $A_i,B_i\in J_\fin$, then by \eqref{EQ:INTERSECTION},
\[
\bigcap_i X_i = H(A,B),
\]
where $A=\cup_i A_i$ and $B=\cup_i B_i$. By the separating property of $R$,
$H(A,B)\not=\emptyset$ if $A$ and $B$ are disjoint countable sets, so the
hypothesis that $\cap_iX_i=\emptyset$ implies that $A\cap B\not=\emptyset$ and
hence $A_i\cap B_j\not=\emptyset$ for some $i,j\ge 1$. But then
$X_{\max(i,j)}=\emptyset$.

\begin{claim}\label{CLAIM:KOL4}
The conclusion of Claim \ref{CLAIM:KOL3} also holds if the assumption
$X_i\in\Rf$ is weakened to $X_i\in\Rf'$.
\end{claim}

Write $X_i=\cup_{j=1}^{k_i} X_{ij}$, where $X_{ij}\in\Rf$ are disjoint nonempty
sets. We may assume (using \eqref{EQ:INTERSECTION}) that every $X_{ij}$ is
contained in exactly one of the sets $X_{i-1,j'}$. This defines a rooted tree
$T$ with nodes $X_{ij}$, where the parent of $X_{ij}$ is the set $X_{i-1,j'}$
defined above. We consider the set $X_{0,1}=I$ as the root.

Suppose that $X_i\not=\emptyset$ for every $i$. Then $T$ is infinite, and by
K\H{o}nig's Lemma, it contains an infinite path $X_{0,1}\supseteq
X_{1,j_1}\supseteq X_{2,j_2}\supseteq\ldots$ Here
\[
\bigcap_{i=1}^\infty X_{i,j_i} \subseteq \bigcap_{i=1}^\infty X_i =\emptyset,
\]
and so as remarked above, $X_{i,j_i}=\emptyset$ for sufficiently large $i$.
This contradicts the definition of $T$, and proves Claim \ref{CLAIM:KOL4}.

Thus $\alpha$ is a pre-measure on $(I,\Rf')$, with
\[
\alpha(I)=\alpha(H(\emptyset,\emptyset))=\fg(I)=1.
\]

The sigma-algebra generated by $\Rf'$ is $\CC$. By Caratheodory's Extension
Theorem, $\alpha$ extends to a probability measure on $(I,\CC)$. Now for any
$A\in J_\fin$,
\[
\fg(A) = \overline{f}(A,\emptyset) = 1-\alpha(H(A,\emptyset)) = 1-\alpha(I\setminus R(A)) =\alpha(R(A)).
\]
This proves (c).
\end{proof}

\subsection{Compact subsets of a compact set}

Choquet proved the following general representation theorem for a large class
of strongly submodular setfunctions (\cite{Choq}, Section 49). Let $K$ be a
compact Hausdorff space and $\FF$, the family of its compact subsets. Let us
call $(K,\FF)$ a {\it compact pair}. Clearly $\FF$ is a lattice of sets.

We consider the relation $R=\{(x,Y):~x\in Y\in\FF\}\subseteq K\times\FF$. Let
$\AA\subseteq 2^\FF$ be the sigma-algebra generated by the sets $R(X)$
($X\in\FF$). One of the  meany results in \cite{Choq} can be stated like this.

\begin{theorem}\label{THM:CHOQUET}
Let $\fg$ be a strongly submodular setfunction on a compact pair $(K,\FF)$.
Then there is a measure $\mu$ on $(\FF,\AA)$ such that $\fg(X)=\mu(R(X))$ for
all $X\in\FF$.
\end{theorem}

For the proof of this theorem, its extension to the locally compact case, and
other related results, we refer to \cite{Choq}.

\section{Problems and projects}\label{SEC:PROBLEMS}

\begin{prob}[Logarithms of determinants in the infinite
case]\label{PROB:CONTINUITY}~\\
Can Example \ref{EXA:LOG-DET} be generalized to the infinite case?
Fuglede--Kadison determinants could play a role here.
\end{prob}

\begin{prob}[Decomposition of submodular setfunctions]\label{PROB:DECOMPOSE}~\\
{\it Can every bounded submodular setfunction be written as the sum of an
increasing and a decreasing submodular setfunction?} This is true in the finite
case, and would strengthen Lemma \ref{LEM:VARPHI-DECOMP}.
\end{prob}

\begin{prob}[Integrality of separation]\label{PROB:SEPARATE}~\\
Does Frank's integral valued version of Theorem \ref{THM:SEPARATE} in the
finite case (see Remark \ref{REM:FRANK}) has an analogue in the infinite case?
\end{prob}

\begin{prob}[Extremal minorizing measures]\label{PROB:EXTREME}~\\
We call a minorizing measure $\alpha\in\matp_+(\fg)$ {\it extremal}, if it
cannot be written as $\alpha=(\alpha_1+\alpha_2)/2$, where $\alpha_1$ and
$\alpha_2$ are minorizing measures and $\alpha_1\not=\alpha_2$. In the finite
case, measures that agree with $\fg$ on a full chain are the same as extremal
measures. What is the connection in the infinite case? Can extremal minorizing
measures be characterized? Let us raise an explicit problem:

{\it Let $\fg$ be an increasing submodular setfunction and
$\alpha\in\matp_+(\fg)$ with $\alpha(J)=\fg(J)$. It is not hard to see that if
$(\fg-\alpha)\land(\fg^c-\alpha^c)=0$, then $\alpha$ is extremal. Is this
condition also necessary?}
\end{prob}

\begin{prob}[Extremal minorizing measures II]\label{PROB:M-INTERSECT-EXTREME}~\\
In the finite case, every vertex of the intersection of two matroid polytopes
is a vertex of both (if they have the same rank, these are the common bases).
This does not remain valid for polymatroid polytopes, but the vertices of the
intersection of two such polytopes are still integral. Can this be generalized
to the situation of the Submodular Intersection Theorem \ref{THM:INTERSECT}?
Perhaps Lindenstrauss' Theorem mentioned in Remark \ref{REM:LIND} can be
extended.
\end{prob}

\begin{prob}[Infinite version of submodular
minimization]\label{PROB:MATROID-GEN}~\\
Finding the minimum of a (non-monotone) submodular setfunction in an efficient
way is a basic algorithmic task in the finite case; this was solved by analytic
methods in \cite{GLS2}, and by combinatorial methods in \cite{Schr1,IFF}.
Submodular setfunction minimization in the finite case features in many
important applications from computing minimum capacity cuts (cf. Example
\ref{EXA:CUT}) to matroid intersection (cf.~Proposition
\ref{PROP:MATRID-INTERSECT}).

Extensions of submodular function minimization from Boolean algebras to more
general structures has been considered (Murota \cite{Mur3}, Hanada and Hirai
\cite{HaHi}. The problem comes up in the measurable case as well (see
e.g.~Theorem \ref{THM:INTERSECT}). However, it is not clear what an
``algorithm'' should mean in this case. Perhaps the following is a valid
approach.

{\it Let us call a subset $X\subseteq[0,1]$ {\it presentable}, if it is the
union of a finite number of rational intervals. Let $\fg$ be a submodular
setfunction on the algebra of presentable sets, and assume we have an oracle
the gives us the value of $\fg(X)$ for any given presentable $X$. Can we
compute in finite time a presentable set on which the value of $\fg$ is within
$\eps$ to its infimum?}

It would be interesting to generalize further results of matroid theory to
submodular setfunctions: matroid partition, sum, duality. Combining with the
results of \cite{LLflow}, extending the theory of submodular flows \cite{EG} to
the measure theoretic setting may be possible.

In combinatorial optimization, many applications of submodularity use
setfunctions that satisfy the submodular inequality for restricted pairs of
subsets only: intersecting pairs, crossing pairs, subsets with odd cardinality,
paramodular pairs of functions, etc. (see Schrijver \cite{Schr}, Chapter 49).
Does any of these have extensions or analogues in the measure-theoretic
setting?
\end{prob}

\begin{prob}[Submodularity in the countable case] \label{PROB:COUNTABLE}~\\
A standard Borel sigma-algebra is either finite, countable, or is isomorphic to
the sigma-algebra of Borel sets of $[0,1]$. We have built on the finite
examples as motivation, and studied the uncountable case in some detail. The
countably infinite case may be interesting on its own right.
\end{prob}

\begin{prob}[Strong and weak maps] \label{PROB:STRONGMAP}~\\
For matroids, strong and weak maps have been studied as a more algebraic
approach to matroid theory, as two categories associated with matroids. Can
results about these categories be extended to the infinite case?
\end{prob}

\begin{prob}[Compactness of the Björner distance] \label{PROB:BJ-COMPACT}~\\
Under what conditions on an increasing submodular setfunction $\fg$ is the set
$\BB$, endowed with the Björner distance, compact? Since we know that it is
complete, the question is equivalent to the total boundedness of the space.
\end{prob}

\begin{prob}[Strong submodularity] \label{PROB:SSMOD}~\\
Lemma \ref{LEM:R-INDUCED} gives a construction for strongly submodular
setfunctions. In various cases it is known that all strongly submodular
setfunctions can be obtained this way, in particular for finite subsets of an
arbitrary set (Theorem \ref{THM:KOLM-BOR}), or compact subsets of a compact
space (Choquet \cite{Choq}, Theorem 49.1). Is there a converse to Lemma
\ref{LEM:R-INDUCED} in general? {\it Is every strongly submodular setfunction a
pullback of a charge?}
\end{prob}

\begin{prob}[Ingleton's Inequality]\label{PROB:INGLETON}~\\
An inequality strengthening submodularity was introduced by Ingleton
\cite{Ingl} (see also \cite{Oxley}):
\begin{align*}
&r(X_1)+r(X_2)+r(X_1\cup X_2\cup X_3)+r(X_1\cup X_2\cup X_4)+r(X_3\cup X_4)~\\
&\le r(X_1\cup X_2)+r(X_1\cup X_3)+r(X_1\cup X_4)+r(X_2\cup X_3)+r(X_2\cup X_4).
\end{align*}
(Submodularity is the special case when $X_4=\emptyset$). This inequality holds
true for the rank functions of linear and algebraic matroids, but not all
matroids. It is also related to the notion of {\it pseudomodularity} mentioned
above. It is not clear what the relationship between strong submodularity and
these other properties is.
\end{prob}

\begin{prob}[Limit theories of matroids]\label{PROB:PSEMOD}~\\
As mentioned in the introduction, there are previous constructions for limits
of matroids sequences: partition lattices of finite sets (which can be viewed
as cycle matroids of complete graphs; Bj\"orner \cite{Bj87}, Haiman
\cite{Hai}); pseudomodular matroids (including full algebraic and full
transversal matroids; Bj\"orner and Lov\'asz \cite{BjLo}). The limit objects
have submodular dimension functions on lattices, which translate to submodular
setfunctions on measurable spaces by Example \ref{EXA:LATTICE}, but no further
details have been worked out. A limit theory for certain (in a sense more
general) sequences of matroids was initiated by Kardo\v{s} et al.~\cite{KKLM},
based on a different concept of the limit object. Limits of matroids of finite
graphs with bounded degrees have been studied in \cite{Lov23b}.

Can these limit theories be connected? Can they provide motivation for a
general limit theory of matroids?
\end{prob}

\begin{prob}[Matroids of graphons] \label{PROB:GRAPHONS}~\\
As remarked in the Introduction, a limit theory seems to be possible for
bounded degree graphs and their limit objects, namely graphings \cite{Lov23b}.
Can the matroid of a graphon (dense graph limit) be defined, and described as
the limit of a convergent sequence of matroids of dense graphs? (See e.g.~
\cite{HomBook} for definitions.)
\end{prob}

\section{Appendices}

\subsection{Charges, a.k.a.~finitely additive measures}\label{SSEC:FIN-ADD}

We collect some facts about charges that we need (see Yosida and Hewitt
\cite{YoHe}, Rao and Rao \cite{Rao}, and Toland \cite{Tol} for more).

Let $(J,\BB)$ be a set-algebra, let $\Bd=\Bd(J,\BB)$ denote the Banach space of
bounded $\BB$-measurable functions $J\to\R$ with the supremum norm, and let
$\ba=\ba(J,\BB)$ be the Banach space of charges on $\BB$ with the norm
$\|\alpha\|=\sup_{X\in\BB}|\alpha(X)|$. We define
\[
\Bd_+ = \{f\in\Bd:~f\ge 0\} \et \ba_+ = \{\mu\in\ba:~\mu\ge 0\}.
\]
The following important fact was proved by Hildebrandt \cite{Hb} and
Fichtenholtz and Kantorovich \cite{FK}:

\begin{prop}\label{PROP:DUAL-BA}
The Banach space $\ba$ is the continuous dual of the space $\Bd$.
\end{prop}

In other words, every continuous linear functional $\Lambda$ on $\Bd$ can be
represented as $\Lambda=\wh\alpha$ by a charge $\alpha$. Explicitly,
\[
\Lambda(h)=\intl_J h \,d\alpha=\wh\alpha(h)
\]
for all $h\in\Bd$.

This duality extends to set-algebras instead of sigma-algebras with some care.
Instead of $\Bd$, we have to consider the space of $\BB$-continuous functions
(functions $f$ such that for every $\eps>0$ there is a partition of $J$ into a
finite number of sets $B_i\in\BB$ such that $f$ varies by less than $\eps$ on
each $B_i$). See \cite{Rao}, Section 4.7.

Charges on a set-algebra $(J,\AA)$ are the same as modular setfunctions
$\alpha$ with $\alpha(\emptyset)=0$. It is well known that if $\alpha$ and
$\beta$ are bounded charges, then so are $\alpha\land\beta$ and
$\alpha\lor\beta$. If $\alpha\in\ba$ and we define $\alpha_A(X)=\alpha(A\cap
X)$ for $A\in\BB$, then $\alpha_A\land\alpha_B =\alpha_{A\cap B}$ and
$\alpha_A\lor\alpha_B=\alpha_{A\cup B}$.

\begin{example}\label{EXA:ULTRA}
Let $J$ be any infinite set, and let $\FF$ be a nontrivial ultrafilter on $J$.
Then the indicator function
\[
\one_\FF(X)=\one(X\in\FF)
\]
is a charge on all subsets of $J$.
\end{example}

The Jordan Decomposition Theorem for countably additive measures remains valid
in the following form. For $\alpha\in\ba(\AA)$, we define its positive and
negative parts by
\[
\alpha_+=\alpha^\ls \et \alpha_-=(-\alpha)_+=-\alpha^\li.
\]
Then $\alpha_+,\alpha_-\in\ba_+(\AA)$, and $\alpha=\alpha_+-\alpha_-$. However,
the Hahn Decomposition Theorem does not remain valid, only approximately; see
\cite{Rao}, Section 2.6.

There is another canonical decomposition of a charge (see \cite{Tol,Rao}). A
charge $\alpha\in\ba_+(\BB)$ is called {\it purely finitely additive}, if the
only countably additive measure $\mu$ such that $0\le\mu\le\alpha$ is the zero
measure. A signed measure $\alpha\in\ba(\BB)$ is {\it purely finitely additive}
if its positive and negative parts are. Every $\alpha\in\ba(\BB)$ has a unique
decomposition as $\alpha=\beta+\gamma$ such that $\beta\in\ba(\BB)$ is purely
finitely additive and $\gamma\in\ca(\BB)$ is countably additive.

A third decomposition of a charge can be defined as follows. Assume that
$(J,\BB)$ is the sigma-algebra of Borel sets in a compact Hausdorff space $K$.
Integration by a charge $\alpha$ defines a continuous linear functional
$\LL_\alpha:~\Bd\to\R$. Continuous functions on $K$ are bounded, so
$\LL_\alpha$ can be restricted to the Borel space of continuous functions on
$K$. By the Riesz Representation Theorem, this restriction can be represented
as integration according to a countably additive measure $\overline{\alpha}$.
So $\alpha(f)=\overline{\alpha}(f)$ for every continuous function $f$. To sum
up,

\begin{prop}\label{PROP:SMOOTHING}
Let $\alpha$ be a charge on the Borel sets of a compact Hausdorff space $K$.
Then there is a countably additive measure $\overline{\alpha}$ such that
$\int_K f\,d\overline{\alpha}=\int_K f\,d\alpha$ for every continuous function
$f:~K\to\R$.
\end{prop}

This proposition remains valid for locally compact Hausdorff spaces $K$, if
instead of all continuous functions, we restrict the condition to continuous
functions vanishing at infinity (i.e., for which every level set $\{|f|\ge
c\}$, $c>0$ is compact).

The difference $\alpha-\overline{\alpha}$ is a charge such that the integral of
any continuous function $f$ with respect to $\alpha-\overline{\alpha}$ is zero.
Note, however, that the map $\alpha\mapsto\overline{\alpha}$ depends on the
compact topology chosen to represent $(J,\BB)$.

Fubini's Theorem does not remain valid for charges (Example
\ref{EXA:NOT-FUBINI}), even if one of the measures is countably additive
(Example \ref{EXA:NOT-FUBINI2}).

\begin{example}\label{EXA:NOT-FUBINI}
Let $J=\Nbb$, let $\FF$ be a nontrivial ultrafilter on $J$, and consider the
charge $\alpha(X)=\one(X\in\FF)$ as in Example \ref{EXA:ULTRA}. Let
$F(x,y)=\one(x\le y)$ for $x,y\in\Nbb$. Then
\[
\intl_\Nbb \Big(\intl_\Nbb F(x,y)\,d\alpha(x)\Big)\,d\alpha(y) = 0
\]
(since for every $y$, $F(x,y)=0$ except for a finite set of elements $x$), but
\[
\intl_\Nbb \Big(\intl_\Nbb F(x,y)\,d\alpha(y)\Big)\,d\alpha(x) = 1
\]
(since for every $x$, $F(x,y)=1$ except for a finite set of elements $y$).
\end{example}

\begin{example}\label{EXA:NOT-FUBINI2}
For $x\in \{0,1\}^\Nbb$, let $x_i$ denote the $i$-th bit of $x$. Let
\[
J=\big\{(i,j,\eps,\delta):~i,j\in\Nbb, i<j, \eps,\delta\in\{0,1\}\big\}.
\]
For every $x\in\{0,1\}^\Nbb$, let $J_x=\{(i,j,\eps,\delta)\in J:~x_i=\eps,
x_j=\delta\}$. Let $\Cf$ be the family of all finite unions of sets $J_x$ and
their subsets.

Clearly $\Cf\subseteq 2^J$ is a set ideal. We claim that $J\notin\Cf$. Indeed,
assume that there is a finite set $Y\subseteq\{0,1\}^\Nbb$ such that
$J=\cup_{y\in Y}J_y$. Then for every $i<j\in\NN$, there is a $y\in Y$ such that
$y_i=0$ and $y_j=1$. So the edges of the complete graph on $\NN$ are colored
with $|Y|$ colors. By Ramsey's Theorem, there is a monochromatic triangle
$\{i,j,k\}$, $i<j<k$. This implies that there is a $y\in Y$ such that $y_i=0$,
$y_j=1$, but also $y_j=0$ and $y_k=1$, which is impossible.

The ideal $\Cf$ can be extended to a maximal ideal $\Df$. Since $\Cf$ contains
all singletons, $\Df$ is not principal. The ultrafilter $\FF=2^J\setminus\Df$
defines a charge $\alpha=\one_\FF$. Let $\lambda$ denote the product measure of
$\{0,1\}^\Nbb$, and let $F:~J\times\{0,1\}^\Nbb\to\{0,1\}$ be defined by
\[
F(q,x)=\one(q\in J_x).
\]
Then
\[
\intl_0^1\Big(\intl_\BB F(q,x)\,d\alpha(q)\Big)\,d\lambda(x) = \intl_0^1\alpha(J_x)\,d\lambda(x) = 0,
\]
but
\[
\intl_\BB \Big( \intl_0^1 F(q,x)\,d\lambda(x)\Big)\,d\alpha(q) = \intl_0^1\lambda\{x:~q\in J_x\}\,d\alpha(q) = \frac14.
\]
So Fubini's Theorem fails for this example.
\end{example}

Fubini's Theorem does work under appropriate assumptions on the integrand, of
which we need only one in this paper (Ghirardato \cite{Ghir}).

\begin{prop}\label{PROP:GHIR}
Let $(I,\AA)$ and $(J,\BB)$ be sigma-algebras, $\alpha\in\ba(I,\AA)$,
$\beta\in\ba(J,\BB)$, and let $F:~I\times J\to\R$ be a bounded measurable
function with the comonotonicity property. Then
\[
\intl_I\intl_J F(x,y)\,d\beta(y)\,d\alpha(x)= \intl_J\intl_I F(x,y)\,d\alpha(x)\,d\beta(y).
\]
\end{prop}

We will need the following corollary. For a charge $\alpha\in\ba$ and function
$f\in\Bd$, we define the measure $f\cdot\alpha$ by
\[
(f\cdot\alpha)(X)=\int_X f\,d\alpha = \int_J \one_Xf\,d\alpha.
\]

\begin{corollary}\label{COR:FN-PROD}
For every charge $\alpha\in\ba_+$ and comonotonic functions $f,g\in\Bd_+$, we
have
\[
f\cdot(g\cdot\alpha) = g\cdot(f\cdot\alpha) = (fg)\cdot\alpha.
\]
\end{corollary}

\begin{proof}
It suffices to prove that the first and third quantities are equal. The measure
on the left can be expressed as
\begin{align*}
(f\cdot(g\cdot\alpha))(X) & = \intl_X f(x)\,d(g\cdot\alpha)(x) =  \intl_0^\infty \one_X\,(g\cdot\alpha)\{f\ge t\}\,dt\\
&= \intl_0^\infty \intl_J \one(x\in X, f(x)\ge t) g(x)\,d\alpha(x)\,dt.
\end{align*}
By Proposition \ref{PROP:GHIR}, the order of integration can be interchanged
provided the integrand $F(x,t)= \one(x\in X, f(x)\ge t)g(x)$ has the
comonotonicity property, which is easy to check. So
\begin{align*}
(f\cdot(g\cdot\alpha))(X) &= \intl_J \intl_0^\infty \one(x\in X, f(x)\ge t) g(x)\,dt\,d\alpha(x)\\
&= \intl_J \one(x\in X) f(x) g(x)\,d\alpha(x) = ((fg)\cdot\alpha)(X).
\end{align*}
\end{proof}

We conclude this section with a simple fact about extending modular
setfunctions (\cite{Rao}, Theorem 3.1.6; for the case that will be most
important for us, namely when $\LL$ is totally ordered, see also Proposition
2.10 in \cite{Denn}).

\begin{prop}\label{PROP:EXTEND}
Let $\LL\subseteq 2^J$ be a lattice family of subsets of a set $J$ such that
$\emptyset\in\LL$, and let $\NN$ be the set-algebra generated by the sets in
$\LL$. Let $\psi$ be a modular setfunction on $\LL$ with $\psi(\emptyset)=0$.
Then $\psi$ has a unique extension to a charge on $\NN$.
\end{prop}

\subsection{Matroids}\label{APP-MATROIDS}

This is an absurdly brief survey of matroid theory, which will be needed to
understand the motivation behind some of the results in this paper. For
monographs about matroid theory, see \cite{Oxley,Recski,Welsh,White}, and also
\cite{Schr}, Part IV.

A {\it matroid} can be defined as a pair $(E,\II)$, where $E$ is a finite set,
$\II\subseteq 2^E$, $\II$ is descending, $\emptyset\in\II$, and the {\it
Augmentation Axiom} is satisfied: for $A,B\in\II$ with $|A|<|B|$ there is an
element $x\in B\setminus A$ such that $A\cup\{x\}\in\II$. Sets in $\II$ are
called {\it independent}. Maximal independent sets are called {\it bases}; they
all have the same cardinality $r(E)$. A {\it fractional independent set} is a
vector $w\in\R^E$ that is a convex combination of indicator vectors of
independent sets; a {\it fractional basis} is defined analogously.

The setfunction $r(X)=\max\{|I|:~I\in\II, I\subseteq X\}$ ($X\subseteq E$) is
the {\it rank function} of the matroid. The rank function $r$ is an integer
valued increasing submodular setfunction such that $r(\emptyset)=0$ and
$r(X)\le|X|$, and these properties characterize matroid rank functions. The
rank function determines the matroid: Independent sets are characterized by the
equation $r(X)=|X|$.

Maximal sets with a given rank are called {\it flats} (also called closed
sets). The flats of a matroid form a lattice with respect to the ordering
defined by inclusion. There are simple and natural axioms defining matroids
based on any of these concepts, for which we refer to the literature.

Matroids were introduced formally by Whitney \cite{Whit}; see \cite{Schr},
Section 39.10 for a detailed history of matroid theory.

A more general structure is a {\it polymatroid}: a finite set $E$ endowed with
an integer valued increasing submodular setfunction $\rho$ such that
$\rho(\emptyset)=0$. If $\HH$ is any family of subsets of a finite set $V$,
then defining $\rho(\XX)=|\cup\XX|$ for $\XX\subseteq\HH$, we get a polymatroid
$(\HH,\rho)$. More generally, if we have a matroid rank function $r$ on the
subsets of $V$, then $\rho(\XX)=r(\cup(\XX))$ defines a polymatroid. It is a
nontrivial fact that every polymatroid arises from an appropriate matroid by
this construction (cf.~Section \ref{SSEC:DIVERGE}).

For every finite graph $G=(V,E)$, subsets of $E$ containing no cycles form the
{\it cycle matroid} of $G$. Subsets of $E$ whose removal does not disconnect
any of the connected components form the {\it cut matroid} or {\it cocycle
matroid} of $G$. Another important construction of matroids are {\it linear
matroids}: Let $d\ge 1$, and let $r(X)=\dim(\text{lin}(X))$ for $X\subseteq
\R^d$. Then $r$ is an increasing integer valued submodular setfunction. The
real space $\R^d$ can be replaced by any other finite dimensional linear space,
and by any subset of it. A further important example is an {\it algebraic
matroid}, defined by a field $\Fbb$ of finite transcendence rank over a
subfield $\Gbb$; the transcendence rank of a subset is a matroid rank function.

The Greedy Algorithm can be used to obtain another characterization of
matroids. Let $r$ be the rank function of a matroid $(E,\II)$, and let
$(v_1,\ldots,v_n)$ be any ordering of $E$. Then there is a basis $B$ of the
matroid such that $|B\cap\{v_1,\ldots,v_k\}| = r(\{v_1,\ldots,v_k\})$ for
$k=1,\ldots, n$. This basis can be obtained ``greedily'', processing the
elements $v_1,\ldots,v_n$ one-by-one, putting each element into $B$ if this
preserves its independence. Conversely, every basis $B$ can be obtained like
this, by considering any ordering starting with the elements of $B$.
Furthermore, the setfunction $\beta(X)=|B\cap X|$ is a measure on the subsets
of $E$ such that $\beta\le r$.

The set of all vectors $x\in \R_+^V$ such that
\[
\sum_{i\in S} x_i\le r(S)
\]
for all $S\subseteq E$ is a convex polytope $\matp_+(E,r)$ called the {\it
matroid polytope}. The vertices of the matroid polytope are the indicator
vectors of independent sets. Furthermore, the face of the matroid polytope
defined by adding the constraint
\[
\sum_{i\in E} x_i=r(E)
\]
has the indicator vectors of bases as vertices, and it is called the {\it basis
polytope} of the matroid.

We illustrate the power of matroid theory by recalling the Matroid Intersection
Theorem (Edmonds \cite{Edm}), which has many applications, from matchings to
rigidity of frameworks.

\begin{prop}\label{PROP:MATRID-INTERSECT}
Let $(E,\II_1)$ and $(E,\II_2)$ be two matroids on a common underlying set,
with rank functions $r_1$ and $r_2$. Then
\[
\max_{A\in\II_1\cap\II_2} |A| = \min_{X\subseteq E} r_1(X)+r_2(E\setminus X).
\]
\end{prop}

We will see that several of these fundamental properties of matroids extend to
submodular setfunctions defined on sigma-algebras. It turns out that a natural
extension of a matroid independent set to the measurable case is a special
charge.

\newpage

\subsection{Summary of notation}\label{SEC:NOTATION}

Choquet integral: $\displaystyle \wh\fg(f) = \int_0^\infty \fg\{f\ge t\}\,dt$.

\noindent Complementary setfunction: $\fg^c(X) =\fg(X^c)$.

\noindent Join and meet: $(\fg\lor\psi)(X) = \sup_{Y\subseteq
X}\big(\fg(Y)+\psi(X\setminus Y)\big),$

\noindent \phantom{Join and meet:} $(\fg\land\psi)(X) = \inf_{Y\subseteq X}
\big(\fg(Y)+\psi(X\setminus Y)\big)$.

\noindent Monotonizations: $\fg^\li(X)=\inf_{Y\in\FF, Y\subseteq
X}\fg(Y),\qquad \fg^\ui(X)=\inf_{Y\in\FF, Y\supseteq X}\fg(Y),$

\noindent \phantom{Monotonizations:} $\fg^\ls(X)=\sup_{Y\in\FF, Y\subseteq
X}\fg(Y),\qquad \fg^\us(X)=\sup_{Y\in\FF, Y\supseteq X}\fg(Y)$.

\noindent Positive and negative part (measure): $\alpha_+ =\alpha^\ls$,

\noindent \phantom{Positive and negative part (measure):}
$\alpha_-=(-\alpha)_+=-\alpha^\li \qquad(\alpha\in\ba(\AA))$.

\noindent Relation applied to a set: $R(X)=\{y:~\exists x\in X, (x,y)\in R\}$,

\noindent \phantom{Relation applied to a set:} $R^{cc}(X)=\{y:~\forall x\in X,
(x,y)\in R\}$,

\noindent Restriction and projection: $\fg_A(X)=\fg(X)$,

\noindent \phantom{Restriction and projection:} $\fg^A(X) = \fg(A^c\cup
X)-\fg(A^c) \qquad (X\subseteq A)$.

\noindent Quotient: $(\fg\circ\Pi^{-1})(X)=\fg(\Pi^{-1}(X))$.

\noindent Pullback: $(\fg\circ\Gamma)(X)=\fg(\Gamma(X))$.

\noindent Positive part (submodular):

$\fg_{\circ}(U)=\inf_\XX\Big\{\sum_{X\in\XX} |\fg(X)|_+:~ \XX\text{~partition
of $U$}~\Big\}$.

\noindent Part continuous from below:

$\fg_{\lc}(U) = \inf_\XX \Big\{\lim_{n\to\infty} \fg(X_n):~\XX= (X_1\subseteq
X_2\subseteq\ldots),~\cup_n X_n=U\Big\}$.

\noindent Majorizing measure:

$\fgx(U) = \sup_\XX \Big\{\sum_{X\in\XX} \fg(X):~\XX\text{~partition of
$U$}~\Big\}$.
\end{document}